\documentclass{article}

 \PassOptionsToPackage{numbers, compress}{natbib}

\usepackage[final]{neurips_2022}
\usepackage{microtype}
\usepackage{graphicx}
\usepackage{subfigure}
\usepackage{booktabs} 
\usepackage{multirow}
\usepackage{caption}
\usepackage{xcolor}

\usepackage{amsmath}
\usepackage{amssymb}
\usepackage{mathtools}
\usepackage{amsthm}
\usepackage{wrapfig}

\usepackage[utf8]{inputenc} 
\usepackage[T1]{fontenc}    
\usepackage{hyperref}       
\usepackage{url}            
\usepackage{booktabs}       
\usepackage{amsfonts}       
\usepackage{nicefrac}       
\usepackage{microtype}      
\usepackage{xcolor}         

\hypersetup{
    colorlinks,
    linkcolor={blue!70!black},
    citecolor={blue!70!black},
}

 \newcommand{\bR}{\mathbb{R}}

 \newcommand{\bE}{\mathbb{{E}}}
 \newcommand{\cA}{\mathcal{A}}

 \newcommand{\cS}{\mathcal{S}}

  \newcommand{\Prtheta}{\textup{Pr}^{\theta}}
   
    \newcommand{\NEgap}{\textup{\texttt{NE-gap}}}

\newcommand{\limtinf}{\lim_{t\rightarrow+\infty}}
 \newcommand{\limsuptinf}{\limsup_{t\rightarrow+\infty}}
 \newcommand{\liminftinf}{\liminf_{t\rightarrow+\infty}}
 \newcommand{\diag}{\textup{diag}}
 \DeclareMathOperator*{\argmax}{argmax}

 \theoremstyle{plain}
\newtheorem{theorem}{Theorem}

\newtheorem{lemma}[theorem]{Lemma}
\newtheorem{corollary}[theorem]{Corollary}
\newtheorem{definition}{Definition}
\newtheorem{assumption}{Assumption}
\theoremstyle{remark}
\newtheorem{remark}{Remark}

\allowdisplaybreaks
\usepackage{setspace}
\title{On the Global Convergence Rates of Decentralized Softmax Gradient Play in Markov Potential Games}

\author{%
  Runyu Zhang\\
  Harvard University\\
  \texttt{\small runyuzhang@fas.harvard.edu}\\
   \And
  Jincheng Mei \\
  Google Research, Brain Team\\
  \texttt{\small jcmei@google.com}
  \AND
  Bo Dai \\
  Google Research, Brain Team\\
  \texttt{\small bodai@google.com}
  \And
  Dale Schuurmans \\
  University of Alberta\\
  Google Research, Brain Team\\
  \texttt{schuurmans@google.com}
  \And
  Na Li \\
  Harvard University \\
  \texttt{nali@seas.harvard.edu}
}

\begin{document}

\maketitle

\begin{abstract}
  Softmax policy gradient is a popular algorithm for policy optimization in single-agent reinforcement learning, particularly since projection is not needed for each gradient update. However, in multi-agent systems, the lack of central coordination introduces significant additional difficulties in the convergence analysis. Even for a stochastic game with identical interest, there can be multiple Nash Equilibria (NEs), which disables  proof techniques that rely on the existence of a unique global optimum. Moreover, the softmax parameterization introduces non-NE policies with zero gradient, making it difficult for gradient-based algorithms in seeking NEs. In this paper, we study the finite time convergence of decentralized softmax gradient play in a special form of game, Markov Potential Games (MPGs), which includes the identical interest game as a special case. We investigate both gradient play and natural gradient play, with and without $\log$-barrier regularization. The established convergence rates for the unregularized cases contain a trajectory dependent constant that can be \emph{arbitrarily large}, whereas  the $\log$-barrier regularization overcomes this drawback, with the cost of slightly worse dependence on other factors such as the action set size. An empirical study on an identical interest matrix game confirms the theoretical findings.
\end{abstract}

\section{Introduction}\label{sec:intro}
Multi-agent systems encounter vast application in real world scenarios, such as network routing \citep{tao2001,claes2011}, social and economic decision making \citep{ventre2013,roscia2013}, and robotic swarms \citep{liu2018,inigo2012}. In these problems, a system  consists of a group of agents interacting in a shared environment. Given the recent success of reinforcement learning (RL), increasing attention has  been drawn to the possibility of applying RL algorithms, such as policy gradient, to multi-agent systems. However, the theoretical foundations for multi-agent reinforcement learning (MARL) remain limited. Unlike single-agent RL, the actions of other agents affect the dynamics and the decision making outcome for each individual in the system, raising additional theoretical challenges when analyzing joint performance. 

The stochastic game (SG) is a classical multi-agent model that has received extensive attention in recent MARL studies. In a stochastic game, the environment is represented by a state space that evolves based on the joint actions of agents. Each agent in a stochastic game tries to maximize its own total reward by making decisions \textit{independently}, based on  state information shared between agents. The stochastic game model was first introduced in \citep{shapley53}, with a series of followup works proposing NE-seeking algorithms, particularly in the RL setting (e.g. \citep{Littman94,  Bowling00, Shoham03,Bucsoniu10,Lanctot17, Zhang19} and citations therein). Given  recent progress in the underlying theory of RL,  many recent works have investigated finite time iteration and sample complexity for learning NE or other general equilibria notions, such as correlated  and coarse correlated equilibria (e.g. \citep{song2021}). 

There are different types of SGs, some with  attributes that merit special attention; for example, two-player zero sum games \citep{bai2020,daskalakis2021}, which are widely used to model two player competitive games such as GO. In this paper, we will focus on another type of SG, the Markov potential game (MPG)~\citep{Macua18,mguni2021learning,zhang21,leonardos21}, which includes the identical interest game as a special case. The  structure of a MPG enables efficient learning through the use of gradient-based algorithms such as gradient play. Recent work \citep{zhang21,leonardos21} has  focused on the iteration and sample complexity of finding a NE in an MPG under the \emph{direct} policy parameterization, which is not practical in most real world scenarios, given the cost of projecting back to the probability simplex on every iteration. This drawback has motivated consideration of the \emph{softmax} parameterization, which bypasses the projection step in the gradient update, and is perhaps the most popular approach to parameterizing policies in practice. \citep{Fox21} have studied natural gradient play for MPG 
under softmax parameterization, but only address asymptotic behavior and leave finite time complexity open. 


From the perspective of analysis and practical performance, the extension from the direct to the softmax parameterization in policies is nontrivial. Even in the single agent case, as shown by \citep{agarwal2020,Mei20}, there are policies in the softmax parameterization that have near-zero gradient   and yet are far from being globally optimal, which creates difficulty for a gradient-based algorithm  to escape suboptimal points. A similar issue  exists for MPGs: due to the more complex interaction between agents, there is even a greater set of policies that obtain small gradient norm but are far from a NE. Based on our analysis and numerical results, even for natural gradient play---which is known to enjoy dimension free convergence in single agent learning \citep{agarwal2020}---we find  in the multiagent setting that it can still become stuck in these undesirable regions. 
Such evidence suggests that preconditioning according to the Fisher information matrix \cite{rao1992information,amari2012differential} is not sufficient to ensure fast convergence in multi-agent learning. A stronger form of regularization is required, which motivates the introduction of $\log$-barrier regularization to avoid undesirable regions of policy space.


\begin{table*}[htbp]
    \centering
    \begin{tabular}{|c|c|c|}
    \hline
      Algorithm   &  Single-agent MDP & Multi-agent MPG\\
    \hline
       {Gradient play,} & $O\left(\frac{|\cA|M^2}{(1-\gamma)^4\epsilon^2}\right)$ & $O\left(\frac{(\phi_{\max} - \phi_{\min})\sum_{i=1}^n|\cA_i|M^2}{(1-\gamma)^4\epsilon^2}\right)$\\
         direct parameterization&\citep{agarwal2020} &\!\!\!\citep{zhang21,leonardos21}\!\!\!\\
    \hline
       {Gradient play, } &$O\left(\frac{M^2 }{(1-\gamma)^3c^2 \epsilon}\right)$&\multirow{2}{*}{$\boldsymbol{O\left(\frac{n\max_i|\cA_i|(\phi_{\max} - \phi_{\min})M^2 }{(1-\gamma)^4c^2 \epsilon^2}\right)^*}$}\\
        softmax parameterization& \citep{Mei20} & \\
    \hline
        {Natural gradient play, } &$O\left(\frac{1}{(1-\gamma)^2\epsilon}\right)$&\multirow{2}{*}{$\boldsymbol{ O\left(\frac{n(\phi_{\max}-\phi_{\min})^2M}{(1-\gamma)^3c\epsilon^2}\right)}^*$}\\
       softmax parameterization & \citep{agarwal2020} & \\
    \hline
        {Gradient play + $\log$-barrier reg.},& {$O\left(\frac{|\cA|^2M^2}{(1-\gamma)^4\epsilon^2}\right)$}&\multirow{2}{*}{$\boldsymbol{O\left(\frac{n\max|\cA_i|^2 (\phi_{\max} - \phi_{\min})M^2}{(1-\gamma)^4\epsilon^2}\right)}$}\\ 
        {softmax parameterization} & \citep{agarwal2020}& \\
    \hline
        {\!\!\!\small Natural gradient play + $\log$-barrier reg.\!\!},& \multirow{2}{*}{Unknown}&\multirow{2}{*}{$\boldsymbol{O\left(\frac{n\max_i|\cA_i|(\phi_{\max}-\phi_{\min})M^2}{(1-\gamma)^4\epsilon^2}\right)}$}\\ {softmax parameterization} && \\
    \hline
    \end{tabular}
    \caption{Summary of known convergence rate results for gradient based methods in Markov decision processes (MDPs) and MPGs respectively. The new results proved in this paper for MPGs are displayed in bold font. Complexity bounds with `*' depend on an additional assumption on the MPG (See Theorem~\ref{thm:asymptotic-convergence} and \ref{thm:unregularized-PG}). The definitions of variables $M$ and $c$ appearing in some bounds can be found in \eqref{eq:def-M} and \eqref{eq:def-c}. \emph{Note that the definition of $M$ is slightly different from the ``distribution mismatch coefficient'' $D_\infty$ defined in \citep{agarwal2020} (see more details in descriptions that follows Assumption \ref{assump:like-ergodicity}). To make the complexity results more comparable, we slightly modify and re-derive the results in \citep{agarwal2020,Mei20,zhang21,leonardos21}.}}
    \label{tab:summary}
\end{table*}


\textbf{Our contribution:}
    In this paper, we provide \emph{finite time} iteration complexity results for gradient play and natural gradient play under the softmax parameterization, considering both unregularized and $\log$-barrier regularized dynamics.  We summarize the convergence rates and compare them to existing results for the direct parameterization and to the corresponding single agent cases in Table \ref{tab:summary}. These findings  suggest that regularization is crucial for obtaining fast convergence to a NE under the softmax parameterization in a MPG.  In Table \ref{tab:summary}, the results for the two unregularized algorithms in the multi-agent case rely on the assumption that the set of stationary policies is isolated (which is also assumed in ~\cite{Fox21} when establishing the asymptotical convergence for natural policy gradient), and the corresponding complexity bounds contain an initialization dependent factor $c$.  By contrast, the $\log$-barrier regularized algorithms overcome both drawbacks, but as a tradeoff, their bounds incur a slightly worse dependence  on $|\cA_i|$ and $M$. We observe numerically that the $\log$-barrier regularized algorithms are indeed more robust against becoming trapped near undesirable non-NE stationary points. To the best of our knowledge, the finite-time iteration complexity results are the first such results for MPGs under the softmax parameterization. Though the analysis for the gradient play follows their single-agent counterparts \cite{agarwal2020,Mei20}, the results for natural gradient play are highly non-trivial, requiring very different analysis tools which have their own merits to the literature (see Remark~\ref{rem:proof-unregularized} and Remark~\ref{rem:proof-regularized} for more details on the technical novelty in the analysis). Our results also convey the following two messages. First, finding the NE of a multi-agent MPG is harder than finding the global optimum for the single-agent case, because multi-agent learning suffers greater risk of becoming trapped near undesirable stationary points. This is reflected in the dependence of the complexity bounds on $\epsilon$  in Table~\ref{tab:summary}.
Second, natural gradient play outperforms gradient play counterparts, suggesting that natural gradient play captures useful information about the geometry of the parameter space that accelerates the learning process.
\section{Problem settings}
We consider an infinite time horizon $n$-agent stochastic game (SG, \citep{shapley53}) $\mathcal{M} \!=\! (N, \cS, \cA \!=\! \cA_1\!\times\!\dots\!\times \!\cA_n, P, r \!=\! (r_1, \!\dots, \!r_n), ~\gamma, \rho)$ which is specified by an agent set $N\!=\! \left\{1,2,\dots,n\right\}$, a finite state space $\cS$, a finite action space $\mathcal{A}_i$ for each agent $i\in N$,  a transition model $P$ (such that $P(s'|s,a) = P(s'|s,a_1, \dots, a_n)$ is the probability of transitioning into state $s'$ upon taking action $a:=(a_1,\ldots,a_n)$ in state $s$ where $a_i\in\cA_i$ is action of agent $i$), a reward function $r_i: \cS\times\cA \rightarrow [0,1]$ for each agent $i$, a discount factor $\gamma \in [0,1)$, and an initial state distribution $\rho$ over $\cS$. We use $s(t)\in \cS$ to denote the state at time step $t$, and $a(t) = (a_{1}(t),\dots,a_{n}(t))\in\cA$ to denote the total action.

A stochastic policy $\pi: \cS\rightarrow\Delta(\cA)$ (where $\Delta(\cA)$ is the probability simplex over $\cA$) specifies a strategy, where agents choose their actions \textit{jointly} based on the current state in a stochastic fashion; i.e. $\Pr(a(t)|s(t)) = \pi(a(t)|s(t))$. 
A \emph{decentralized} stochastic policy is a special subclass of stochastic policies with $\pi = \pi_1\times\ldots\times\pi_n$, such that $\pi_i: \cS \rightarrow \Delta(\cA_i)$, where $\pi_i$ is agent $i$'s own local policy. For decentralized stochastic policies, each agent takes its action based on the current state $s$ \textit{independently of} other agents' action choices; i.e.,
\begin{equation*}
\textstyle
 \Pr(a(t)|s(t)) \!=\! \pi(a(t)|s(t)) \!=\! \prod_{i=1}^n \pi_i(a_i(t)|s(t)). 
\end{equation*}
For notation simplicity, we define~ $\!\pi_{I}(a_I|s)\!\!:=\!\! \prod_{i\in I}\! \pi_i(a_i|s)\!,$ where $I\subseteq N$ is an index set. Further, we use the notation $-i$ to denote the index set $N\backslash \{i\}$. In this paper we focus on tabular softmax parameterization for a policy, where  policy $\pi_{\theta} = (\pi_{\theta_1},\dots,\pi_{\theta_n})$ is parameterized by a set of parameters $\theta = (\theta_1,\dots,\theta_n)$, with $\theta_i = \{\theta_{s,a_i}\}_{s\in\cS,a_i\in\cA_i}$, and where
\begin{equation}
    \pi_{\theta_i}(a_i|s) = \frac{\exp{(\theta_{s, a_i})}}{\sum_{a_i'}\exp{(\theta_{s,a_i'})}}.
\end{equation}
We denote agent $i$'s total reward starting from initial states $s(0)\!\sim\! \rho$ as:
    $J_i(\theta) \!:=\!\!\! \bE_{s(0)\!\sim\!\rho} \left[\sum_{t=0}^\infty \gamma^t r_i(s(t),a(t))\big|~\pi_\theta,s(0)=s\right].$
Agent $i$'s objective is to maximize its own total reward $J_i$.  A Nash equilibrium (NE) is often used to characterize the equilibrium (a joint policy) where no agent has a unilateral incentive to deviate from it.

\begin{definition} \label{def:NE}
{(Nash equilibrium)}
A policy $\theta^* = (\theta_1^*, \dots, \theta_n^*)$ is called a (Markov perfect) Nash equilibrium (NE) if 
\begin{equation}\label{eq:NE-def}
\textstyle
    J_i(\theta_i^*, \theta_{-i}^*) \ge J_i(\theta_i', \theta_{-i}^*), \quad \forall \theta_i',\quad i\in N
\end{equation}
Further, we define the `NE-gap' of a policy $\theta$ to be:
\begin{align*}
\textstyle
    \NEgap_i(\theta):= \sup_{\theta_i'} J_i(\theta_i', \theta_{-i}) - J_i(\theta_i, \theta_{-i}); \quad 
    \NEgap(\theta):= \max_i \NEgap_i(\theta).
\end{align*}
A policy $\theta$ is an $\epsilon$-Nash equilibrium if:
   ~~$ \NEgap(\theta) \le \epsilon.$
\end{definition}

We define the value function with respect to stage cost $r_i$ as: 
$$
\textstyle
    V_i^\theta(s) := \bE \left[\sum_{t=0}^\infty \gamma^t r_i(s(t),a(t))\big|~\pi_\theta,s(0)=s\right].$$
 We define agent $i$'s $Q$-function and advantage function $Q_i^\theta, A_i^\theta: \cS\times\cA \rightarrow \bR$, 
\begin{align*}
\textstyle
    Q_i^\theta(s,a) \!:=\! \bE\! \left[\sum_{t=0}^\infty \gamma^t r_i(s(t),a(t))\big|~\pi_\theta,s(0)\!=\!s, a(0)\!=a\right]\!,~~~~
    A_i^\theta(s,a)\!:=\! Q_i^\theta(s,a) - V_i^\theta(s).
\end{align*}
We further define agent $i$'s \textit{`averaged' Q-function} $\overline {Q_i^{\theta}}: \cS\times\cA_i \rightarrow \bR$ and \textit{`averaged' advantage-function} $\overline {A_i^{\theta}}: \cS\times\cA_i \rightarrow \bR$ as:
\begin{align*}
\textstyle
\overline {Q_i^{\theta}}(s,a_i):=\sum_{a_{-i}}\pi_{\theta_{-i}}(a_{-i}|s)Q_i^\theta(s,a_i,a_{-i}),~~~~
\overline {A_i^{\theta}}(s,a_i):=\sum_{a_{-i}}\pi_{\theta_{-i}}(a_{-i}|s)A_i^\theta(s,a_i,a_{-i}).
\end{align*}
Finally, define the \textit{discounted state visitation distribution} $d_\theta$ of a policy $\pi_\theta$ given an initial state distribution $\rho$ as:
\begin{equation}\label{eq:discounted state visitation distribution}
\textstyle
    d_\theta(s) := \bE_{s(0)\sim\rho} (1-\gamma) \sum_{t=0}^\infty\gamma^t \Prtheta(s(t)=s|s(0)),
\end{equation}
where $\Prtheta(s(t)=s|s(0))$ is the state visitation probability that $s(t)=s$ when executing $\pi_\theta$ starting at state $s(0)$. 
From the policy gradient theorem~\citep{Sutton1999}, we have that (proof given in Appendix \ref{apdx:gradient}): 
\begin{equation}\label{eq:gradient-formula}
\begin{split}
    \frac{\partial J_i(\theta)}{\partial\theta_{s,a_i}}&=\frac{1}{1-\gamma}d_\theta(s)\pi_{\theta_i}(a_i|s)\overline{A_i^\theta}(s,a_i).
\end{split}
\end{equation}
For the remainder of the paper, we make the following assumptions on the stochastic games we study.
\begin{assumption}\label{assump:like-ergodicity}
The stochastic game $\mathcal{M}$ satisfies:~
$  \inf_\theta\min_{s\in\cS}  d_{\theta}(s) > 0$. 
\end{assumption}
Assumption \ref{assump:like-ergodicity} requires that every state is visited with positive probability for any policy, which is a standard assumption for convergence proofs in the RL literature (e.g. \citep{agarwal2020, Mei20}). We will use $M$ to denote the following quantity 
\begin{equation}\label{eq:def-M}
\textstyle
    M := \sup_{\theta}\max_s\frac{1}{d_\theta(s)}.
\end{equation}
Note that $M$ can be viewed as a measure of exploration sufficiency in the stochastic game, which is slightly different from the ``distributional mismatch coefficient'' introduced in \citep{agarwal2020} defined by $\sup_{\theta,\theta'}\max_s\frac{d_{\theta'}(s)}{d_\theta(s)}$; however, both can be upper bounded by $\max_s\frac{1}{(1-\gamma) \rho(s)}$.

We primarily focus on the following subclass of stochastic games in this paper:
\begin{definition}
A stochastic game is called a Markov potential game (MPG, \citep{zazo16,Macua18,zhang21,leonardos21,mguni20}) if there exists a potential function $\phi: \cS \times \cA_1\times\cdots\times\cA_n \rightarrow \mathbb{R}$ such that for any agent $i$ and any pair of policy parameters $(\theta_i',\theta_{-i}), (\theta_i,\theta_{-i})$ :
{\small
\begin{align}
\textstyle
   & \!\bE \!\left[\sum_{t=0}^\infty \!\gamma^t r_i(s(t),a(t))\big|\pi\! =\! (\theta_i',\theta_{-i}),s(0)\!=\!s\right] \!-\!\bE\! \left[\sum_{t=0}^\infty\! \gamma^t r_i(s(t),a(t))\big|\pi \!=\! (\theta_i,\theta_{-i}),s(0)\!=\!
   s\right]\nonumber\\
   = 
    & \!\bE \!\left[\sum_{t=0}^\infty \!\gamma^t \phi(s(t),a(t))\big|\pi\!=\! (\theta_i',\theta_{-i}),s(0)\!=\!s\right] \! -\! \bE \!\left[\sum_{t=0}^\infty \!\gamma^t \phi(s(t),a(t))\big|\pi 
    \!=\! (\theta_i,\theta_{-i}),s(0)\!=\!s\right]\!,~\forall~s. \label{eq:def-mpg}
\end{align}
}
\end{definition}
 Without loss of generality, we assume that 
~$\phi_{\min} \le \phi(s,a)\le \phi_{\max}$ for all $(s,a)$. The definition of MPG is a generalization of the notion potential game in the one-shot setting \citep{Monderer96}. Note that identical reward game where agents share a same reward function naturally satisfies the above condition and serves as one important special case of MPG. For non-identical reward settings, \citep{Macua18,Gonzalez13} found that continuous MPGs can model applications such as the great fish war \citep{Levhari80}, the stochastic lake game \citep{Dechert06}, medium access control \citep{Macua18} etc. For tablular MPGs, \citep{zhang21, leonardos21} also discuss necessary/sufficient conditions that implies a MPG, as well as its application and counterexamples.

Given a MPG, we define the \textit{total potential function} $\Phi$ as:
\begin{equation*}
    \textstyle
    \Phi(\theta) := \bE_{s(0)\sim\rho} \left[\sum_{t=0}^\infty \gamma^t \phi(s(t),a(t))\big|~\pi_\theta,s(0)=s\right].
\end{equation*}
Given the property in (\ref{eq:def-mpg}), it is straightforward to verify that the NE condition \eqref{eq:NE-def} is equivalent to $
    \Phi(\theta_i^*, \theta_{-i}^*) \ge \Phi(\theta_i', \theta_{-i}^*), \forall \theta_i', i\in N
$ and that for the policy gradient, $\frac{\partial J_i(\theta)}{\partial \theta_{s,a_i}}=\frac{\partial \Phi(\theta)}{\partial \theta_{s,a_i}}$ for all $i$, $s$, $a_i$.

\begin{remark}[\textbf{Differences between MPG and single-agent/centralized MDP}]\label{rem:mpg}
Because of the existence of the total potential function $\Phi$, it is natural to ask whether MPG renders the multi-agent policy gradient similar to single agent policy gradient and thus results and analysis tools developed for single agent policy gradient in e.g., \citep{agarwal2020, Mei20} would be easily extended to the multiagent case.  Unfortunately, this is not the case. To illustrate how it differs from single-agent/centralized case, we can focus on the special type of MPGs where every agent has the same reward function, namely the identical interest case. In the single agent/centralized case, there is a unique global optimal solution which corresponds to the convergent stationary policy. However, in the multiagent case, even if the rewards are identical, because the policy is \textit{decentralized}, i.e., agents taking \textit{independent} policies $\pi:=\pi_1\times\ldots\times\pi_n$, we loose the connection between stationary policies and optimal policies. As we shown later, the convergent stationary policies are Nash equilibria, which are unfortunately non-unique even for the identical interest case.  Moreover, a key condition that is used in establishing the convergence rate,  Łojasiewicz condition (Lemma~\ref{lemma:non-uniform-gradient-domination}), is also much weaker for the multiagent case compared to single agent \cite{Mei20}: the left hand side is the \textit{Nash gap}  $\max_{i,\theta_i^*}\Phi(\theta_i^*,\theta_{-i}) -\Phi(\theta)$ instead of the \textit{optimality gap} $\max_{\theta^*}\Phi(\theta^*) -\Phi(\theta)$). Note that zero Nash gap does not imply zero optimality gap, as there exists \textit{many} NEs of \textit{different} values. These differences disable many proof technique used for single agent case and make the analysis harder and lead to different performance results, as demonstrated in the rest of the paper. 
\end{remark}
\section{Relationship between first order stationary point and Nash equilibrium}\label{section:stationary-point-NE-relationship}

Before studying convergence performance of gradient play algorithms, it is important to first understand the relationship between the stationary points and the NEs. Unfortunately, equivalence cannot be established in this setting. Standard optimization theory guarantees that all NEs are stationary points, but unfortunately not vice versa. Under softmax parameterization, there exist non-NE stationary points. For example, from the gradient formulation \eqref{eq:gradient-formula}, it can be shown that any non-NE deterministic policies are also stationary points. However, the notion of NE and stationarity are indeed closely related. This section aims to characterize some differences between NE and non-NE stationary points. This differentiation of the NE and non-NE stationary points is established by the non-uniform Łojasiewicz condition (also known as gradient domination) for stochastic games.
\begin{lemma}(Non-uniform Łojasiewicz inequality; proof given in Appendix \ref{apdx:non-uniform-gradient-domination})\label{lemma:non-uniform-gradient-domination}
Define
\begin{equation}\label{eq:c-theta}
\textstyle
    M(\theta):=\max_s \frac{1}{d_\theta(s)}, \quad c(\theta):= \min_s \sum_{a_i^* \in \argmax_{a_i} \overline{Q_i^{\theta}}(s,a_i)}\pi_{\theta_i}(a_i^*|s).
\end{equation}
Then we have that
\begin{align*}
&\NEgap_i(\theta)\le \frac{\sqrt{|\cA_i|}M(\theta)}{c(\theta)} \|\nabla_{\theta_i}\!J_i(\theta)\|_2
.
\end{align*}
\end{lemma}
The Łojasiewicz condition (gradient domination) implies that the NE-gap of a policy can be bounded by the norm of its gradient, whereas the term `non-uniform' refers to the factor $\frac{\sqrt{|\cA_i|}M(\theta)}{c(\theta)}$, which cannot be bounded uniformly for all $\theta$. The counterpart of Lemma \ref{lemma:non-uniform-gradient-domination} for a single-agent MDP was first introduced in \citep[Lemma 8]{Mei20}. One major difference between Lemma \ref{lemma:non-uniform-gradient-domination} and \citep[Lemma 8]{Mei20} is how $c(\theta)$ is defined. In \citep{Mei20}, $c(\theta):= \min_s\pi_{\theta}(a^*(s)|s)$, where $a^*(s)$ is the optimal action on state $s$ (i.e., $a^* = \argmax_a Q^*(s,a)$), whereas in MPG, because there's no globally defined $Q^*$, the $a_i^*$ in \eqref{eq:c-theta} is chosen as the \emph{greedy} optimal action of the \emph{current} averaged $Q$-function (i.e., $a_i^* \in \argmax_{a_i} \overline{Q_i^{\theta}}(s,a_i)$).

Note that because $c(\theta)$ on the denominator can be zero for certain policies (e.g. one can verify that any non-NE deterministic policy have $c(\theta)=0$), 
which implies that a $\theta$ with gradient norm close to zero is not necessarily near a NE. Given this observation, we could differentiate the non-NE stationary points with NEs by whether $c(\theta^*)$ equals to zero, which is formally stated in the following lemma:

\begin{lemma}\label{lemma:c-NE-equal-1} (Proof given in Appendix \ref{apdx:non-uniform-gradient-domination})
Suppose $\theta^*$ is a stationary point, i.e. $\|\nabla\Phi(\theta^*)\| = 0$, then $\theta^*$ is a NE if and only if $c(\theta^*) = 1$, $\theta^*$ is not a NE if and only if $c(\theta^*) = 0$.
\end{lemma}


\section{Unregularized gradient play}\label{sec:unregularized}
We first investigate the convergence to NE for gradient and natural gradient play, respectively. Under the softmax parameterization, the two schemes are given by
\begin{align}
     \textit{{Gradient Play:  }}\quad~& \theta_i^{(t+1)}= \theta_i^{(t)}\!+\!\eta\nabla_{\theta_i}J_i(\theta_i^{(t)})\label{eq:unregularized-PG}\\
     \textit{{Natural Gradient Play:}}~~~& \theta_i^{(t+1)} = \theta_i^{(t)}+\!\eta F_i(\theta^{(t)})^\dagger\nabla_{\!\theta_i}J_i(\theta_i^{(t)})\label{eq:unregularized-NPG}
\end{align}
where $\dagger$ denotes the Moore-Penrose inverse and $F_i\!(\theta)$ is the Fisher information matrix for $\pi_{\theta_i}$:
$$\!\!F_i(\theta):=\bE_{s\!\sim\! d_\theta(\cdot)}\bE_{a_i\!\sim\!\pi_{\theta_i}(\cdot|s)}\!\left[\nabla_{\theta_i} \log\pi_{\theta_i}(a_i|s)\nabla_{\theta_i} \log\pi_\theta(a_i|s)^{\!\top}\right].$$
For notational simplicity,  we  abbreviate the variables $d_{\theta^{(t)}}$, $A_i^{\theta^{(t)}}$ and $\overline {A_i^{\theta^{(t)}}}$ as $d^{(t)}$, $A_i^{(t)}$ and $\overline {A_i^{(t)}}$ respectively; and denote $\pi_{\theta^{(t)}}(a|s)$ and $\pi_{\theta_i^{(t)}}(a_i|s)$ as $ \pi^{(t)}(a|s)$ and $\pi_i^{(t)}(a_i|s)$ respectively.
For the softmax parameterization, we can establish the equivalence of natural gradient play and soft Q-learning \citep{haarnoja2017}, formally stated in the following lemma.
\begin{lemma} (Proof given in Appendix \ref{apdx:NPG-derivation})
Natural gradient play is equivalent to
\begin{equation}\label{eq:unregularized-NPG-2}
\textstyle
    \pi_i^{(t+1)}(a_i|s)\propto \pi_i^{(t)}(a_i|s)\exp\left({\eta \overline{A_i^{(t)}}(s,a_i)}{\big/}({1-\gamma})\right)
\end{equation}
\end{lemma}
\paragraph{Asymptotic convergence to Nash Equilibrium.}As stated in Section \ref{section:stationary-point-NE-relationship}, there exist stationary points that are not NEs. It is not immediately obvious why running gradient methods can avoid converging to these points, thus before studying convergence rate to NE, it is necessary to first examine whether asymptotic convergence holds. Moreover, the asymptotic convergence result is used to establish the finite time convergence rate results later (see the subsection~\ref{subsec:finite-time}).

\begin{theorem}\label{thm:asymptotic-convergence}
(Proof given in Appendix \ref{apex:asymptotic-convergence}) Suppose Assumption \ref{assump:like-ergodicity} holds and that the stationary policies are isolated,   gradient play \eqref{eq:unregularized-PG} with $\eta \le \frac{(1-\gamma)^3}{6n}$  guarantees that
$    \limtinf \theta^{(t)} = \theta^{(\infty)}$, where $\theta^{(\infty)}$ is a NE. The same argument also holds for natural gradient play \eqref{eq:unregularized-NPG-2} 
with $\eta \le \frac{(1-\gamma)^2}{2n(\phi_{\max}-\phi_{\min})}$.
\end{theorem}
 The proof of Theorem \ref{thm:asymptotic-convergence} resembles the technique used in \citep{agarwal2020} for the single agent case, where the additional assumption on the isolated stationary policies is introduced due to some specific technical difficulties encountered in multi-agent learning (see more discussion in Appendix \ref{apex:asymptotic-convergence}, which is also introduced in \citep{Fox21} for establishing the asymptotic convergence of NPG. We believe it is a conservative condition for ensuring the asymptotic convergence. It remains an interesting open question to establish convergence without this assumption.
\subsection{Finite time convergence rate}\label{subsec:finite-time}
This section considers finite time convergence rate for gradient play and natural gradient play. Corresponding results for the single-agent setting can be found in \citep{Mei20} (for gradient play) and \citep{agarwal2020,Khodadadian2021,Mei2021} (for natural gradient play). Some aspects of these analyses can be carried over to the multi-agent MPG setting; however, as will be discussed later, there are several fundamental differences that make the multi-agent case more challenging. 

Our convergence results rely on the observation from Section \ref{section:stationary-point-NE-relationship} and the asymptotic convergence to NE. Combining Theorem \ref{thm:asymptotic-convergence} and Lemma \ref{lemma:c-NE-equal-1}, we know that $c(\theta^{(t)})$ asymptotically converges to 1 for (natural) gradient play, and since $c(\theta^{(t)}) > 0$ for any softmax policy (because $\pi_{\theta_i}(a_i|s) >0$), 
\begin{equation}\label{eq:def-c}
\textstyle
    c:=\inf_t c(\theta^{(t)}) > 0.
\end{equation}
We are now ready to give formal convergence rates for gradient  and natural gradient play respectively.
\begin{theorem}\label{thm:unregularized-PG} (Gradient play and natural gradient play; proof given in \ref{apdx:unregularized-PG-NPG})
Suppose Assumption \ref{assump:like-ergodicity} holds and that the stationary policies are isolated,  gradient play \eqref{eq:unregularized-PG} with $\eta = \frac{(1-\gamma)^3}{6n}$  will guarantee that for all $T$,
\begin{equation}
    \frac{\sum_{t=0}^{T-1}\NEgap(\theta^{(t)})^2}{T} \lesssim O\left( \frac{n\max_i|\cA_i|(\phi_{\max} - \phi_{\min})M^2 }{(1-\gamma)^4c^2 T}\right),
\end{equation}
Natural gradient play \eqref{eq:unregularized-NPG-2} with $\eta = \frac{(1-\gamma)^2}{2n(\phi_{\max}-\phi_{\min})}$ will guarantee that for all $T$,
\begin{equation}
     \frac{\sum_{t=0}^{T-1}\NEgap(\theta^{(t)})^2}{T}\lesssim O\left( \frac{n(\phi_{\max}-\phi_{\min})^2M}{(1-\gamma)^3c T}\right).
\end{equation}
Here $O(\cdot)$ hides constant factors, $M$ and $c$ are defined as in \eqref{eq:def-M} and \eqref{eq:def-c}, respectively.
\end{theorem}
\begin{remark}[\textbf{Proof sketch and novelty}]\label{rem:proof-unregularized} The proof for gradient play is relatively straightforward from the non-uniform Łojasiewicz inequality and standard non-convex optimization results, which we refer readers to the appendix for more details. However, the proof for natural gradient play is more involved and existing analysis on NPG cannot be generalized to this setting. For single-agent MDP, the analysis on NPG leverages the unique existence of optimal value function $V^*$ so that similar analysis for mirror-descent can also carry over to NPG analysis, and thus obtain dimension free convergence. However, in the multi-agent setting, there's no well-defined $V^*$ as NEs can be non-unique with different potential values, thus, we need to further deploy additional structures of the total potential function $\Phi$. Our analysis rely on the sufficient ascent lemma (Lemma \ref{lemma:unregularized-NPG-1}) that lower bounds the ascent amount $\Phi(\theta^{(t+1)}) - \Phi(\theta^{(t)})$ for each natural gradient step (we would like to further note that this sufficient ascent lemma cannot be trivially obtained by the smoothness of $\Phi$). Then, we further lower bound the ascent amount in terms of $\NEgap$ (Lemma \ref{lemma:unregularized-NPG-2}). Lastly, the theorem follows by conducting standard telescoping techniques.
\end{remark}
\noindent\textbf{Discussion on $\frac{1}{c}$:} The complexity results in Theorem \ref{thm:unregularized-PG} both depend on $\frac{1}{c}$. However, this term can become arbitrarily large. In fact, \citep{Li2021softmax} show that $c$ can be exponentially small in terms of the number of states $|\cS|$ for a general finite MDP, even under uniform initialization, hence convergence can be very slow. This conclusion is also confirmed by numerical evidence. As pointed out by \citep{Mei20}, even for single agent settings, policy gradient can get stuck at regions with small gradient yet far from being global optimal. Similar or even worse phenomena can be observed for multi-agent MPG, as shown in Figure \ref{fig:numerics-main-text}(a)-(c): even for a single state game ($|\cS| = 1$) with uniform initialization, unregularized gradient based algorithms can still enter regions with a relatively large $\NEgap$ while the gradient norm and $c(\theta)$ are close to zero.

\noindent \textbf{More comparison with learning for single-agent MDP:} 
For gradient play, we have established an iteration complexity of $O\left(\frac{n\max_i|\cA_i|(\phi_{\max} - \phi_{\min})M^2}{(1-\gamma)^4c^2\epsilon^2}\right)$ to find an $\epsilon$-NE, whereas \citep{Mei20} show a complexity of $O\left(\frac{(\phi_{\max} - \phi_{\min})M^2}{(1-\gamma)^4c^2\epsilon}\right)$ to reach an $\epsilon$-global optimum for policy gradient in a single agent MDP. The dependence on  $\frac{1}{\epsilon}$ is better in the single agent case because of the existence of a global optimal policy $\pi^*$ and optimal total reward $V^*$, which justify the definition of optimality gap $\delta_t = V(\theta^{(t)}) - V^*$. This, combined with the non-uniform Łojasiewicz condition which bounds $\delta_t$ by the gradient norm, allows one to use  techniques  from convex smooth analysis to show that $\delta_t$ is on the scale of $\frac{1}{t}$. By contrast, for multi-agent learning, there can be multiple NEs with different values, hence  $\delta_t$ is ill-defined. Further, note that the $\NEgap$ is different from the optimality gap, hence gradient ascent no longer guarantees monotonic decreasing of $\NEgap$ (Figure \ref{fig:numerics-main-text}(a)), and we can only exploit non-convex optimization techniques that yield $O(\frac{1}{\epsilon^2})$ complexities.

For the same reason, the rate of convergence we obtain for natural gradient play is $O\left(\frac{n(\phi_{\max}-\phi_{\min})^2M}{(1-\gamma)^3c\epsilon^2}\right)$, which is worse than the dimension free convergence rate of $O\left(\frac{1}{(1-\gamma)^2\epsilon}\right)$ given in \citep{agarwal2020} for single-agent MDPs.  (A better exponential convergence rate for natural PG has also been proved in \citep{Khodadadian2021, Mei2021}  with the exponential factor being problem dependent.) 
Nevertheless, the dependence on $\frac{1}{c}$, $\frac{1}{1-\gamma}$ and $M$ is better than gradient play, suggesting that the preconditioning of natural gradient play at least partially captures the geometry of the parameter space. We  also note that the quadratic dependence on $(\phi_{\max} - \phi_{\min})$ might be a proof artifact. It remains an open question whether this can be reduced to a linear dependence. 


\section{Gradient play with $\log$-barrier regularization}
The previous section has shown that, for unregularized objectives, the convergence rate for gradient based algorithms  depends on a factor $\frac{1}{c} $ that can be arbitrarily large for bad initializations. This motivates us to investigate  regularization,  in hopes of removing the dependence on $\frac{1}{c}$. For this purpose,  we consider  $\log$-barrier regularization:
\begin{equation*}
    \widetilde{J}_i(\theta) = J_i(\theta) + \lambda \sum_{s,a_i}\log\pi_{\theta_i}(a_i|s).
\end{equation*}
Define:
\begin{equation}\label{eq:widetilde-Phi}
    \widetilde{\Phi}(\theta) = \Phi(\theta) + \lambda\sum_{i=1}^n \sum_{s,a_i}\log\pi_{\theta_i}(a_i|s). 
\end{equation}
It is not hard to verify that the gradient with respect to $J_i$ is:
\begin{align*}
     \frac{\partial \widetilde{J}_i(\theta)}{\partial\theta_{s,a_i}} = \frac{\partial \widetilde{\Phi}(\theta)}{\partial\theta_{s,a_i}} = \frac{1}{1-\gamma}d_\theta(s)\pi_{\theta_i}(a_i|s)\overline{A_i^\theta}(s,a_i) + \lambda - \lambda|\cA_i|\pi_{\theta_i}(a_i|s).
\end{align*}
{\textbf{Discussion on the choice of the regularizer:}} Before analyzing the resulting algorithm we  first discuss the motivation for this regularizer. First,  note  that for each agent, the additional regularizer only depends on an agent's own local policy, which is desirable for multiagent RL. As an alternative, one might impose regularization by choosing 
\begin{equation*}
\textstyle
    \widetilde\Phi(\theta) = \Phi(\theta) + \lambda \bE_{s\sim d_\theta(\cdot)}\sum_{i=1}^n\sum_{a_i}\log\pi_{\theta_i}(a_i|s);
\end{equation*}
i.e., so that the regularization weight imposed on a state $s$ depends on the state visitation probability $d_\theta(s)$. However, in this case the gradient of the $i$-th agent $\nabla_{\theta_i}\widetilde\Phi(\theta)$ will not only depend on its own policy parameter $\theta_i$, but also on other parameters of the other agents' policies $\theta_{-i}$. Thus, running gradient based algorithms with such a regularization scheme can no longer be executed in a fully decentralized manner using local policy information. Therefore, we prefer   regularization \eqref{eq:widetilde-Phi} which does not depend on $d_\theta(s)$.
Secondly, {we adopt the $\log$-barrier instead of entropy regularization due to technical rather than practical considerations. Although entropy regularization achieves fast exponential convergence in single agent learning \citep{cen2021,Mei20}, for multi-agent learning, we haven't been able to obtain results as strong as the $\log$-barrier regularization. Intuitively, the $\log$-barrier regularized gradient field  repels the trajectory from regions with small $\pi_i(a_i|s)$ values (where the geometry becomes close to singular) more strongly, which enables us to obtain our current analysis. However, we emphasize that our result does not imply that log-barrier is better than entropy regularization in practice. It remains future work to determine whether entropy regularization, or other methods such as trust region based methods, can achieve the same, or even better convergence rates.} 

\subsection{Gradient play}
We first consider gradient play algorithm, i.e.,
\begin{equation}\label{eq:log-barrier-PG}
\textstyle
    \theta_i^{(t+1)} = \theta_i^{(t)} + \eta \nabla_{\theta_i}\widetilde{J}_i(\theta^{(t)}).
\end{equation}
{Fortunately, similar analysis from \cite{agarwal2020} for single-agent MDP can be generalized to MPG with slight modifications. Here we only state the result and defer the proof to Appendix \ref{apdx:log-barrier-pg}.}
\begin{theorem}\label{thm:log-barrier-PG}
Under Assumption \ref{assump:like-ergodicity}, for $\eta = \frac{(1-\gamma)^3}{6n + 2\lambda\max_i|\cA_i|(1-\gamma)^3}$, and  $\lambda = \frac{\epsilon}{M\max_i|\cA_i|}$, let $\theta^{(0)}$ be the uniform random policy, i.e., $\theta^{(0)} = \mathbf{0}$, then running  gradient play \eqref{eq:log-barrier-PG} for $T$ steps, where $
    T \gtrsim O\left(\frac{n\max_i|\cA_i|^2 (\phi_{\max} - \phi_{\min})M^2}{(1-\gamma)^4\epsilon^2}\right)$
will guarantee that
$
    \min_{0\le t\le T-1} \NEgap(\theta^{(t)}) \le \epsilon.
$
\end{theorem}

Note that compared to the unregularized case in~Theorem \ref{thm:unregularized-PG}, it only requires Assumption~\ref{assump:like-ergodicity}, while the convergence rate is accelerated by eliminating the dependence on $\frac{1}{c}$. However, as a (worthy) tradeoff, the dependence on the action space size $\max_i|\cA_i|$ now becomes quadratic. The key reason for these differences is that $\log$-barrier regularization assures that any policy with sufficiently small gradient norm cannot be close to the boundary of the probability simplex where the non-uniform Łojasiewicz constant is large. 
\subsection{Natural gradient play}
In the unregularized setting, we have seen that natural gradient play enjoys a better convergence rate than gradient play, which motivates us to consider whether a similar advantage still holds for the regularized case. In this section we consider natural gradient play
\begin{equation}\label{eq:log-barrier-NPG}
\textstyle
    \theta_i^{(t+1)} = \theta_i^{(t)}+\eta F_i(\theta^{(t)})^\dagger\nabla_{
    \theta_i}\widetilde{J}_i(\theta_i^{(t)}),
\end{equation}
which is equivalent to (see the proof in Appendix \ref{apdx:NPG-derivation})
\begin{equation}\label{eq:log-barrier-NPG-2}
    \pi_i^{(t+1)}(a_i|s)\propto \pi_i^{(t)}(a_i|s)\exp\left(\frac{\eta}{1-\gamma}\overline{A_i^{(t)}}(s,a_i)+\frac{\eta\lambda}{d^{(t)}(s)\pi_i^{(t)}(a_i|s)} - \frac{\eta\lambda|\cA_i|}{d^{(t)}(s)}\right).
\end{equation}
\begin{theorem}\label{thm:log-barrier-NPG} (Proof given in Appendix \ref{apdx:log-barrier-NPG})
Under Assumption \ref{assump:like-ergodicity}, for \\$\eta\! =\! \min\!\left\{\!\frac{1}{15\left(\frac{1}{(1\!-\!\gamma)^2} + \lambda|\cA_i|M \right)},\frac{1}{4\left(4\lambda\max_i|\cA_i|M^2 + \frac{4M}{(1-\gamma)^2} + \frac{3nM}{(1-\gamma)^3}\right)}\right\}$, the natural gradient play \eqref{eq:log-barrier-NPG-2} will guarantee that
$   \frac{\sum_{t=0}^{T-1}\NEgap(\theta^{(t)})}{T}\le \frac{9\left(\widetilde{\Phi}(\theta^{(T)}) - \widetilde{\Phi}(\theta^{(0)})\right)}{\eta\lambda T} + \lambda \max_i|\cA_i|M,
$
 Further, by setting $\lambda = \frac{\epsilon}{2\max_i|\cA_i|M}$, $\theta^{(0)} = \mathbf{0}$,  for $
    T\gtrsim O\left(\frac{n\max_i|\cA_i|(\phi_{\max}-\phi_{\min})M^2}{(1-\gamma)^4\epsilon^2}\right)$,
we have
$    \frac{\sum_{t=0}^{T-1}\NEgap(\theta^{(t)})}{T}\le\epsilon.
$
\end{theorem}
\begin{remark}(\textbf{Proof sketch and novelty})\label{rem:proof-regularized} As also stated for unregularized natural gradient play, there's no direct analysis tools we could borrow from literature for the analysis of natural gradient play. Our analysis depends on two key lemmas. The first is a sufficient ascent lemma on $\widetilde\Phi(\theta^{(t+1)}) - \widetilde\Phi(\theta^{(t)})$ for each natural gradient step (Lemma \ref{lemma:NPG-log-barrier-sufficient-ascent}). Another key lemma (Lemma \ref{lemma:NPG-bounded-probability}) states that the algorithm implicitly ensures that the policies never go near the boundary of the probability simplex, i.e., it can be uniformly lower-bounded by $\pi_i^{(t)}(a_i|s)\!\ge\!\frac{\lambda}{4\left(\lambda |\cA_i|M + \frac{1}{(1-\gamma)^2}\right)}, ~\forall t$. Combining the two lemmas, it can be concluded that the ascent value $\widetilde\Phi(\theta^{(t+1)}) - \widetilde\Phi(\theta^{(t)})$ can be bounded by $\NEgap(\theta^{(t)})$ plus a $\lambda \max_i|\cA_i|M$ bias term (Lemma \ref{lemma:NPG-log-barrier-Delta-to-f} and \ref{lemma:NPG-log-barrier-f-to-NEgap}), thus the proof is finished by standard telescoping technique and choosing an appropriate $\lambda$.
\end{remark}

Compared with gradient play, natural gradient play manages to reduce the time complexity by a $\max_i|\cA_i|$ factor. Further, gradient play only guarantees the minimal NE-gap smaller than $\epsilon$, while natural gradient play guarantees the average NE-gap along the trajectory smaller than $\epsilon$. To the best of our knowledge, this is the best time complexity bound for the softmax parameterization in a MPG.

\section{An Illustrative example}\label{sec:numerics}
\begin{table*}[t]
\centering
\begin{minipage}{.3\linewidth}
\centering
\begin{tabular}{|c||c|c|}
\hline
 & \!\!$a_2 \!= \!1$\!\! & \!\!$a_2 \!= \!2$\!\!\\
\hline\hline
\!\!$a_1 \!= \!1$\!\! & -1& 0.14\\
\hline
\!\!$a_1 \!=\! 2$\!\! &0.16& 0.15\\
\hline
\!\!$a_1 \!=\! 3$\!\! &0.2& -1\\
\hline
\end{tabular}
\vspace{10pt}
{Reward table}
    \label{tab:reward_table}
\end{minipage}
\begin{minipage}{.28\linewidth}
\centering
\includegraphics[width = \linewidth]{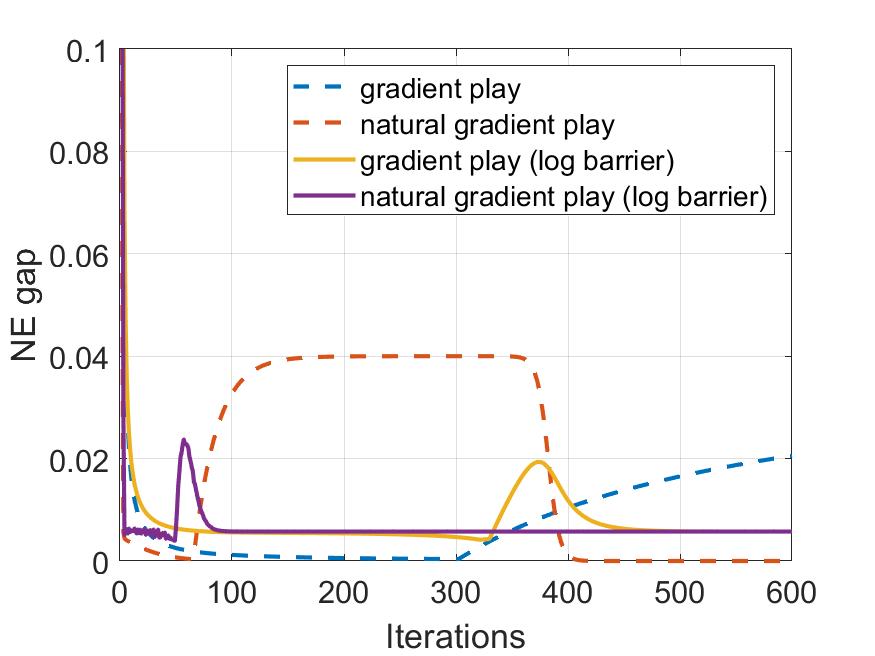}\\
\vspace{-5pt}
{\small(a)}
\end{minipage}
\begin{minipage}{.28\linewidth}
\centering
\includegraphics[width = \linewidth]{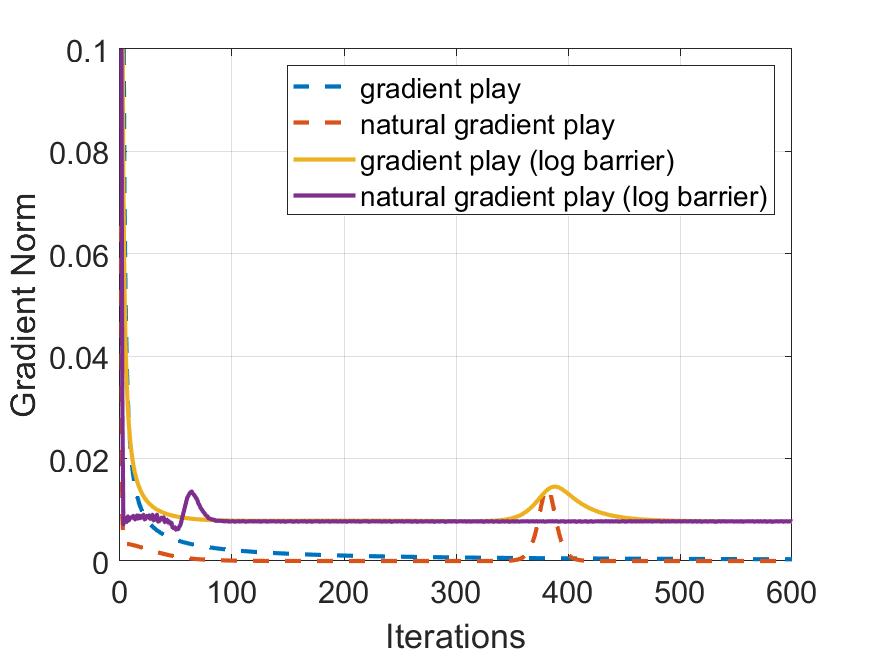}\\
\vspace{-5pt}
{\small(b)}
\end{minipage}\\
\begin{minipage}{.52\linewidth}
\centering
\includegraphics[width = \linewidth]{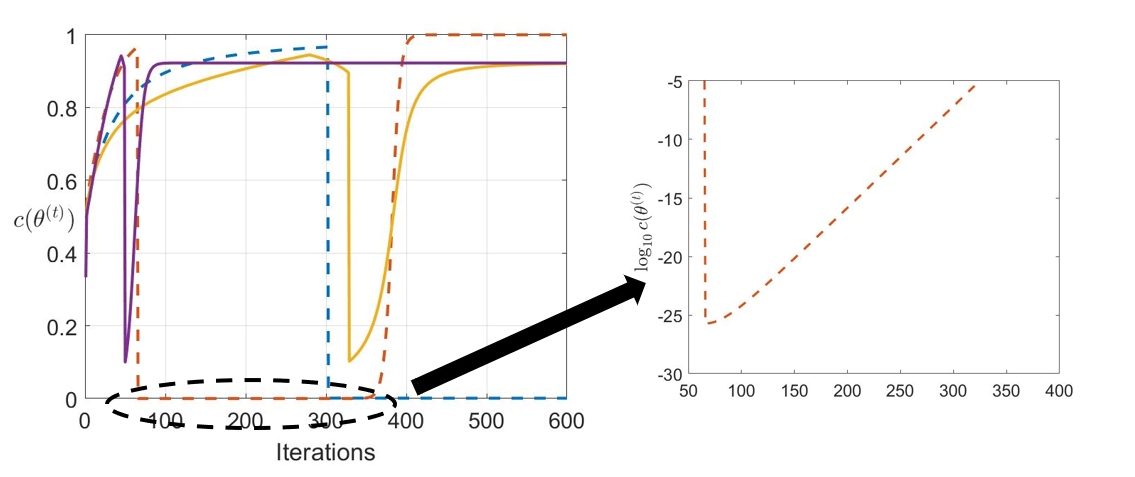}\\
\vspace{-5pt}
{\small(c)}
\end{minipage}
\hfill
\begin{minipage}{.47\linewidth}
\centering
\includegraphics[width = \linewidth]{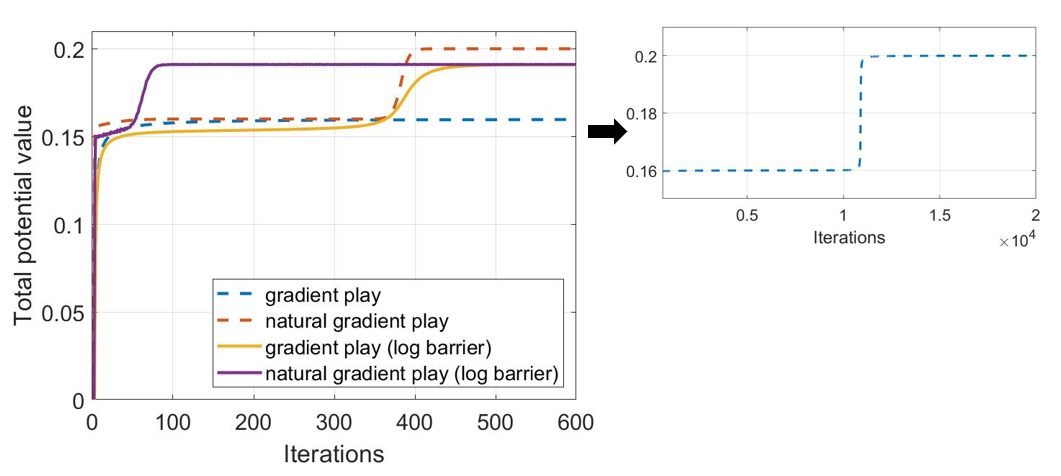}\\
\vspace{-8pt}
{\small(d)}
\end{minipage}
\captionof{figure}{We consider a two-player identical reward matrix game as shown in the reward table. We run gradient play and natural gradient play (with and without $\log$-barrier regularization) with initial policies being the uniform distribution (i.e., $\pi_1 = [\frac{1}{3},\frac{1}{3},\frac{1}{3}]$, $\pi_2 = [\frac{1}{2},\frac{1}{2}]$). The subfigures (a)-(d) show how the $\NEgap(\theta^{(t)})$, $\|\nabla_\theta\Phi(\theta^{(t)})\|_2$, $c(\theta^{(t)})$ (defined in \eqref{eq:c-theta}) and $\Phi(\theta^{(t)})$ change with each iteration respectively. In Figure (c), we zoom in on the $\log_{10}c(\theta)$ factor for natural gradient play. In Figure (d), we also zoom out the trajectory for running gradient play to iteration $2\times 10^4$. Here the step sizes were chosen to be $\eta = 5$ while the regularization weight $\lambda$ was chosen to be $\lambda = 0.003$.  In consideration of numerical stability issues, we truncate the update step of natural gradient play with $\log$-barrier regularization by a maximum absolute value of $1$ for each entry. For more numerical results and corresponding analysis see Appendix \ref{apdx:Numerics}.}
\label{fig:numerics-main-text}
\vspace{-13pt}
\end{table*}

This section aims to gain a better understanding of the four gradient play algorithms, \eqref{eq:unregularized-PG}, \eqref{eq:unregularized-NPG-2}, \eqref{eq:log-barrier-PG}, and \eqref{eq:log-barrier-NPG-2}. To better justify our theoretical results and provide additional insights, we choose a carefully designed simple two-player game so that our theoretical results can be easily revealed from the empirical observations. However the four algorithms also works for settings with more agents. \footnote{Code can be found in \url{https://github.com/DianYu420376/NeurIPS2022-softmax-MPG}} Due to space limits, we defer the simulation with more agents in Appendix \ref{apdx:Numerics}.

The reward table as well as the performance of the four algorithms are shown in Figure \ref{fig:numerics-main-text}. Comparing the $\log$-barrier regularized algorithms to the unregularized counterparts, one can see that the regularized dynamics converge faster but with a bias induced by the regularizer.
This finding  corroborates the analyses given in Theorem \ref{thm:log-barrier-PG} and \ref{thm:log-barrier-NPG}. By contrast, the unregularized dynamics are able to find a policy with zero $\NEgap$ asymptotically, but tend to get stuck in regions where $c(\theta^{(t)})$ is very close to zero, as illustrated in Fig~\ref{fig:numerics-main-text}(a)(b). Specifically unregularized natural gradient play  gets stuck around iteration 100-400 in a region where the gradient norm and $c(\theta^{(t)})$ are both close to zero while the $\NEgap$ is not. This corroborates the finding in Lemma \ref{lemma:non-uniform-gradient-domination}. Similar behavior can be observed for gradient play if we keep running the algorithm. 
In comparing the natural gradient play to gradient play algorithms, natural gradient play generally converges faster, which matches with our complexity analysis. However, natural gradient play with $\log$-barrier regularization can suffer from numerical instability  due to the $1/{\pi_i^{(t)}(a_i|s)}$ term in the exponential factor. In this case, the stepsize needs to be chosen carefully. To bypass the numerical instability, we truncate the update step of natural gradient play with $\log$-barrier regularization by a maximum absolute value of $1$ for each entry.
 \section{Discussions and conclusions}\label{sec:con}
We have established finite time iteration complexity bounds for gradient and natural gradient play under the softmax parameterization, considering both unregularized and $\log$-barrier regularized dynamics, in the Markov potential game setting. 
To our best knowledge, these are the first finite time global convergence results for softmax gradient play for MPGs. However, our work suffers from the following limitations: firstly, the paper mainly focuses on MPG settings, which limits its application to general-sum Markov games; secondly, convergence results for the unregularized case relies on an extra assumption that the stationary points are isolated; thirdly, for the regularized case, we consider $\log$-barrier regularization, which is admittedly a stronger regularization compared with entropy regularization which is more frequently used in practice. Some limitations are due to technical challenges, some might be caused by the fundamental difficulties of multi-agent learning. It remains interesting open questions to sharpen the analysis, derive similar or better bounds for other regularizations, and to develop more fundamental understandings of multi-agent learning.
\section*{Acknowledgment}
Runyu Zhang would like to thank Shicong Cen for enlightening discussions. Runyu Zhang is supported by NSF AI institute: 2112085, ONR YIP: N00014-19-1-2217, NSF CNS: 2003111 and NSF CPS: 2038603.



\bibliographystyle{abbrvnat}
\bibliography{bib.bib}

\section{Numerical Simulations}\label{apdx:Numerics}

This section provides more material for the numerical example shown in Section \ref{sec:numerics}. Figure \ref{fig:numerics-apdx} displays numerical performance for different initialization policies. All four algorithms perform well given a good initialization, i.e., initial policy close to a stable NE. However for bad initialization that is close to a non-NE stationary point, $\log$-barrier regularized algorithms can escape bad regions and converge to NE much faster than unregularized dynamics.

To examine why multi-agent learning suffers more from getting stuck at undesirable stationary points, we plot out the trajectory for $\overline{Q_i^{(t)}}(a_i), \pi_i^{(t)}(a_i)$ for both agents in Figure \ref{figure:NPG-detail-dynamic}. We will mainly focus our attention on the two plots on the left. Note that for the first few steps, $\overline{Q_1^{(t)}}(a_1 = 2)$ is much larger than $\overline{Q_1^{(t)}}(a_1 = 3)$, thus the natural gradient play scheme \eqref{eq:unregularized-NPG-2} will drive $\pi_1^{(t)}(a_1=2)$ close to $1$ and $\pi_1^{(t)}(a_1=3)$ close to $0$ very quickly. However, at around iteration $70$, $\overline{Q_1^{(t)}}(a_1 = 3)$ becomes slightly larger than $\overline{Q_1^{(t)}}(a_1 = 2)$. Unfortunately, at this stage, most of the probability is assigned to the suboptimal action $a_1=2$ and the optimal action receives $\pi_1^{(t)}(a_1=3)$ close to zero. Thus it will take more steps to bring $\pi_1^{(t)}(a_1=2)$ from $1$ to $0$ and $\pi_1^{(t)}(a_1=3)$ from $0$ to $1$, which reflects as the trajectory being stuck at the non-NE stationary policy with $\pi_1(a_1=3) =1$ in numerical behavior. From this simulation, we may conclude that one important reason for natural gradient play to get stuck at undesirable stationary points is due to the fact that the value of averaged $Q$-functions $\overline{Q_i^{(t)}}$'s for different actions might switch order during the learning process. In contrast, for single agent bandit learning, the averaged $Q$-function as well as the $Q$-function itself is the same as the reward value of a certain action $r(a)$, and thus will not change order, which explains why it can achieve dimension free convergence in single agent learning.
\begin{wrapfigure}{r}{.62\textwidth}
\vspace{-10pt}
    \centering
\begin{minipage}{.3\textwidth}
\centering
\includegraphics[width=.9\textwidth]{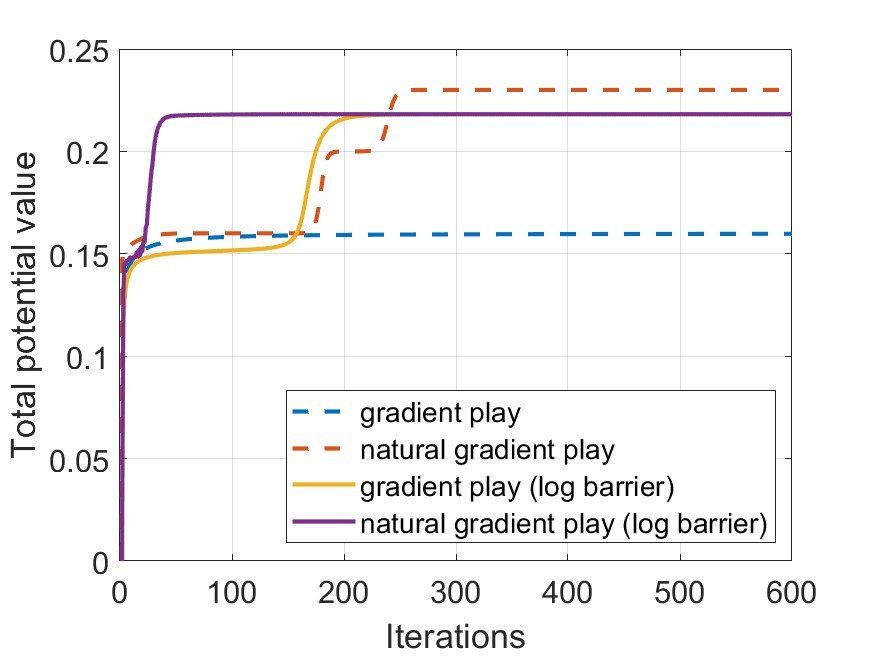}
\vspace{-5pt}
\caption{A 3-agent Example}\label{fig:3-agent}
\end{minipage}
\hspace{-10pt}
\begin{minipage}{.3\textwidth}
\centering
\includegraphics[width=.9\textwidth]{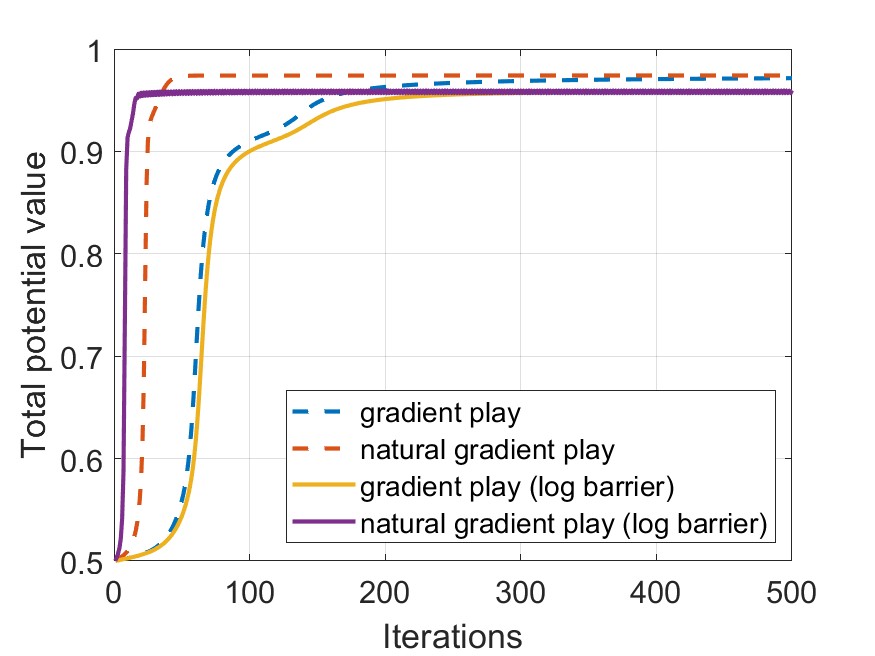}
\vspace{-5pt}
\caption{An 8-agent Example}\label{fig:8-agent}
\end{minipage}
\hspace{-20pt}
\vspace{-25pt}
\end{wrapfigure}

Additionally, we would like to remark that the algorithms considered in this paper also generalizes to settings with more agents, and similar phenomenon will still be observed. See Figure \ref{fig:3-agent} and \ref{fig:8-agent} for numerical simulations on a 3-agent example and an 8-agent example.

Running trajectories for all the four algorithms for one set of initializations takes approximately 2.04 seconds of CPU running time (Intel(R) Core(TM) i5-8250U CPU @ 1.60GHz 1.80 GHz).

\section{Derivation of Gradient and Performance Difference Lemma}\label{apdx:gradient}
\begin{proof} (of Equation \ref{eq:gradient-formula})
According to policy gradient theorem \citep{Sutton1999}:
\begin{align*}
    \frac{\partial J_i(\theta)}{\partial{\theta_{s,a_i}}} &= \frac{1}{1-\gamma}\sum_{s'} \sum_{a'} d_{\theta}(s') \pi_\theta(a'|s') \frac{\partial \log\pi_\theta(a'|s')}{\partial \theta_{s,a_i}}Q_i^\theta(s,a)
\end{align*}
Since for softmax parameterization:
\begin{align*}
    \frac{\partial \log\pi_\theta(a'|s')}{\partial \theta_{s,a_i}} =  \frac{\partial \log\pi_{\theta_i}(a_i'|s')}{\partial \theta_{s,a_i}} &= \mathbf{1}\{a_i' = a_i, s' = s\} - \mathbf{1}\{s' = s\}\pi_{\theta_i}(a_i|s)
\end{align*}
Thus we have that:
\begin{align*}
      &\quad \frac{\partial J_i(\theta)}{\partial{\theta_{s,a_i}}} = \frac{1}{1-\gamma}\sum_{s'} \sum_{a'} d_{\theta}(s') \pi_\theta(a'|s') \left( \mathbf{1}\{a_i' = a_i, s' = s\} -\mathbf{1}\{s' = s\}  \pi_{\theta_i}(a_i|s)\right) Q_i^\theta(s,a')\\
    &= \frac{1}{1\!-\!\gamma}d_\theta (s) \pi_{\theta_i}(a_i|s)\sum_{a_{-i}'}\pi_{\theta_{-i}}(a_{-i}'|s) Q_i^\theta(s,a_i, a_{-i}') - \frac{1}{1\!-\!\gamma}d_\theta(s)\pi_{\theta_i}(a_i|s) \sum_{a'}\pi_\theta(a'|s')Q_i^\theta(s,a')\\
    &=\frac{1}{1-\gamma}d_\theta (s) \pi_{\theta_i}(a_i|s)\overline{Q_i^\theta}(s,a_i, a_{-i}') - \frac{1}{1-\gamma}d_\theta(s)\pi_{\theta_i}(a_i|s) V_i^\theta(s)\\
    &= \frac{1}{1-\gamma}d_{\theta}(s)\pi_{\theta_i}(a_i|s)  \overline {A_i^{\theta}}(s, a_i)
\end{align*}
\end{proof}
We also introduce a useful lemma used throughout the proof which is derived from the performance difference lemma in MDP \citep{Kakade02}.
\begin{lemma}\label{lemma:performance-difference-lemma} Let $\theta' = (\theta_i', \theta_{-i}) $,
\begin{align*}
    J_i(\theta_i', \theta_{-i})- J_i(\theta_i, \theta_{-i}) = \frac{1}{1-\gamma} \sum_{s,a_i} d_{\theta'}(s) \pi_{\theta'_i}(a_i|s)\overline{A_i^{\theta}}(s, a_i)
\end{align*}
\begin{proof} From performance difference lemma \citep{Kakade02}
\begin{align*}
    J_i(\theta_i', \theta_{-i})- J_i(\theta_i, \theta_{-i}) &= \frac{1}{1-\gamma} \sum_{s,a} d_{\theta'}(s) \pi_{\theta'}(a|s) A_i^{\theta}(s, a)\\
    &=\frac{1}{1-\gamma} \sum_{s,a_i} d_{\theta'}(s) \pi_{\theta'_i}(a_i|s)\sum_{a_{-i}} \pi_{\theta_{-i}}(a_{-i}|s)A_i^{\theta}(s, a_i,a_{-i})\\
     &=\frac{1}{1-\gamma} \sum_{s,a_i} d_{\theta'}(s) \pi_{\theta'_i}(a_i|s)\overline{A_i^{\theta}}(s, a_i). \qedhere
\end{align*}
\end{proof}
\end{lemma}
\section{Derivation of Natural Gradient Play}\label{apdx:NPG-derivation}
\begin{lemma}\label{lemma:fisher-info-matrix}
\begin{equation*}
    \bE_{a\sim\pi_\theta(\cdot|s)}\left[\nabla_{\theta_{i,s}} \ \log\pi_{\theta_i}(a_i|s)\nabla_{\theta_{i,s}} \ \log\pi_{\theta_i}(a_i|s)^\top\right] = \diag\{\pi_{\theta_i,s}\} - \pi_{\theta_i,s}\pi_{\theta_i,s}^\top:= F_{i,s}(\theta_{i,s}),
\end{equation*}
where $\diag\{\cdot\}$ denotes the diagonal matrix generated by the corresponding vector, and $\pi_{\theta_i,s}\in \bR^{|\cA_i|}$ is the vector that denotes $\pi_{\theta_i}(\cdot|s)$. Further, $F_{i,s}(\theta_{i,s})$ is a semi-positive definite matrix, where the eigenvalue $0$ has the eigenspace of dimension 1 that is the span of the all one-vector $\mathbf{1}$.
\begin{proof} Calculating the gradient using chain rule we have
\begin{align*}
    \frac{\partial \log\pi_{\theta_i}(a_i|s)}{\partial \theta_{a_i',s}} = \mathbf{1}\{a_i' = a_i\} - \pi_{\theta_i}(a_i'|s).
\end{align*}
Let $\mathbf{1}_{a_i} \in \bR^{|\cA_i|}$ denote the vector where the entry corresponds to $a_i$ is 1 and other entries are zero. Then
\begin{align*}
    &\nabla_{\theta_{i,s}} \log\pi_{\theta_i}(a|s) = \mathbf{1}_{a_i} - \pi_{\theta_i,s}\\
    \Longrightarrow~~ &\nabla_{\theta_{i,s}} \ \log\pi_{\theta_i}(a_i|s)\nabla_{\theta_{i,s}}\ \log\pi_{\theta_i}(a_i|s)^\top = \diag\{\mathbf{1}_{a_i}\} - \pi_{\theta_i,s}\mathbf{1}_{a_i}^\top - \mathbf{1}_{a_i}\pi_{\theta_i,s}^\top + \pi_{\theta_i,s}\pi_{\theta_i,s}^\top
\end{align*}
Taking the expectation $\bE_{a\sim\pi_{\theta_i}(\cdot|s)}$ we have
\begin{align*}
    \bE_{a\sim\pi_{\theta_i}(\cdot|s)}\!\left[\nabla_{\theta_{i,s}}\!\! \log\pi_{\theta_i}(a|s)\nabla_{\theta_{i,s}}\!\! \log\pi_{\theta_i}(a|s)^{\!\top}\right] &= \diag\{\pi_{\theta_i,s}\} \!-\! \pi_{\theta_i,s}\pi_{\theta_i,s}^\top \!-\! \pi_{\theta_i,s}\pi_{\theta_i,s}^\top \!+\! \pi_{\theta_i,s}\pi_{\theta_i,s}^\top\\
    & = \diag\{\pi_{\theta_i,s}\} - \pi_{\theta_i,s}\pi_{\theta_i,s}^\top
\end{align*}
Further, for softmax parameterization, $\pi_{\theta_i}(a_i|s) > 0, ~\forall a_i$. Thus $F_{i,s}(\theta_{i,s})$ is a (non-strict) diagonally dominant matrix with diagonal entries all being positive and off-diagonal entries all being negative, in which case the all-one vector $\mathbf{1}$ is the only eigenvector for eigenvalue $0$.
\end{proof}
\end{lemma}
\begin{corollary}
\begin{equation*}
    F_i(\theta) = \textup{blkdiag}\{d_\theta(s)F_{i,s}(\theta_{i,s})\}_{s\in \cS},
\end{equation*}
where $\textup{blkdiag}\{\cdot\}$ denotes the block-diagonal matrix generated by corresponding sub-matrices.
\begin{proof}This is a direct corollary of Lemma \ref{lemma:fisher-info-matrix} , since
\begin{align*}
    \frac{\partial \log\pi_{\theta_i}(a_i|s)}{\partial \theta_{a_i',s'}} = 0, ~~\textup{for } s'\neq s,
\end{align*}
we have that
\begin{equation*}
F_i(\theta) = \bE_{s\!\sim\! d_\theta(\cdot)}\bE_{a_i\!\sim\!\pi_{\theta_i}(\cdot|s)}\left[\nabla_{\theta_i}\ \log\pi_{\theta_i}(a_i|s)\nabla_{\theta_i} \log\pi_\theta(a_i|s)^{\top}\right] = \textup{blkdiag}\{d_\theta(s)F_{i,s}(\theta_{i,s})\}    
\end{equation*}
\end{proof}
\end{corollary}

\begin{lemma}\label{lemma:NPG-derivation}
For vector $g:\cS\times\cA_i\rightarrow\bR$, with $\sum_{a_i}g(s,a_i)=0,~\forall s\in\cS$, we have that
\begin{equation*}
    \left[F_i(\theta)^\dagger g\right]_{(s,a_i)} = \frac{1}{d_\theta(s)\pi_{\theta_i}(a_i|s)}g(s,a_i) + c(s),
\end{equation*}
where $c(s)$ is a function that depend on state $s$ but not on $a_i$.
\begin{proof}
Since $F_i(\theta)$ is a block diagonal matrix,
\begin{align*}
    \left[F_i(\theta)^\dagger g\right]_{(s,\cdot)} = \frac{1}{d_\theta(s)}F_{i,s}(\theta_{i,s})^\dagger g(s,\cdot).
\end{align*}
From Lemma \ref{lemma:fisher-info-matrix}, since $F_{i,s}$ only has a one-dimensional eigenspace for eigenvalue $0$, and the eigenspace is the span of the all-one vector $\mathbf{1}$, we have that
\begin{equation*}
    F_{i,s}(\theta_{i,s})^\dagger F_{i,s}(\theta_{i,s}) = I - \frac{1}{|\cA_i|}\mathbf{1}\mathbf{1}^\top.
\end{equation*}
Let $f(s,a_i): = \frac{1}{d_\theta(s)\pi_{\theta_i}(a_i|s)}g(s,a_i)$
\begin{align*}
   d_\theta(s) \left[F_{i,s}(\theta_{i,s})f(s,\cdot)\right]_{a_i} &= d_\theta(s)\left(\pi_{\theta_i}(a_i|s)f(s,a_i) - \pi_{\theta_i}(a_i|s)\sum_{a_i'}\pi_{\theta_i}(a_i'|s)f(s,a_i')\right)\\
   &= g(s,a_i) - \pi_{\theta_i}(a_i|s)\sum_{a_i'}g(s,a_i') = g(s,a_i),
\end{align*}
i.e.,
\begin{align*}
    d_\theta(s)F_{i,s}(\theta_{i,s})f(s,\cdot) &= g(s,\cdot)\\
    \Longrightarrow~~ \frac{1}{d_\theta(s)}F_{i,s}(\theta_{i,s})^\dagger g(s,\cdot) &= F_{i,s}(\theta_{i,s})^\dagger F_{i,s}(\theta_{i,s})f(s,\cdot)\\
    &= \left(I - \frac{1}{|\cA_i|}\mathbf{1}\mathbf{1}^\top\right)f(s,\cdot)\\
    &= f(s,\cdot) - c(s)\mathbf{1},
\end{align*}
i.e.,
\begin{align*}
    \left[F_i(\theta)^\dagger g\right]_{(s,a_i)}= f(s,a_i) - c(s),
\end{align*}
which completes the proof.
\end{proof}
\end{lemma}
\begin{lemma}
Scheme \eqref{eq:unregularized-NPG} and \eqref{eq:unregularized-NPG-2} are equivalent. Similarly, \eqref{eq:log-barrier-NPG} and \eqref{eq:log-barrier-NPG-2} are equivalent.
\begin{proof}
It is not hard to check that $\nabla_{\theta_i}J_i(\theta), \nabla_{\theta_i}\widetilde J_i(\theta)$ satisfies
\begin{equation*}
    \sum_{a_i}\left[\nabla_{\theta_i}J_i(\theta)\right]_{(s,a_i)} =0,\quad \sum_{a_i}\left[\nabla_{\theta_i}\widetilde J_i(\theta)\right]_{(s,a_i)} =0,
\end{equation*}
thus we can apply Lemma \ref{lemma:NPG-derivation} and conclude
\begin{align*}
F_i(\theta^{(t)})^\dagger\nabla_{
    \theta_i}{J}_i(\theta_i^{(t)}) &= \frac{\overline{A_i^{(t)}}(s,a_i)}{1-\gamma} + c(s)\\
    F_i(\theta^{(t)})^\dagger\nabla_{
    \theta_i}\widetilde{J}_i(\theta_i^{(t)}) &= \frac{\overline{A_i^{(t)}}(s,a_i)}{1-\gamma}+\frac{\lambda}{d^{(t)}(s)\pi_i^{(t)}(a_i|s)} - \frac{\lambda|\cA_i|}{d^{(t)}(s)} + c(s),
\end{align*}
which completes the proof.
\end{proof}
\end{lemma}
\section{Proof of Lemma \ref{lemma:non-uniform-gradient-domination} and Lemma \ref{lemma:c-NE-equal-1}}\label{apdx:non-uniform-gradient-domination}
\begin{lemma}\label{lemma:NE-gap-bound}
\begin{equation*}
    \NEgap_i(\theta) \le \frac{1}{1-\gamma} \max_{s, a_i}\overline{A_i^{\theta}}(s, a_i), \quad \NEgap(\theta) \le \frac{1}{1-\gamma} \max_i\max_{s, a_i}\overline{A_i^{\theta}}(s, a_i).
\end{equation*}
\begin{proof}
From performance difference lemma
\begin{align*}
  J_i(\theta_i', \theta_{-i})- J_i(\theta_i, \theta_{-i})
     &=\frac{1}{1-\gamma} \sum_{s,a_i} d_{\theta'}(s) \pi_{\theta'_i}(a_i|s)\overline{A_i^{\theta}}(s, a_i) \quad \textup{(Lemma \ref{lemma:performance-difference-lemma})}\\
     &\le \frac{1}{1-\gamma} \sum_{s} d_{\theta'}(s) \max_{a_i}\overline{A_i^{\theta}}(s, a_i)\\\\
     &\le \frac{1}{1-\gamma} \sum_{s} d_{\theta'}(s) \max_{a_i}\overline{A_i^{\theta}}(s, a_i)\\
     &\le \frac{1}{1-\gamma} \max_{s, a_i}\overline{A_i^{\theta}}(s, a_i).
\end{align*}
Thus we have that
\begin{equation*}
    \NEgap_i(\theta) \le \frac{1}{1-\gamma} \max_{s, a_i}\overline{A_i^{\theta}}(s, a_i), \quad \NEgap(\theta) \le \frac{1}{1-\gamma} \max_i\max_{s, a_i}\overline{A_i^{\theta}}(s, a_i).
\end{equation*}
\end{proof}
\end{lemma}
\begin{proof}[Proof of Lemma \ref{lemma:non-uniform-gradient-domination}]
From Lemma \ref{lemma:NE-gap-bound} we have that
\begin{equation*}
    \NEgap_i(\theta) \le \frac{1}{1-\gamma} \max_{s, a_i}\overline{A_i^{\theta}}(s, a_i).
\end{equation*}
Since
\begin{align*}
    \max_{a_i}\overline{A_i^{\theta}}(s, a_i)&\le \frac{1}{\sum_{a_i^* \in \argmax_{a_i} \overline{Q_i^{\theta}}(s,a_i)}\pi_{\theta_i}(a_i^*|s)} \sum_{a_i}|\pi_{\theta_i}(a_i|s)\overline{A_i^\theta}(s,a_i)|\\
    &\le \frac{\sqrt{|\cA_i|}}{\sum_{a_i^* \in \argmax_{a_i} \overline{Q_i^{\theta}}(s,a_i)}\pi_{\theta_i}(a_i^*|s)} \sqrt{\sum_{a_i}\left(\pi_{\theta_i}(a_i|s)\overline{A_i^\theta}(s,a_i)\right)^2}\\
    &=  \frac{\sqrt{|\cA_i|}}{\sum_{a_i^* \in \argmax_{a_i} \overline{Q_i^{\theta}}(s,a_i)}\pi_{\theta_i}(a_i^*|s)}\frac{1-\gamma}{d_\theta(s)} \sqrt{\sum_{a_i}\left(\frac{1}{1-\gamma}d_\theta(s)\pi_{\theta_i}(a_i|s)\overline{A_i^\theta}(s,a_i)\right)^2}\\
    &\le \frac{(1-\gamma)M(\theta)\sqrt{|\cA_i|}}{c(\theta)}\|\nabla_{\theta_i}J_i(\theta)\|_2.
\end{align*}
Thus
\begin{align*}
     \NEgap_i(\theta) &\le\frac{1}{1-\gamma} \max_{s, a_i}\overline{A_i^{\theta}}(s, a_i)\\
     &\le\frac{\sqrt{|\cA_i|}M(\theta)}{c(\theta)}\|\nabla_{\theta_i}J_i(\theta)\|_2
\end{align*}
\end{proof}
\begin{remark}
{ A similar bound to Lemma \ref{lemma:non-uniform-gradient-domination} can be obtained by leveraging equation (259) in \citep{Mei20}
\begin{equation}\label{eq:non-uniform-gradient-domination-different}
    \NEgap_i(\theta) \le \frac{\sqrt{|\cS||\cA_i|} D_\infty}{c(\theta)}\|\nabla_{\theta_i}J_i(\theta)\|_2,~~\textup{where } D_\infty = \sup_{\theta, \theta'}\max_s \frac{d_{\theta'}(s)}{d_\theta(s)}.
\end{equation}
Notice that there's an additional $\sqrt{S}$ dependency on the right hand side compared with Lemma \ref{lemma:non-uniform-gradient-domination}, while replacing the term $M(\theta)$ by $D_\infty$. We remark that there's no fundamental difference between these two bounds. There's no significant difference in the proof techniques and it is hard to tell which one is better. One can also easily re-derive the set of analysis in the paper using \eqref{eq:non-uniform-gradient-domination-different}, with bounds that depends on $D_\infty$ instead of $M$ and slightly differs in the dependency on $\cS$ from our current result.}
\end{remark}
\begin{proof}[Proof of Lemma \ref{lemma:c-NE-equal-1}]
Firstly, it is straightforward to see that if $c(\theta^*)\neq 0$, $\theta^*$ is a NE by applying Lemma ~\ref{lemma:non-uniform-gradient-domination}. So we only need to focus on proving that if $\theta^*$ is a NE, then $c(\theta^*)=1$. 

From performance difference lemma, let $\theta' := (\theta_i', \theta_{-i}^*)$ 
\begin{align*}
     J_i(\theta_i', \theta_{-i}^*)- J_i(\theta_i^*, \theta_{-i}^*) =\frac{1}{1-\gamma} \sum_{s,a_i} d_{\theta'}(s) \pi_{\theta'_i}(a_i|s)\overline{A_i^{\theta^*}}(s, a_i)
\end{align*}
Select $a_i^*(s) \in \argmax_{a_i} \overline{A_i^{\theta^*}}(s,a_i)$ and set:
\begin{equation*}
    \pi_{\theta_i'}(a_i|s) =  \mathbf{1}\{a_i = a_i^*(s)\},
\end{equation*}
then
\begin{align*}
    J_i(\theta_i', \theta_{-i}^*)- J_i(\theta_i^*, \theta_{-i}^*) &=\frac{1}{1-\gamma} \sum_{s,a_i} d_{\theta'}(s) \pi_{\theta'_i}(a_i|s)\overline{A_i^{\theta^*}}(s, a_i)\\
    &=\frac{1}{1-\gamma} \sum_{s}d_\theta(s)\max_{a_i}\overline{A_i^{\theta^*}}(s, a_i) \ge 0.
\end{align*}
Since $\theta^*$ is a NE,
\begin{equation*}
    \Longrightarrow \max \overline{A_i^{\theta^*}}(s,a_i) = 0, ~~\forall~s, ~\forall~i.
\end{equation*}
Let $\Delta:= \min_s\min_{a_i\notin\argmax_{a_i} \overline{A_i^{\theta^*}}(s,a_i)} |\overline{A_i^{\theta^*}}(s,a_i)|$.
Since $\sum_{a_i}\pi_{\theta_i^*}(a_i|s)\overline{A_i^{\theta^*}}(s,a_i) = 0$
\begin{align*}
    \Longrightarrow~0&=  \sum_{a_i\in\argmax_{a_i} \overline{A_i^{\theta^*}}(s,a_i)} \pi_{\theta_i^*}(a_i|s)\max_{a_i}\overline{A_i^{\theta^*}}(s,a_i) +  \sum_{a_i\notin\argmax_{a_i} \overline{A_i^{\theta^*}}(s,a_i)} \pi_{\theta_i^*}(a_i|s)\overline{A_i^{\theta^*}}(s,a_i)\\
    &\le -\Delta \sum_{a_i\notin\argmax_{a_i} \overline{A_i^{\theta^*}}(s,a_i)} \pi_{\theta_i^*}(a_i|s)\\
    \Longrightarrow~& \sum_{a_i\notin\argmax_{a_i} \overline{A_i^{\theta^*}}(s,a_i)} \pi_{\theta_i^*}(a_i|s) = 0\\
    \Longrightarrow~& \sum_{a_i\in\argmax_{a_i} \overline{A_i^{\theta^*}}(s,a_i)} \pi_{\theta_i^*}(a_i|s) = 1\\
    \Longrightarrow~& \sum_{a_i\in\argmax_{a_i} \overline{Q_i^{\theta^*}}(s,a_i)} \pi_{\theta_i^*}(a_i|s) = 1\\
    \Longrightarrow~& c(\theta^*) = 1 \qedhere
\end{align*}
\end{proof}

\section{Proof of Theorem \ref{thm:asymptotic-convergence}}\label{apex:asymptotic-convergence}
\subsubsection{Asymptotic convergence for gradient play}
\begin{lemma}\label{lemma:monotonicity-value-function}
For $\eta \le \frac{(1-\gamma)^3}{6n}$, running scheme \eqref{eq:unregularized-PG} will guarantee that \( \lim_{t\rightarrow+\infty}\nabla\Phi(\theta^{(t)}) = 0\).
\begin{proof}
Since $\Phi(\theta)$ is $\beta$-smooth w.r.t. $\theta$, where $\beta = \frac{6n}{(1-\gamma)^3}$
\begin{align*}
    \Phi(\theta^{(t+1)}) - \Phi(\theta^{(t)})&\ge \left<\nabla\Phi(\theta^{(t)}), \theta^{(t+1)}-\theta^{(t)}\right>-\frac{\beta}{2}\|\theta^{(t+1)}-\theta^{(t)}\|_2^2\\
    &\ge \frac{\eta}{2}\|\nabla\Phi(\theta^{(t)}\|_2^2 \ge 0
\end{align*}
which proves the monotonicity of $\Phi(\theta^{(t)})$. Since $\phi$ is a bounded function, this gives:
\begin{equation*}
    \lim_{t\rightarrow+\infty}\|\nabla\Phi(\theta^{(t)})\|_2 = 0. \qedhere
\end{equation*}
\end{proof}
\end{lemma}

From Lemma \ref{lemma:monotonicity-value-function} and that the stationary points are isolated, we know that the limit for $\theta^{(t)}$ exists, i.e., it is valid to define
\begin{equation*}
    \theta^{(\infty)}:= \limtinf \theta^{(t)}.
\end{equation*}
We abbreviate the related functions with respect to $\theta^{(\infty)}$ as follows:
\begin{align*}
    Q_i^{(\infty)}(s,a)&:= Q_i^{\theta^{(\infty)}}(s,a),\qquad 
    V_i^{(\infty)}(s) := V_i^{\theta^{(\infty)}}(s),\qquad A_i^{(\infty)}(s,a) := Q_i^{(\infty)}(s,a) - V_i^{(\infty)}(s)\\
    \overline{Q_i^{(\infty)}}(s,a_i) &:= \sum_{a_{-i}}\pi_{-i}^{(\infty)}(a_{-i}|s)Q^{(\infty)}(s,a_i,a_{-i}), \quad \overline{A_i^{(\infty)}}(s,a_i) := \sum_{a_{-i}}\pi_{-i}^{(\infty)}(a_{-i}|s)A^{(\infty)}(s,a_i,a_{-i})
\end{align*}

Since $\theta^{(\infty)}$ is the limit of $\theta^{(t)}$, we have that:
\begin{equation}\label{eq:convergence-average-Q-A}
    \limtinf \overline{Q_i^{(t)}}(s,a_i) =\overline{Q_i^{(\infty)}}(s,a_i), \quad 
    \limtinf \overline{A_i^{(t)}}(s,a_i) = \overline{A_i^{(\infty)}}(s,a_i)
\end{equation}
Define:
\begin{align*}
    I_0^{i,s}&:=\{a_i|\overline{Q_i^{(\infty)}}(s,a_i) = V^{(\infty)}(s)\} = \{a_i|\overline{A_i^{(\infty)}}(s,a_i) = 0\}\\
    I_+^{i,s}&:=\{a_i|\overline{Q_i^{(\infty)}}(s,a_i) > V^{(\infty)}(s)\} = \{a_i|\overline{A_i^{(\infty)}}(s,a_i) >0\}\\
    I_-^{i,s}&:=\{a_i|\overline{Q_i^{(\infty)}}(s,a_i) < V^{(\infty)}(s)\} = = \{a_i|\overline{A_i^{(\infty)}}(s,a_i) < 0\}\\
\end{align*}
Let
\begin{equation}\label{eq:def-Delta}
    \Delta:=\min_i\min_{\{s,a_i|A_i^{(\infty)}(s,a_i)\neq 0\}}|A_i^{(\infty)}(s,a_i)|
\end{equation}
From Lemma \ref{lemma:NE-gap-bound}, it is sufficient to show that $I_+^{i,s}=\emptyset, ~\forall~ i, s$.

From the Lemma \ref{lemma:monotonicity-value-function} and the above definitions we have the following corollaries:
\begin{corollary}\label{coro:average-A-bound-T_1} There exists $T_1$, such that $\forall t>T_1, ~\forall s\in\cS,~\forall i\in\{1,2,\dots,n\}$,
\begin{align*}
    A_i^{(t)}(s,a_i) &< -\frac{\Delta}{4},~~\forall a_i\in I_-^{i,s}\\
    A_i^{(t)}(s,a_i) &> \frac{\Delta}{4},~~\forall a_i\in I_+^{i,s}\\
    |A_i^{(t)}(s,a_i)|& < \frac{\Delta}{4},~~\forall a_i\in I_0^{i,s}
\end{align*}
\begin{proof}
This is a direct corollary from \eqref{eq:convergence-average-Q-A} and \eqref{eq:def-Delta}.
\end{proof}
\end{corollary}
\begin{corollary}\label{coro:measure-on-three-sets}
\begin{align*}
    \limtinf \sum_{a_i\in I_0^{i,s}}\pi_i^{(t)}(a_i|s) &= 1\\
    \limtinf \sum_{a_i\in I_+^{i,s}\cup I_-^{i,s}}\pi_i^{(t)}(a_i|s) &= 0
\end{align*}

\begin{proof}
This is a direct corollary from Lemma \ref{lemma:monotonicity-value-function},
\begin{align*}
    &\limtinf\nabla\Phi(\theta^{(t)}) = 0\\
    \Longrightarrow &\limtinf \frac{\partial \Phi(\theta^{(t)})}{\partial \theta_{s,a_i}} = \limtinf \frac{1}{1-\gamma}d^{(t)}(s)\pi_i^{(t)}(a_i|s)\overline{A_i^{(t)}}(s,a_i) =0\\
    \Longrightarrow &\limtinf \pi_i^{(t)}(a_i|s) \limtinf \overline{A_i^{(t)}}(s,a_i) = 0\\
    \Longrightarrow &\limtinf \pi_i^{(t)}(a_i|s) = 0,~~\forall a_i \notin I_0^{i,s}\\
    \Longrightarrow &\limtinf \sum_{a_i\in I_+^{i,s}\cup I_-^{i,s}}\pi_i^{(t)}(a_i|s) = 0\\
    \Longrightarrow   &\limtinf \sum_{a_i\in I_0^{i,s}}\pi_i^{(t)}(a_i|s) = 1 - \limtinf \sum_{a_i\in I_+^{i,s}\cup I_-^{i,s}}\pi_i^{(t)}(a_i|s) = 1 \qedhere
\end{align*}
\end{proof}

\end{corollary}
\begin{lemma}\label{lemma:go-to-infinity-I-minus}
$\forall a_i\in I_+^{i,s}, \theta_{s,a_i}^{(t)}$ is bounded from below. $\forall a_i\in I_-^{i,s}, \limtinf\theta_{s,a_i}^{(t)} = -\infty$.
\begin{proof}
The first statement, $\forall a_i\in I_+^{i,s}, \theta_{s,a_i}^{(t)}$ is bounded from below, is trivial from Corollary \ref{coro:average-A-bound-T_1}. We only need to prove the second statement. The key observation is that:
\begin{equation*}
    \sum_{a_i} \frac{\partial \Phi(\theta^{(t)})}{\partial \theta_{s,a_i}} = \frac{1}{1-\gamma}d^{(t)}(s)\sum_{a_i}\pi_i^{(t)}(a_i|s)\overline{A_i^{(t)}}(s,a_i) = 0
\end{equation*}
Thus
\begin{equation*}
    \sum_{a_i}\theta_{s,a_i}^{(t)} = \sum_{a_i}\theta_{s,a_i}^{(0)}.
\end{equation*}
From Corollary \ref{coro:measure-on-three-sets}, we have that
\begin{align*}
    \limtinf \sum_{a_i\in I_+^{i,s}\cup I_-^{i,s}}\pi_i^{(t)}(a_i|s) = 0\\
    \Longrightarrow \exists ~a_i\in I_0^{i,s},~s.t.~ \limsuptinf \theta_{s,a_i}^{(t)} = +\infty\\
\end{align*}
And since all $\theta_{s,a_i}^{(t)}$ sum up to a constant and that $\forall a_i\in I_+^{i,s}, \theta_{s,a_i}^{(t)}$ is bounded from below, we have that:
\begin{equation}\label{eq:overline-a-i}
    \exists~~\overline{a_i}\in I_0^{i,s}\cup I_-^{i,s},~s.t.~ \liminftinf \theta_{s,\overline{a_i}}^{(t)} = -\infty.
\end{equation}
From Corollary \ref{coro:average-A-bound-T_1}, for $a_i\in I_-^{i,s}$, $\theta_{s,a_i}^{(t)}$ is monotonically decreasing for $t>T_1$, thus
\begin{equation*}
    \limtinf \theta_{s,a_i}^{(t)} := \theta_{s,a_i}^{(\infty)},
\end{equation*}
where $\theta_{s,a_i}^{(\infty)}$ is either a constant or $-\infty$. We'll prove by contradiction. Suppose $\theta_{s,a_i}^{(\infty)}$ is a constant, then for any $\delta>0$ there exists $T_1'\ge T_1$ such that $\forall~t\ge T_1', ~|\theta_{s,a_i}^{(t)} - \theta_{s,a_i}^{(\infty)}| \le \delta$.

Let $\overline{a_i}\in\cA_i$ be defined as in \eqref{eq:overline-a-i}, define:
\begin{equation*}
    \tau(t) :=\left\{
    \begin{array}{lr}
        t+1, & \textup{if } \theta_{s,\overline{a_i}}^{(t)} > \theta_{s,a_i}^{(\infty)}-\delta \\
        \min_{t'}\{T_1'\le t'\le t| \theta_{s,\overline{a_i}}^{(\tau)} \le \theta_{s,a_i}^{(\infty)}-\delta, ~\forall t'\le\tau\le t \}, &\textup{otherwise} 
    \end{array}
    \right.
\end{equation*}
 We will focus on the set where $\{t|\tau(t)\le t\}$. Since $\liminftinf\theta_{s,\overline{a_i}}^{(t)} = -\infty$, there are infinitely many elements in this set. 
 
 For all $\tau(t)\le\tau\le t$, we have that:
 \begin{align*}
     \left|\frac{\frac{\partial \Phi(\theta^{(\tau)})}{\partial \theta_{s,a_i}}}{\frac{\partial \Phi(\theta^{(\tau)})}{\partial \theta_{s,\overline{a_i}}}}\right| = \left|\frac{\pi_i^{(\tau)}(a_i|s)\overline{A_i^{(\tau)}}(s,a_i)}{\pi_i^{(\tau)}(\overline{a_i}|s)\overline{A_i^{(\tau)}}(s,\overline{a_i})}\right| &= \exp{(\theta_{s,a_i}^{(\tau)} - \theta_{s,\overline{a_i}}^{(\tau)})}\left|\frac{\overline{A_i^{(\tau)}}(s,a_i)}{\overline{A_i^{(\tau)}}(s,\overline{a_i})}\right|\\
     &\ge \left|\frac{\overline{A_i^{(\tau)}}(s,a_i)}{\overline{A_i^{(\tau)}}(s,\overline{a_i})}\right| \ge \frac{\Delta(1-\gamma)}{4}
 \end{align*}
 Thus
 \begin{align}
     &\frac{\partial \Phi(\theta^{(\tau)})}{\partial \theta_{s,a_i}} \le \frac{\Delta(1-\gamma)}{4}\frac{\partial \Phi(\theta^{(\tau)})}{\partial \theta_{s,\overline{a_i}}}, ~~\tau(t)\le\tau\le t\notag\\
     \Longrightarrow &\frac{1}{\eta}(\theta_{s,a_i}^{(t+1)} - \theta_{s,a_i}^{(\tau(t))}) =\sum_{\tau(t)}^t\frac{\partial \Phi(\theta^{(\tau)})}{\partial \theta_{s,a_i}} \le \frac{\Delta(1-\gamma)}{4} \sum_{\tau(t)}^t \frac{\partial \Phi(\theta^{(\tau)})}{\partial \theta_{s,\overline{a_i}}}= \frac{\Delta(1-\gamma)}{4\eta} (\theta_{s,\overline{a_i}}^{(t+1)} - \theta_{s,\overline{a_i}}^{(\tau(t))})\label{eq:asymptotic-eq-1}
 \end{align}
 Since:
 \begin{equation*}
     \theta_{s,\overline{a_i}}^{(\tau(t))} \ge\theta_{s,\overline{a_i}}^{(\tau(t)-1)} - \eta \frac{1}{(1-\gamma)^2} \ge \theta_{s,a_i}^{(\infty)}-\delta- \eta \frac{1}{(1-\gamma)^2}
 \end{equation*}
 is bounded from below, and that $\theta_{s,a_i}^{\tau(t)}$ is also bounded from above by $\theta_{s,a_i}^{(T_1)}$, thus
 taking $\liminftinf$ on both sides of eq \eqref{eq:asymptotic-eq-1} will give
 \begin{align*}
     &\liminftinf \theta_{s,a_i}^{(t+1)} - \theta_{s,a_i}^{(\tau(t))} \le \frac{\Delta(1-\gamma)}{4} \left(\liminftinf \theta_{s,\overline{a_i}}^{(t+1)} -\theta_{s,a_i}^{(\infty)}+\delta+ \eta \frac{1}{(1-\gamma)^2}\right) = -\infty\\
     &\Longrightarrow \liminftinf \theta_{s,a_i}^{(t)} = -\infty
 \end{align*}
 which contradicts the assumption that $\theta_{s,a_i}^{(\infty)}$ is a constant, and thus we can conclude that 
\begin{equation*}
    \limtinf\theta_{s,a_i}^{(t)} = -\infty,~~\forall a_i\in I_-^{i,s}. \qedhere
\end{equation*}
\end{proof}

\end{lemma}

\begin{lemma}\label{lemma:bound-I-0-by-I-plus}
$\forall a_i^+\in I_+^{i,s}$, for any $a\in I_0^{i,s}$, if there exists $t \ge T_1$ such that $\pi_i^{(t)}(a_i|s)\le\pi_i^{(t)}(a_i^+|s)$, then for all $\tau\ge t$, $\pi_i^{(\tau)}(a_i|s)\le\pi_i^{(\tau)}(a_i^+|s)$
\end{lemma}
\begin{proof}
We will prove by induction. Suppose for a certain $\tau \ge t$, it holds that $\pi_i^{(\tau)}(a_i|s)\le\pi_i^{(\tau)}(a_i^+|s)$, then:
\begin{align*}
    \frac{\partial \Phi(\theta^{(\tau)})}{\partial \theta_{s,a_i^+}} &= \frac{1}{1-\gamma}d^{(\tau)}(s)\pi_i^{(\tau)}(a_i^+|s)\overline{A_i^{(\tau)}}(s,a_i^+)\\
    &\ge \frac{1}{1-\gamma}d^{(\tau)}(s)\pi_i^{(\tau)}(a_i|s)\overline{A_i^{(\tau)}}(s,a_i^+)\\
    &\ge \frac{1}{1-\gamma}d^{(\tau)}(s)\pi_i^{(\tau)}(a_i|s)\overline{A_i^{(\tau)}}(s,a_i)\\
    &= \frac{\partial \Phi(\theta^{(\tau)})}{\partial \theta_{s,a_i}}
\end{align*}
Since $\pi_i^{(\tau)}(a_i|s)\le\pi_i^{(\tau)}(a_i^+|s) ~\Longrightarrow~ \theta_{s,a_i}^{(\tau)}\le\theta_{s,a_i^+}^{(\tau)}$, we have:
\begin{align*}
    \theta_{s,a_i^+}^{(\tau+1)} = \theta_{s,a_i^+}^{(\tau)} + \eta \frac{\partial \Phi(\theta^{(\tau)})}{\partial \theta_{s,a_i^+}} \ge \theta_{s,a_i}^{(\tau)} + \eta \frac{\partial \Phi(\theta^{(\tau)})}{\partial \theta_{s,a_i}} = \theta_{s,a_i}^{(\tau+1)}
\end{align*}
Thus $\pi_i^{(\tau+1)}(a_i|s)\le\pi_i^{(\tau+1)}(a_i^+|s)$ also holds, which completes the proof.
\end{proof}

\begin{lemma}\label{lemma:I-plus-empty}
$I_+^{i,s} = \emptyset$.
\end{lemma}
\begin{proof}
We will prove by contradiction. If $I_+^{i,s} \neq \emptyset$, select an arbitrary $a_i^+\in I_+^{i,s}$ and define
$$B_0^{i,s}(a_i^+):= \{a_i\in I_0^{i,s}~|~ \pi_i^{(t)}(a_i|s)\le\pi_i^{(t)}(a_i^+|s), \forall t \ge T_1\}.$$
From Lemma \ref{lemma:go-to-infinity-I-minus}, we have that for any $a_i\in I_-^{i,s}$ $\limtinf\frac{\pi_i^{(t)}(a_i|t)}{\pi_i^{(t)}(a_i^+|t)} = 0$, thus there exists $T_2>T_1$ such that for any $t\ge T_2$,
\begin{equation*}
    \frac{\pi_i^{(t)}(a_i|t)}{\pi_i^{(t)}(a_i^+|t)} \le \frac{(1-\gamma)\Delta}{16|\cA_i|}, \quad \forall a_i\in I_-^{i,s}.
\end{equation*}
Additionally, since for any $a_i\in I_0^{i,s}, \limtinf \overline{A_i^{(t)}}(s, a_i) =0$, there exists $T_3>T_1$ such that for any $t\ge T_3$,
\begin{equation*}
    \overline{A_i^{(t)}}(s, a_i) \ge \frac{-\Delta}{16|\cA_i|},\quad \forall a_i\in I_0^{i,s}.
\end{equation*}
Thus, for $t\ge\max\{T_2, T_3\}$, from the fact that $\sum_{a_i}\pi_i^{(t)}(a_i|s)\overline{A_i^{(t)}}(s,a_i) = 0$, we have:
\begin{align*}
    0 &= \sum_{a_i\in I_0^{i,s}}\pi_i^{(t)}(a_i|s)\overline{A_i^{(t)}}(s,a_i) + \sum_{a_i\in I_+^{i,s}}\pi_i^{(t)}(a_i|s)\overline{A_i^{(t)}}(s,a_i) + \sum_{a_i\in I_-^{i,s}}\pi_i^{(t)}(a_i|s)\overline{A_i^{(t)}}(s,a_i)\\
    &\ge  \sum_{a_i\in I_0^{i,s}\backslash B_0^{i,s}(a_i^+)}\pi_i^{(t)}(a_i|s)\overline{A_i^{(t)}}(s,a_i) + \sum_{a_i\in B_0^{i,s}(a_i^+)}\pi_i^{(t)}(a_i|s)\overline{A_i^{(t)}}(s,a_i)  \\&\qquad+\pi_i^{(t)}(a_i^+|s)\overline{A_i^{(t)}}(s,a_i^+)+ \sum_{a_i\in I_-^{i,s}}\pi_i^{(t)}(a_i|s)\overline{A_i^{(t)}}(s,a_i)\\
    &\ge  \sum_{a_i\in I_0^{i,s}\backslash B_0^{i,s}(a_i^+)}\pi_i^{(t)}(a_i|s)\overline{A_i^{(t)}}(s,a_i)+\sum_{a_i\in B_0^{i,s}(a_i^+)} \pi_i^{(t)}(a_i|s)\frac{-\Delta}{16|\cA_i|}\\&\qquad+\pi_i^{(t)}(a_i^+|s)\overline{A_i^{(t)}}(s,a_i^+)+ \sum_{a_i\in I_-^{i,s}}\frac{(1-\gamma)\Delta}{16|\cA_i|}\pi_i^{(t)}(a_i^+|s)\overline{A_i^{(t)}}(s,a_i)\\
    &\ge \sum_{a_i\in I_0^{i,s}\backslash B_0^{i,s}(a_i^+)}\pi_i^{(t)}(a_i|s)\overline{A_i^{(t)}}(s,a_i)+ |\cA_i|\pi_i^{(t)}(a_i|s)\frac{-\Delta}{16|\cA_i|}\\&\qquad+\pi_i^{(t)}(a_i^+|s)\frac{\Delta}{4}+ |\cA_i|\frac{(1-\gamma)\Delta}{16|\cA_i|}\pi_i^{(t)}(a_i^+|s)\frac{-1}{1-\gamma}\\
    &\ge \sum_{a_i\in I_0^{i,s}\backslash B_0^{i,s}(a_i^+)}\pi_i^{(t)}(a_i|s)\overline{A_i^{(t)}}(s,a_i)+ \pi_i^{(t)}(a_i^+|s)\frac{\Delta}{8}\\
    \Longrightarrow~& \sum_{a_i\in I_0^{i,s}\backslash B_0^{i,s}(a_i^+)}\pi_i^{(t)}(a_i|s)\overline{A_i^{(t)}}(s,a_i) < 0.
\end{align*}
Thus for $t\ge\max\{T_2, T_3\}$,
\begin{equation*}
\begin{split}
    \sum_{a_i\in I_0^{i,s}\backslash B_0^{i,s}(a_i^+)}\theta^{(t+1)}_{s,a_i} &= \sum_{a_i\in I_0^{i,s}\backslash B_0^{i,s}(a_i^+)}\theta^{(t)}_{s,a_i} + \eta \frac{1}{1-\gamma}d^{(t)}(s) \sum_{a_i\in I_0^{i,s}\backslash B_0^{i,s}(a_i^+)}\pi_i^{(t)}(a_i|s)\overline{A_i^{(t)}}(s,a_i)\\ &< \sum_{a_i\in I_0^{i,s}\backslash B_0^{i,s}(a_i^+)}\theta^{(t)}_{s,a_i},
\end{split}
\end{equation*}
which leads to the fact that $\sum_{a_i\in I_0^{i,s}\backslash B_0^{i,s}(a_i^+)}\theta^{(t)}_{s,a_i}$ is bounded from above. Further, from Lemma \ref{lemma:go-to-infinity-I-minus}, $\theta_{s,a_i^+}^{(t)}$ is bounded from below, thus the value
$$\frac{\sum_{a_i\in I_0^{i,s}\backslash B_0^{i,s}(a_i^+)}\pi_i^{(t)}(a_i|s)}{\pi_i^{(t)}(a_i^+|s)}$$
is bounded from above. However from Corollary \ref{coro:measure-on-three-sets},
$$\limtinf\frac{\sum_{a_i\in I_0^{i,s}}\pi_i^{(t)}(a_i|s)}{\pi_i^{(t)}(a_i^+|s)} = +\infty.$$
Thus
$$\limtinf\frac{\sum_{a_i\in B_0^{i,s}(a_i^+)}\pi_i^{(t)}(a_i|s)}{\pi_i^{(t)}(a_i^+|s)} = +\infty,$$
which contradicts the fact that
$$\pi_i^{(t)}(a_i|s)\le \pi_i^{(t)}(a_i^+|s), \quad \forall a_i\in B_0^{i,s}(a_i^+)$$
and finishes the proof by contradiction.
\end{proof}
Lemma \ref{lemma:I-plus-empty} directly implies asymptotic convergence for gradient play as state in Theorem \ref{thm:asymptotic-convergence}.

\begin{remark}\textbf{(Discussion on the isolated stationary points assumption)}
The proof of Theorem \ref{thm:asymptotic-convergence} resembles the technique used in \citep{agarwal2020} for the single agent case, which relies heavily on the fact that the sequence of $Q$-functions $Q^{(t)}(s,a)$ obtains a limit $Q^{(\infty)}(s,a)$. The existence of such a limit in the single agent case  follows from the monotonicity of the $Q$-functions. However, generalizing this proof to the multi-agent case requires the assumption that the sequence of averaged $Q$-functions $\overline{Q_i^{(t)}}(s,a_i)$ (which can be non-monotonic, see, e.g., Figure \ref{figure:NPG-detail-dynamic} in Appendix) has a limit $\overline{Q_i^{(\infty)}}(s,a_i)$, which is not necessarily true in general.  For instance, if the set of stationary policies $\mathcal{SP}:=\!\! \left\{\pi\!:\!\pi_{i}(a_i|s)\overline{A_i^\pi}(s,a_i) \!= \!0,\forall s\!\in\!\cS,a_i\!\in\!\cA,i \!= \!1\!,\!2\!,\!\dots\!,\!n\right\}$ is not isolated, one cannot rule out the possibility that (natural) gradient play will not converge to a fixed point $\pi^{(\infty)}$ 
(see e.g. \citep{absil2005} for counterexamples). Consequently, $\overline{Q_i^{(t)}}(s,a_i)$ might not converge to a single value. 
For the above reasons, we assume the stationary policies are isolated to ensure that $\pi^{(t)}$ converges to a fixed stationary policy $\pi^{(\infty)}$ and thus $\overline{Q_i^{(t)}}(s,a_i)$ obtains a limit. We believe that this assumption is  a conservative  condition that is sufficient to imply asymptotic convergence. It remains an interesting open question to establish convergence without this assumption.
\end{remark}
\subsubsection{Asymptotic convergence for natural gradient play}
The asymptotic convergence for natural gradient play is easier to establish compared with gradient play. 

From Lemma \ref{lemma:unregularized-NPG-1} and the assumption that $\phi(s,a)$ is upper-bounded, we know 
\begin{align*}
    \limtinf \sum_{a_i}\pi_i^{(t)}(a_i|s)\exp{\left(\frac{\eta\overline{A_i^{(t)}}(s,a_i)}{1-\gamma}\right)} = 1, ~\forall s, i = 1,2,\dots, n.
\end{align*}
Since
\begin{align*}
    \sum_{a_i}\pi_i^{(t)}(a_i|s)\exp{\left(\frac{\eta\overline{A_i^{(t)}}(s,a_i)}{1-\gamma}\right)} &\ge \sum_{a_i}\pi_i^{(t)}(a_i|s)\left(1 +\left(\frac{\eta\overline{A_i^{(t)}}(s,a_i)}{1-\gamma}\right) + \frac{1}{4}\left(\frac{\eta\overline{A_i^{(t)}}(s,a_i)}{1-\gamma}\right)^2 \right)\\
    &\qquad\qquad\qquad\qquad\qquad\qquad\qquad \textup{($e^x \ge 1 + x + \frac{x^2}{4}$ for $|x|\le 1$)}\\
    & = 1 + \frac{\eta^2}{4(1-\gamma)^2}\sum_{a_i}\pi_i^{(t)}(a_i|s)\overline{A_i^{(t)}}(s,a_i)^2\\
\Longrightarrow \quad &\limtinf \sum_{a_i}\pi_i^{(t)}(a_i|s)\overline{A_i^{(t)}}(s,a_i)^2 = 0\\
\Longrightarrow \quad &\limtinf \pi_i^{(t)}(a_i|s)\overline{A_i^{(t)}}(s,a_i) = 0, ~\forall s, a_i, i = 1,2, \dots, n\\
\Longrightarrow \quad &\limtinf \|\nabla_\theta\Phi(\theta^{(t)})\|_2 = 0
\end{align*}
Similar to the proof for gradient play, from the assumption that stationary points are isolated, we can conclude that $\pi^{(t)}$ converges to some stationary policy $\pi^{(\infty)}$,  and we can define $\overline{Q_i^{(\infty)}}(s,a_i), \overline{A_i^{(\infty)}}(s,a_i)$ accordingly. Asymptotic convergence is equivalent to
$$I_+^{i,s}:=\left\{a_i: \overline{A_i^{(\infty)}}(s,a_i) >0\right\} = \emptyset, \quad \forall s, ~i=1,2,\dots,n$$
We prove by contradiction. Suppose there exists $a_i^+$ such that $\overline{A_i^{(\infty)}}(s,a_i^+) >0$. From $\limtinf \pi_i^{(t)}(a_i|s)\overline{A_i^{(t)}}(s,a_i) = 0 $, we have that $\limtinf \pi_i^{(t)}(a_i^+|s) = 0.$

Select $a_i^0$ such that $\limtinf \pi_i^{(t)}(a_i^0|s) > 0.$ From $\limtinf \pi_i^{(t)}(a_i|s)\overline{A_i^{(t)}}(s,a_i) = 0 $, we have that $ \limtinf \overline{A_i^{(t)}}(s,a_i^0) = 0.$ Thus there exists $\Delta > 0$ and $T$ such that for $t > T$, 
\begin{align*}
    \overline{A_i^{(t)}}(s,a_i^+) > \Delta, \overline{A_i^{(t)}}(s,a_i^0) < \frac{\Delta}{2}
\end{align*}
Thus from natural gradient play scheme \eqref{eq:unregularized-NPG-2}
\begin{align*}
    \frac{\pi_i^{(t)}(a_i^+|s)}{\pi_i^{(t)}(a_i^0|s)} =\frac{\pi_i^{(T)}(a_i^+|s)}{\pi_i^{(T)}(a_i^0|s) }  \exp{\left(\frac{\eta}{1-\gamma}\sum_{\tau=T}^{t-1}\overline{A_i^{(\tau)}}(s,a_i^+)-\overline{A_i^{(\tau)}}(s,a_i^0)\right)} \ge \frac{\pi_i^{(T)}(a_i^+|s)}{\pi_i^{(T)}(a_i^0|s) },
\end{align*}
which contradict the fact that $\limtinf\frac{\pi_i^{(t)}(a_i^+|s)}{\pi_i^{(t)}(a_i^0|s)} =0$, and thus completes the proof.
\section{Proof of Theorem \ref{thm:unregularized-PG}}\label{apdx:unregularized-PG-NPG}
\subsection{Proof of Theorem \ref{thm:unregularized-PG} (Gradient play part)}\label{apdx:unregularized-PG}
\begin{proof}[Proof of Theorem \ref{thm:unregularized-PG}, gradient play]
From Lemma \ref{lemma:smoothness-unregularized}, $\Phi$ is $\beta$-smooth with $\beta = \frac{6n}{(1-\gamma)^3}$, we have that:
\begin{align*}
    \Phi(\theta^{(t+1)}) -  \Phi(\theta^{(t)}) &\ge \left<\nabla\Phi(\theta^{(t)}), \theta^{(t+1)} - \theta^{(t)}\right> - \frac{\beta}{2}\|\theta^{(t+1)} - \theta^{(t)}\|^2\\
    & = (\eta - \frac{\beta\eta^2}{2})\|\nabla\Phi(\theta^{(t)})\|^2\\
    & \ge \frac{\eta}{2}\|\nabla\Phi(\theta^{(t)})\|^2
\end{align*}
Summing over $t$ we get:
\begin{align*}
    \frac{\phi_{\max} - \phi_{\min}}{1-\gamma} \ge \Phi(\theta^{(T)}) - \Phi(\theta^{(0)}) \ge \frac{\eta}{2}\sum_{t=0}^{T-1}\|\nabla\Phi(\theta^{(t)})\|^2
\end{align*}
From Theorem \ref{lemma:non-uniform-gradient-domination} we have that
\begin{align*}
    \|\nabla\Phi(\theta^{(t)})\| \ge \frac{c}{M\sqrt{\max_i|\cA_i|}}\NEgap(\theta^{(t)})
\end{align*}
Thus
\begin{align*}
    \frac{1}{T}\sum_{t=0}^{T-1}\NEgap(\theta^{(t)})^2 \le \frac{2\max_i|\cA_i|M^2(\phi_{\max}-\phi_{\min})}{(1-\gamma)c^2\eta T}
\end{align*}
which completes the proof.
\end{proof}

\subsection{Proof of Theorem \ref{thm:unregularized-PG} (Natural gradient play part)}\label{apdx:unregularized-NPG}
\begin{lemma}\label{lemma:unregularized-NPG-1}
For $\eta \le \frac{(1-\gamma)^2}{2n(\phi_{\max}-\phi_{\min})}$, running scheme \eqref{eq:unregularized-NPG-2} will guarantee that
\begin{equation*}
    \begin{split}
        \Phi(\theta^{(t+1)}) - \Phi(\theta^{(t)}) \ge \frac{1}{\eta}\sum_{i=1}^n \sum_s d^{(t+1)}(s)\log Z_t^{i,s},
    \end{split}
\end{equation*}
where $Z_t^{i,s}$ is defined by
\begin{equation*}
    Z_t^{i,s}:= \sum_{a_i}\pi_i^{(t)}(a_i|s)\exp{\left(\frac{\eta\overline{A_{i,\phi}^{(t)}}(s,a_i)}{1-\gamma}\right)}.
\end{equation*}
\end{lemma}
\begin{proof}
From performance difference lemma we have that
\begin{align*}
    \Phi(\theta^{(t+1)})-\Phi(\theta^{(t)}) = \frac{1}{1-\gamma}\sum_sd^{(t+1)}(s)\sum_a\left(\pi^{(t+1)}(a|s)-\pi^{(t)}(a|s)\right)A_\phi^{(t)}(s,a).
\end{align*}
We  define
\begin{equation}
    \widetilde{A_{i,\phi}^{(t)}}(s,a_i):= \sum_{a_{-i}}\prod_{j=1}^{i-1}\pi_j^{(t+1)}(a_j|s)\prod_{j=i+1}^n\pi_j^{(t)}(a_j|s)A_\phi^{(t)}(s,a_i,a_{-i}).
\end{equation}
Then
\begin{align*}
    &\quad \Phi(\theta^{(t+1)})-\Phi(\theta^{(t)}) = \frac{1}{1-\gamma}\sum_sd^{(t+1)}(s)\sum_a\left(\pi^{(t+1)}(a|s)-\pi^{(t)}(a|s)\right)A_\phi^{(t)}(s,a)\\
    &= \!\frac{1}{1\!-\!\gamma}\!\sum_s\!d^{(t\!+\!1)}(s)\!\sum_a\!\sum_{i=1}^n\!\!\left(\prod_{j=1}^{i}\!\pi_j^{(t\!+\!1)}\!(a_j|s)\!\!\prod_{j=i+1}^n\!\!\!\pi_j^{(t)}\!(a_j|s)\!-\!\!\prod_{j=1}^{i-1}\!\pi_j^{(t\!+\!1)}\!(a_j|s)\prod_{j=i}^n\!\pi_j^{(t)}\!(a_j|s))\!\!\right)\!\!A_{\!\phi}^{(t)}\!(s,a)\\
    &=\frac{1}{1-\gamma}\sum_sd^{(t+1)}(s)\sum_{i=1}^n\sum_{a_i}\left(\pi_i^{(t+1)}(a_i|s)-\pi_i^{(t)}(a_i|s)\right) \widetilde{A_{i,\phi}^{(t)}}(s,a_i)\\
    &=\frac{1}{1-\gamma}\sum_sd^{(t+1)}(s)\sum_{i=1}^n\sum_{a_i}\left(\pi_i^{(t+1)}(a_i|s)-\pi_i^{(t)}(a_i|s)\right) \overline{A_{i,\phi}^{(t)}}(s,a_i) \\&\quad+ \frac{1}{1-\gamma}\sum_sd^{(t+1)}(s)\sum_{i=1}^n\sum_{a_i}\left(\pi_i^{(t+1)}(a_i|s)-\pi_i^{(t)}(a_i|s)\right) \left(\widetilde{A_{i,\phi}^{(t)}}(s,a_i)-\overline{A_{i,\phi}^{(t)}}(s,a_i)\right)\\
    &=\underbrace{\frac{1}{1-\gamma}\sum_sd^{(t+1)}(s)\sum_{i=1}^n\sum_{a_i}\pi_i^{(t+1)}(a_i|s) \overline{A_{i,\phi}^{(t)}}(s,a_i)}_{\textup{Part A}} \\&\quad+ \underbrace{\frac{1}{1-\gamma}\sum_sd^{(t+1)}(s)\sum_{i=1}^n\sum_{a_i}\left(\pi_i^{(t+1)}(a_i|s)-\pi_i^{(t)}(a_i|s)\right) \left(\widetilde{A_{i,\phi}^{(t)}}(s,a_i)-\overline{A_{i,\phi}^{(t)}}(s,a_i)\right)}_{\textup{Part B}}.
\end{align*}
From scheme \eqref{eq:unregularized-NPG-2},
\begin{equation*}
    \overline{A_{i,\phi}^{(t)}}(s,a_i) = \frac{1-\gamma}{\eta}\left(\log\left(\frac{\pi_i^{(t+1)}(a_i|s)}{\pi_i^{(t)}(a_i|s)}\right)+\log\left(Z_t^{i,s}\right)\right)
\end{equation*}
Substitute this into Part A, we have
\begin{align*}
    \textup{Part A} &= \frac{1}{\eta}\sum_sd^{(t+1)}(s)\sum_{i=1}^n\sum_{a_i}\pi_i^{(t+1)}(a_i|s) \overline{A_{i,\phi}^{(t)}}(s,a_i)\\
    & = \frac{1}{\eta}\sum_sd^{(t+1)}(s)\sum_{i=1}^n\sum_{a_i}\pi_i^{(t+1)}(a_i|s) \left(\log\left(\frac{\pi_i^{(t+1)}(a_i|s)}{\pi_i^{(t)}(a_i|s)}\right)+\log\left(Z_t^{i,s}\right)\right)\\
    &= \frac{1}{\eta}\sum_s\sum_{i=1}^nd^{(t+1)}(s)\textup{KL}(\pi_{i,s}^{(t+1)}||\pi_{i,s}^{(t)})  + \frac{1}{\eta}\sum_s\sum_{i=1}^nd^{(t+1)}(s)\log\left(Z_t^{i,s}\right).
\end{align*}
Further, we have that
\begin{align*}
    &\quad \left|\widetilde{A_{i,\phi}^{(t)}}(s,a_i)-\overline{A_{i,\phi}^{(t)}}(s,a_i)\right| \\
    &= \left|\sum_{a_{-i}}\left(\prod_{j=1}^{i-1}\pi_j^{(t+1)}(a_j|s)-\prod_{j=1}^{i-1}\pi_j^{(t)}(a_j|s)\right)\prod_{j=i+1}^n\pi_j^{(t)}(a_j|s)A_\phi^{(t)}(s,a_i,a_{-i})\right|\\
    &\le \frac{\phi_{\max}-\phi_{\min}}{1-\gamma}\sum_{j=1}^{i-1}\|\pi_{j,s}^{(t+1)}-\pi_{j,s}^{(t)}\|_1\\
    &\le \frac{\phi_{\max}-\phi_{\min}}{1-\gamma}\sum_{j=1}^{n}\|\pi_{j,s}^{(t+1)}-\pi_{j,s}^{(t)}\|_1.
\end{align*}
Thus
\begin{align*}
    |\textup{Part B}|&\le \frac{1}{1-\gamma}\sum_s d^{(t+1)}(s)\sum_{i=1}^n\sum_{a_i}\left|\pi_i^{(t+1)}(a_i|s)-\pi_i^{(t)}(a_i|s)\right| \left|\widetilde{A_{i,\phi}^{(t)}}(s,a_i)-\overline{A_{i,\phi}^{(t)}}(s,a_i)\right|\\
    & \le \frac{\phi_{\max}-\phi_{\min}}{(1-\gamma)^2}\sum_{i=1}^n\sum_s d^{(t+1)}(s) \sum_{a_i}\left|\pi_i^{(t+1)}(a_i|s)-\pi_i^{(t)}(a_i|s)\right|\sum_{j=1}^{n}\|\pi_{j,s}^{(t+1)}-\pi_{j,s}^{(t)}\|_1\\
    &\le  \frac{\phi_{\max}-\phi_{\min}}{(1-\gamma)^2}\sum_sd^{(t+1)}(s)\left(\sum_{i=1}^n\|\pi_{i,s}^{(t+1)}-\pi_{i,s}^{(t)}\|_1\right)^2\\
     &\le \frac{n(\phi_{\max}-\phi_{\min})}{(1-\gamma)^2}\sum_sd^{(t+1)}(s)\sum_{i=1}^n\|\pi_{i,s}^{(t+1)}-\pi_{i,s}^{(t)}\|_1^2\\
     &\le \frac{2n(\phi_{\max}-\phi_{\min})}{(1-\gamma)^2}\sum_sd^{(t+1)}(s)\sum_{i=1}^n\textup{KL}(\pi_{i,s}^{(t+1)}||\pi_{i,s}^{(t)}) \quad \textup{(Pinsker's inequality)}
\end{align*}
Thus, when $\eta \le \frac{(1-\gamma)^2}{2n(\phi_{\max}-\phi_{\min})}$, we have that
\begin{align*}
    &\Phi(\theta^{(t+1)})-\Phi(\theta^{(t)}) = \textup{Part A} + \textup{Part B}\\
    &\ge \left(\frac{1}{\eta}-\frac{2n(\phi_{\max}-\phi_{\min})}{(1-\gamma)^2}\right)\sum_s\sum_{i=1}^nd^{(t+1)}(s)\textup{KL}(\pi_{i,s}^{(t+1)}||\pi_{i,s}^{(t)})  + \frac{1}{\eta}\sum_s\sum_{i=1}^nd^{(t+1)}(s)\log\left(Z_t^{i,s}\right)\\
    &\ge \frac{1}{\eta}\sum_s\sum_{i=1}^nd^{(t+1)}(s)\log\left(Z_t^{i,s}\right),
\end{align*}
which completes the proof.
\end{proof}

\begin{lemma}\label{lemma:unregularized-NPG-2} For $\eta \le (1-\gamma)^2$
\begin{equation*}
    \sum_{i=1}^n \sum_s d^{(t+1)}(s)\log Z_t^{i,s} \ge \frac{c\eta^2}{3M} \NEgap(\theta^{(t)})^2
\end{equation*}
\begin{proof}
From Lemma \ref{lemma:NE-gap-bound} we have that
$    \NEgap(\theta) \le \frac{1}{1-\gamma} \max_i\max_{s, a_i}\overline{A_i^{\theta}}(s, a_i).
$ On the other hand,
\begin{align*}
    &\quad Z_t^{i,s}= \sum_{a_i}\pi_i^{(t)}(a_i|s)\exp{\left(\frac{\eta\overline{A_i^{(t)}}(s,a_i)}{1-\gamma}\right)}\\
    &= \!\!\!\!\!\!\!\!\!\!\!\!\sum_{a_i \notin \argmax_{a_i}\! \!\overline{Q_i^{(t)}}(s,a_i)}\!\!\!\!\!\!\!\!\!\!\!\!\pi_i^{(t)}(a_i|s)\exp{\left(\frac{\eta\overline{A_i^{(t)}}(s,a_i)}{1-\gamma}\right)} +\!\!\!\!\!\!\!\! \sum_{a_i \in \argmax_{a_i} \!\overline{Q_i^{(t)}}(s,a_i)}\!\!\!\!\!\!\!\!\!\!\!\!\pi_i^{(t)}(a_i|s)\exp{\left(\frac{\eta\max_{a_i}\overline{A_i^{(t)}}(s,a_i)}{1-\gamma}\right)}\\
    &\ge \sum_{a_i \notin \argmax_{a_i} \!\overline{Q_i^{(t)}}(s,a_i)}\pi_i^{(t)}(a_i|s)\left(1 + \frac{\eta\overline{A_i^{(t)}}(s,a_i)}{1-\gamma}\right) \\
    &+ \sum_{a_i \in \argmax_{a_i} \!\overline{Q_i^{(t)}}(s,a_i)}\pi_i^{(t)}(a_i|s)\left(1 + \frac{\eta\max_{a_i}\overline{A_i^{(t)}}(s,a_i)}{1-\gamma} + \frac{1}{2}\left(\frac{\eta\max_{a_i}\overline{A_i^{(t)}}(s,a_i)}{1-\gamma}\right)^2\right)\\
    &= \sum_{a_i}\pi_i^{(t)}(a_i|s) + \frac{\eta}{1-\gamma}\sum_{a_i}\pi_i^{(t)}(a_i|s)\overline{A_i^{(t)}}(s,a_i) + \frac{1}{2} \!\!\!\!\!\!\!\!\!\!\!\!\sum_{a_i \in \argmax_{a_i} \!\overline{Q_i^{(t)}}(s,a_i)}\!\!\!\!\!\!\!\!\!\!\!\!\pi_i^{(t)}(a_i|s)\left(\frac{\eta\max_{a_i}\overline{A_i^{(t)}}(s,a_i)}{1-\gamma}\right)^2\\
    &= 1+ \frac{1}{2} \sum_{a_i \in \argmax_{a_i} \!\overline{Q_i^{(t)}}(s,a_i)}\pi_i^{(t)}(a_i|s)\left(\frac{\eta\max_{a_i}\overline{A_i^{(t)}}(s,a_i)}{1-\gamma}\right)^2\\
    &\ge 1 + \frac{c}{2}\left(\frac{\eta\max_{a_i}\overline{A_i^{(t)}}(s,a_i)}{1-\gamma}\right)^2.
\end{align*}
Thus
\begin{align*}
    \log(Z_t^{i,s}) &\ge \log\left(1 + \frac{c}{2}\left(\frac{\eta\max_{a_i}\overline{A_i^{(t)}}(s,a_i)}{1-\gamma}\right)^2\right).
\end{align*}
Because when $\eta \le (1-\gamma)^2$, we have $\frac{c}{2}\left(\frac{\eta\max_{a_i}\overline{A_i^{(t)}}(s,a_i)}{1-\gamma}\right)^2 \le \frac{1}{2}$, and that
$$\log(1+x)\ge \frac{2}{3}x,~~\textup{for } 0\le x\le \frac{1}{2}, $$
thus
\begin{align*}
    \log(Z_t^{i,s}) &\ge \log\left(1 + \frac{c}{2}\left(\frac{\eta\max_{a_i}\overline{A_i^{(t)}}(s,a_i)}{1-\gamma}\right)^2\right)\\
    &\ge \frac{c}{3}\left(\frac{\eta\max_{a_i}\overline{A_i^{(t)}}(s,a_i)}{1-\gamma}\right)^2.
\end{align*}
Thus
\begin{align*}
    \sum_{i=1}^n \sum_s d^{(t+1)}(s)\log Z_t^{i,s} &\ge  \frac{c}{3}\sum_{i=1}^n \sum_s d^{(t+1)}(s) \left(\frac{\eta\max_{a_i}\overline{A_i^{(t)}}(s,a_i)}{1-\gamma}\right)^2\\
    &\ge\frac{c\eta^2}{3M(1-\gamma)^2}\max_i\max_s\max_{a_i}\overline{A_i^{(t)}}(s,a_i)^2\\
    &\ge \frac{c\eta^2}{3M}\NEgap(\theta^{(t)})^2. \qedhere
\end{align*}
\end{proof}
\end{lemma}
We are now ready to prove the bound for natural gradient play in Theorem \ref{thm:unregularized-PG}.
\begin{proof}[Proof of Theorem \ref{thm:unregularized-PG}, natural gradient play]
Combining Lemma \ref{lemma:unregularized-NPG-1} and \ref{lemma:unregularized-NPG-2} we have
\begin{align*}
    \Phi(\theta^{(t+1)}) - \Phi(\theta^{(t)}) \ge \frac{1}{\eta}\sum_{i=1}^n \sum_s d^{(t+1)}(s)\log Z_t^{i,s}\\
    \ge \frac{c\eta}{3M}\NEgap(\theta^{(t)})^2
\end{align*}
Summing over $t$ we have
\begin{align*}
    \frac{\phi_{\max}-\phi_{\min}}{1-\gamma}\ge \Phi(\theta^{(T)}) - \Phi(\theta^{(0)}) \ge \frac{c\eta}{3M}\sum_{t=0}^{T-1}\NEgap(\theta^{(t)})^2,
\end{align*}
thus
\begin{align*}
    \frac{\sum_{t=0}^{T-1}\NEgap(\theta^{(t)})^2}{T}\le \frac{3M(\phi_{\max}-\phi_{\min})}{(1-\gamma)c\eta T},
\end{align*}
which completes the proof.
\end{proof}
\section{Proof for $\log$-barrier regularization}
\subsection{Proof of Theorem \ref{thm:log-barrier-PG}}\label{apdx:log-barrier-pg}
We start with the following lemma:
\begin{lemma}\label{lemma:log-barrier-pg}
Suppose $\theta$ is such that $ \|\nabla_{\theta_i}\widetilde{J_i}(\theta)\|_2 \le \lambda$,
then $\NEgap_i(\theta)\le \lambda M|\cA_i|$,
where $M$ is defined as in Assumption \ref{assump:like-ergodicity}.
\begin{proof}
 From $\|\nabla_{\theta_i}\widetilde{J}_i(\theta)\|_2\le \frac{\lambda}{2}$ we have that
\begin{align*}
    \frac{\partial \widetilde{J}_i(\theta)}{\partial\theta_{s,a_i}} &= \frac{1}{1-\gamma}d_\theta(s)\pi_{\theta_i}(a_i|s)\overline{A_i^\theta}(s,a_i) + \lambda - \lambda|\cA_i|\pi_{\theta_i}(a_i|s)\\
    &= \pi_{\theta_i}(a_i|s)\left(\frac{1}{1-\gamma}d_\theta(s)\overline{A_i^{\theta}}(s,a_i) - \lambda|\cA_i|\right) + \lambda \le \lambda\\
\Longrightarrow &\quad \pi_{\theta_i}(a_i|s)\left(\frac{1}{1-\gamma}d_\theta(s)\overline{A_i^{\theta}}(s,a_i) - \lambda|\cA_i|\right) \le 0\\
\Longrightarrow &\quad \frac{1}{1-\gamma}d_\theta(s)\overline{A_i^{\theta}}(s,a_i) - \lambda|\cA_i| \le 0\\
\Longrightarrow &\quad \overline{A_i^{\theta}}(s,a_i)\le \frac{\lambda|\cA_i|(1-\gamma)}{d_\theta(s)} \le \lambda|\cA_i|(1-\gamma)M
\end{align*}
Thus,
\begin{align*}
    \NEgap_i(\theta) = \sup_{\theta_i^*}J_i(\theta_i^*,\theta_{-i}) - J_i(\theta_i, \theta_{-i}) &= \frac{1}{1-\gamma} \sum_{s,a_i}d_{\theta^*}(s)\pi_{\theta_i^*}(a_i|s)\overline{A_i^\theta}(s,a_i)\\
    &\le \frac{1}{1-\gamma} \sum_{s,a_i}d_{\theta^*}(s)\max_{s,a_i}\overline{A_i^\theta}(s,a_i)\\
    &\le \frac{1}{1-\gamma} \sum_{s,a_i}d_{\theta^*}(s)\lambda|\cA_i|(1-\gamma)M\\
    &\le \lambda|\cA_i|M. \qedhere
\end{align*}
\end{proof}

\end{lemma}
Lemma \ref{lemma:log-barrier-pg} implies that any policy with gradient norm smaller than $\lambda$ is also a $\lambda M\max_i|\cA_i|$-NE. Thus by properly choosing $\lambda$, agents can find a $\epsilon$-NE by running gradient play.

We now prove Theorem \ref{thm:log-barrier-PG}.

\begin{proof}[Proof of Theorem \ref{thm:log-barrier-PG}]
From Lemma \ref{lemma:smoothness-log-barrier}, $\widetilde\Phi$ is $\beta$-smooth with $\beta = \frac{6n}{(1-\gamma)^3} + 2\lambda\max_i|\cA_i|$, we have that:
\begin{align*}
    \widetilde\Phi(\theta^{(t+1)}) -  \widetilde\Phi(\theta^{(t)}) &\ge \left<\nabla\widetilde\Phi(\theta^{(t)}), \theta^{(t+1)} - \theta^{(t)}\right> - \frac{\beta}{2}\|\theta^{(t+1)} - \theta^{(t)}\|^2\\
    & = (\eta - \frac{\beta\eta^2}{2})\|\nabla\Phi(\theta^{(t)})\|^2\\
    & \ge \frac{\eta}{2}\|\nabla\Phi(\theta^{(t)})\|^2
\end{align*}
For $\theta^{(0)} = \textbf{0}$, summing over $t$ we get:
\begin{align*}
    \frac{\phi_{\max} - \phi_{\min}}{1-\gamma} \ge \widetilde\Phi(\theta^{(T)}) - \widetilde\Phi(\theta^{(0)}) \ge \frac{\eta}{2}\sum_{t=0}^{T-1}\|\nabla\widetilde\Phi(\theta^{(t)})\|^2
\end{align*}
Thus,
\begin{equation*}
    \min_{0\le t\le T-1}\|\nabla\widetilde\Phi(\theta^{(t)})\| \le \frac{2(\phi_{\max} - \phi_{\min})}{(1-\gamma)\eta T}.
\end{equation*}
Thus for
\begin{align*}
    T &\ge \frac{2(\phi_{\max} - \phi_{\min})}{(1-\gamma)\eta \lambda^2}\\
    &= \frac{2\max_i|\cA_i|^2(\phi_{\max} - \phi_{\min})M^2}{(1-\gamma)\eta \epsilon^2},
\end{align*}
it can be guaranteed that
\begin{align*}
    \min_{0\le t\le T-1}\|\nabla\widetilde\Phi(\theta^{(t)})\| \le \lambda = \frac{\epsilon}{\max_i|\cA_i|M}.
\end{align*}
Then applying Lemma \ref{lemma:log-barrier-pg} completes the proof.
\end{proof}

\subsection{Proof of Theorem \ref{thm:log-barrier-NPG}}\label{apdx:log-barrier-NPG}
For notational simplicity, we define the following variables:
\begin{align*}
    f_i^{(t)}(s,a_i)&:= \frac{1}{1-\gamma}\overline{A_i^{(t)}}(s,a_i)+\frac{\lambda}{d^{(t)}(s)\pi_i^{(t)}(a_i|s)} - \frac{\lambda|\cA_i|}{d^{(t)}(s)}\\
    Z_t^{i,s}&:= \sum_{a_i}\pi_i^{(t)}(a_i|s)\exp{\left(\eta f^{(t)}(s,a_i)\right)}\\
    \Delta_i^{(t)}(s,a_i) &:= \frac{\pi_i^{(t+1)}(a_i|s)}{\pi_i^{(t)}(a_i|s)} - 1
\end{align*}
\begin{lemma}
\begin{align*}
    \sum_{a_i}\pi_i^{(t)}(a_i|s) f_i^{(t)}(s,a_i) &= 0\\
    \sum_{a_i}\pi_i^{(t)}(a_i|s) \Delta_i^{(t)}(s,a_i) &= 0\\
    Z_t^{i,s}&\ge1
\end{align*}
\begin{proof}
From the definition of $f_i^{(t)}(s,a_i), \Delta_i^{(t)}(s,a_i)$,
\begin{align*}
    &\quad \sum_{a_i}\pi_i^{(t)}(a_i|s) f_i^{(t)}(s,a_i)\\ &= \frac{1}{1\!-\!\gamma}\sum_{a_i}\pi_i^{(t)}(a_i|s)\overline{A_i^{(t)}}(s,a_i) \!+ \!\lambda \sum_{a_i}\pi_i^{(t)}(a_i|s)\frac{1}{d^{(t)}(s)\pi^{(t)}(a_i|s)} \!-\!\frac{\lambda|\cA_i|}{d^{(t)}(s)}\sum_{a_i}\pi_i^{(t)}(a_i|s)\\
    &=\frac{1}{d^{(t)}(s)}(\lambda|\cA_i|-\lambda|\cA_i|)=0\\
     &\quad \sum_{a_i}\pi_i^{(t)}(a_i|s) \Delta_i^{(t)}(s,a_i) =\sum_{a_i}\pi_i^{(t)}(a_i|s) \left(\frac{\pi_i^{(t+1)}(a_i|s)}{\pi_i^{(t)}(a_i|s)} - 1\right)\\
    &= \sum_{a_i}\pi_i^{(t+1)}(a_i|s) - \sum_{a_i}\pi_i^{(t)}(a_i|s) = 1-1=0
\end{align*}
Using the fact that $e^x \ge 1+x$,
\begin{align*}
   Z_t^{i,s}&=\sum_{a_i}\pi_i^{(t)}(a_i|s)\exp{\left(\eta f^{(t)}(s,a_i)\right)}\\
   &\ge \sum_{a_i}\pi_i^{(t)}(a_i|s)\left(1 + \eta f^{(t)}(s,a_i)\right)\\
   &\ge \sum_{a_i}\pi_i^{(t)}(a_i|s) + \eta \sum_{a_i}\pi_i^{(t)}(a_i|s) f_i^{(t)}(s,a_i)\ge 1. \qedhere
\end{align*}
\end{proof}
\end{lemma}
\begin{lemma}\label{lemma:NPG-bounded-probability}
For $\eta \le \frac{1}{15\left(\frac{1}{(1-\gamma)^2} + \lambda|\cA_i|M \right)}, \theta^{(0)} = \textbf{0}$, running scheme \eqref{eq:log-barrier-NPG-2} will guarantee that
\begin{equation*}
    \pi_i^{(t)}(a_i|s)\ge\frac{\lambda}{4\left(\lambda |\cA_i|M + \frac{1}{(1-\gamma)^2}\right)}.
\end{equation*}
\begin{proof}
We will prove by induction. For  $\theta^{0} = \textbf{0}$ apparently $\pi^{(0)}$ satisfies the lower bound. Suppose that
\begin{equation*}
    \pi_i^{(t)}(a_i|s)\ge\frac{\lambda}{4\left(\lambda |\cA_i|M + \frac{1}{(1-\gamma)^2}\right)},
\end{equation*}
then from the definition of $f_i^{(t)}(s,a)$ we have that
\begin{equation*}
    -\frac{1}{(1-\gamma)^2} - \lambda|\cA_i|M\le f_i^{(t)}(s,a_i)\le 5\left(\frac{1}{(1-\gamma)^2} + \lambda|\cA_i|M\right)
\end{equation*}
Thus
\begin{equation*}
    -\frac{1}{15}\le\eta f_i^{(t)}(s,a_i) \le \frac{1}{3},
\end{equation*}
which leads to the fact that
\begin{align}
    Z_t^{i,s} &= \sum_{a_i}\pi_i^{(t)}(a_i|s)\exp{\left(\eta f_i^{(t)}(s,a_i)\right)}\notag\\
    &\le \sum_{a_i}\pi_i^{(t)}(a_i|s) \left(1+\left(\eta f_i^{(t)}(s,a_i)\right)+\left(\eta f_i^{(t)}(s,a_i)\right)^2\right) \\
    &\qquad \qquad\qquad\qquad\qquad(e^x \le 1+x+x^2, \textup{for} -\frac{1}{15}\le x \le \frac{1}{3})\notag\\
    &= 1 + \sum_{a_i}\pi_i^{(t)}(a_i|s) \left(\eta f_i^{(t)}(s,a_i)\right)^2\label{eq:ineq-log-barrier-1}\\
    &\le 1 + \frac{1}{3^2} = \frac{10}{9}\notag.
\end{align}
Thus we have that
\begin{equation*}
     \frac{\pi_i^{(t+1)}(a_i|s)}{\pi_i^{(t)}(a_i|s)} = \frac{\exp{\left(\eta f_i^{(t)}(s,a_i)\right)}}{Z_t^{i,s}}\ge \frac{1+\eta f^{(t)}(s,a_i)}{Z_t^{i,s}}\ge \frac{1-\frac{1}{15}}{\frac{10}{9}} = \frac{21}{25}. 
\end{equation*}
Thus, for $a_i$ such that $\pi_i^{(t)}(a_i|s)\ge\frac{\lambda}{3\left(\lambda |\cA_i|M + \frac{1}{(1-\gamma)^2}\right)}$, we have
\begin{equation*}
    \pi_i^{(t+1)}(a_i|s) \ge \frac{21}{25}\frac{\lambda}{3\left(\lambda |\cA_i|M + \frac{1}{(1-\gamma)^2}\right)} \ge \frac{\lambda}{4\left(\lambda |\cA_i|M + \frac{1}{(1-\gamma)^2}\right)}.
\end{equation*}
On the other hand, for $a_i$ such that $\pi_i^{(t)}(a_i|s)<\frac{\lambda}{3\left(\lambda |\cA_i|M + \frac{1}{(1-\gamma)^2}\right)}$, we have
\begin{align*}
    f_i^{(t)}(s,a_i) \ge -\frac{1}{(1-\gamma)^2} - \lambda|\cA_i|M + 3\left(\frac{1}{(1-\gamma)^2} + \lambda|\cA_i|M\right) = 2\left(\frac{1}{(1-\gamma)^2} + \lambda|\cA_i|M\right),
\end{align*}
From inequality \eqref{eq:ineq-log-barrier-1} we have that
\begin{align*}
    Z_{i,t} &\le 1 + \eta^2 \sum_{a_i}\pi_i^{(t)}(a_i|s) f_i^{(t)}(s,a_i)^2\\
    & \le 1 + 25\eta^2 \left(\frac{1}{(1-\gamma)^2} + \lambda|\cA_i|M\right)^2\\
    &\le 1 + \frac{5}{3}\eta \left(\frac{1}{(1-\gamma)^2} + \lambda|\cA_i|M\right)
\end{align*}
Thus
\begin{align*}
    \frac{\pi_i^{(t+1)}(a_i|s)}{\pi_i^{(t)}(a_i|s)} = \frac{\exp{\left(\eta f_i^{(t)}(s,a_i)\right)}}{Z_t^{i,s}} \ge\frac{1+\left(\eta f_i^{(t)}(s,a_i)\right)}{Z_t^{i,s}} \ge \frac{1+2\eta \left(\frac{1}{(1-\gamma)^2} + \lambda|\cA_i|M\right)}{1 + \frac{5}{3}\eta\left(\frac{1}{(1-\gamma)^2} + \lambda|\cA_i|M\right)}\ge 1,
\end{align*}
then according to the induction assumption, we have
\begin{equation*}
    \pi_i^{(t+1)}(a_i|s)\ge\frac{\lambda}{4\left(\lambda |\cA_i|M + \frac{1}{(1-\gamma)^2}\right)},
\end{equation*}
which completes the proof of the lemma.
\end{proof}
\end{lemma}
\begin{corollary}\label{coro:NPG-log-barrier-bounded-quantities}
Under the condition of Lemma \ref{lemma:NPG-bounded-probability}, running \eqref{eq:log-barrier-NPG-2} will guarantee that
\begin{align*}
    -\frac{1}{15}\le\eta f_i^{(t)}(s,a_i) \le \frac{1}{3},\quad Z_t^{i,s} \le \frac{10}{9},\quad
    -\frac{1}{5}\le \Delta_i^{(t)}(s,a_i) \le \frac{1}{2},
\end{align*}
\begin{proof}
The first two inequalities are proved in the proof of Lemma \ref{lemma:NPG-bounded-probability}, we only need to show $-\frac{1}{5}\le \Delta_i^{(t)}(s,a_i) \le \frac{1}{2}$. In the proof of Lemma \ref{lemma:NPG-bounded-probability}, we have already shown that
\begin{equation*}
     \frac{\pi_i^{(t+1)}(a_i|s)}{\pi_i^{(t)}(a_i|s)} = \frac{\exp{\left(\eta f_i^{(t)}(s,a_i)\right)}}{Z_t^{i,s}}\ge \frac{1+\eta f^{(t)}(s,a_i)}{Z_t^{i,s}}\ge \frac{1-\frac{1}{15}}{\frac{10}{9}} = \frac{21}{25} \ge \frac{4}{5}, 
\end{equation*}
thus
\begin{equation*}
    \Delta_i^{(t)}(s,a_i)\ge \frac{4}{5}-1 = -\frac{1}{5}.
\end{equation*}
On the other hand,
\begin{equation*}
    \frac{\pi_i^{(t+1)}(a_i|s)}{\pi_i^{(t)}(a_i|s)} = \frac{\exp{\left(\eta f_i^{(t)}(s,a_i)\right)}}{Z_t^{i,s}} \le \exp{\left(\eta f_i^{(t)}(s,a_i)\right)} \le \exp{(\frac{1}{3})} \le \frac{3}{2},
\end{equation*}
thus
\begin{equation*}
    \Delta_i^{(t)}(s,a_i)\le \frac{3}{2}-1 = \frac{1}{2},
\end{equation*}
which completes the proof of the corollary.
\end{proof}
\end{corollary}
\begin{lemma}\label{lemma:NPG-log-barrier-sufficient-ascent}
\begin{equation*}
    \widetilde{\Phi}(\theta^{(t\!+\!1)}\!) \!-\! \widetilde{\Phi}(\theta^{(t)}\!) \!\ge\!\! \left(\!\frac{1}{2\eta}\!-\!4\lambda\max_{i}\!|\cA_i|M^2\!-\!\frac{4M}{(1\!-\!\gamma)^2}\!-\!\frac{3nM}{(1\!-\!\gamma)^3}\!\right)\!\sum_{i=1}^n\!\sum_{s,a_i}\!d^{(t)}\!(s)\pi_i^{(t)}\!(a_i|s)\Delta_i^{(t)}\!(s,a_i)^2
\end{equation*}
\begin{proof}
Let $\widetilde{\theta}^{i,(t)}$ be defined as:
$$\widetilde{\theta}^{i,(t)}:= \left(\theta_1^{(t)},\dots,\theta_{i-1}^{(t)},\theta_i^{(t+1)},\dots,\theta_n^{(t+1)}\right).$$
Then we have that
\begin{align*}
    \Phi(\theta^{(t+1)}) - \Phi(\theta^{(t)}) &= \sum_{i=1}^n\Phi(\widetilde{\theta}^{i,(t)}) - \Phi(\widetilde{\theta}^{i+1,(t)})\\
    &=\sum_{i=1}^n J_i(\widetilde{\theta}^{i,(t)}) - J_i(\widetilde{\theta}^{i+1,(t)})\\
    & =\frac{1}{1-\gamma}\sum_{i=1}^n\sum_s d_{\widetilde{\theta}^{i,(t)}}(s)\sum_{a_i}\pi_i^{(t+1)}(a_i|s)\overline{A_i^{\widetilde{\theta}^{i+1,(t)}}}(s,a_i)\quad  (\textup{Lemma \ref{lemma:performance-difference-lemma}})\\
    &=\frac{1}{1-\gamma}\sum_{i=1}^n\sum_s d_{\widetilde{\theta}^{i,(t)}}(s)\sum_{a_i}\left(\pi_i^{(t+1)}(a_i|s)-\pi_i^{(t)}(a_i|s)\right)\overline{Q_i^{\widetilde{\theta}^{i+1,(t)}}}(s,a_i).
\end{align*}
Thus
\begin{align*}
    &\quad\widetilde{\Phi}(\theta^{(t+1)})-\widetilde{\Phi}(\theta^{(t)})=\Phi(\theta^{(t+1)}) - \Phi(\theta^{(t)}) + \lambda\sum_i\sum_{s,a_i}\log\left(\frac{\pi_i^{(t+1)}(a_i|s)}{\pi_i^{(t)}(a_i|s)}\right)\\
    &= \frac{1}{1\!-\!\gamma}\!\sum_{i=1}^n\!\sum_s d_{\widetilde{\theta}^{i\!,(t)}}\!(s)\!\sum_{a_i}\!\left(\pi_i^{(t\!+\!1)}\!(a_i|s)-\pi_i^{(t)}\!(a_i|s)\right)\overline{Q_i^{\widetilde{\theta}^{i\!+\!1\!,(t)}}}\!(s,a_i) \!+\! \lambda\!\sum_i\!\sum_{s,a_i}\!\log\left(1\!+\!\Delta^{(t)}\!(s,a_i)\right)\\
    &= \frac{1}{1-\gamma}\sum_{i=1}^n\sum_s d^{(t)}(s)\sum_{a_i}\left(\pi_i^{(t+1)}(a_i|s)-\pi_i^{(t)}(a_i|s)\right)\overline{Q_i^{(t)}}(s,a_i)\\
    &\quad + \frac{1}{1-\gamma}\sum_{i=1}^n\sum_s d^{(t)}(s)\sum_{a_i}\left(\pi_i^{(t+1)}(a_i|s)-\pi_i^{(t)}(a_i|s)\right)\left(\overline{Q_i^{\widetilde{\theta}^{i+1,(t)}}}(s,a_i)-\overline{Q_i^{(t)}}(s,a_i)\right)\\
    &\quad + \frac{1}{1-\gamma}\sum_{i=1}^n\sum_s \left(d_{\widetilde{\theta}^{i,(t)}}(s)-d^{(t)}(s)\right)\sum_{a_i}\left(\pi_i^{(t+1)}(a_i|s)-\pi_i^{(t)}(a_i|s)\right)\overline{Q_i^{\widetilde{\theta}^{i+1,(t)}}}(s,a_i)\\
    &\quad + \lambda\sum_i\sum_{s,a_i}\log\left(1+\Delta^{(t)}(s,a_i)\right)\\
    &=\underbrace{\frac{1}{1-\gamma}\sum_{i=1}^n\sum_s d^{(t)}(s)\sum_{a_i}\left(\pi_i^{(t+1)}(a_i|s)-\pi_i^{(t)}(a_i|s)\right)\overline{Q_i^{(t)}}(s,a_i)+\lambda\Delta_i^{(t)}(s,a_i)}_{\textup{Part A}}\\
    &\quad + \underbrace{\lambda\sum_i\sum_{s,a_i}\log\left(1+\Delta^{(t)}(s,a_i)\right)-\Delta_i^{(t)}(s,a_i)}_{\textup{Part B}}\\
    &\quad + \underbrace{\frac{1}{1-\gamma}\sum_{i=1}^n\sum_s d^{(t)}(s)\sum_{a_i}\left(\pi_i^{(t+1)}(a_i|s)-\pi_i^{(t)}(a_i|s)\right)\left(\overline{Q_i^{\widetilde{\theta}^{i+1,(t)}}}(s,a_i)-\overline{Q_i^{(t)}}(s,a_i)\right)}_{\textup{Part C}}\\
    &\quad + \underbrace{\frac{1}{1-\gamma}\sum_{i=1}^n\sum_s \left(d_{\widetilde{\theta}^{i,(t)}}(s)-d^{(t)}(s)\right)\sum_{a_i}\left(\pi_i^{(t+1)}(a_i|s)-\pi_i^{(t)}(a_i|s)\right)\overline{Q_i^{\widetilde{\theta}^{i+1,(t)}}}(s,a_i)}_{\textup{Part D}}.
\end{align*}

We will now bound each part separately. We first get a lower bound for part A:
\begin{align*}
    \textup{Part A} &= \sum_{i=1}^n\sum_{s,a_i}d^{(t)}(s)\left[\left(\pi_i^{(t+1)}(a_i|s)-\pi_i^{(t)}(a_i|s)\right)\left(\frac{1}{1-\gamma}\overline{A_i^{(t)}}(s,a_i)\right) + \frac{\lambda}{d^{(t)}(s)}\Delta_i^{(t)}(s,a_i)\right]\\
    &=\sum_{i=1}^n\sum_{s,a_i}d^{(t)}(s)\left[\pi_i^{(t)}(a_i|s)\Delta_i^{(t)}(s,a_i)\left(\frac{1}{1-\gamma}\overline{A_i^{(t)}}(s,a_i)\right) + \frac{\lambda}{d^{(t)}(s)}\Delta_i^{(t)}(s,a_i)\right]\\
    &=\sum_{i=1}^n\sum_{s,a_i}d^{(t)}(s)\left[\pi_i^{(t)}(a_i|s)\Delta_i^{(t)}(s,a_i)\left(\frac{1}{1-\gamma}\overline{A_i^{(t)}}(s,a_i) - \frac{\lambda|\cA_i|}{d^{(t)}(s)}\right) + \frac{\lambda}{d^{(t)}(s)}\Delta_i^{(t)}(s,a_i)\right]\\
    &=\sum_{i=1}^n\sum_{s,a_i}d^{(t)}(s)\pi_i^{(t)}(a_i|s)\Delta_i^{(t)}(s,a_i)\left(\frac{1}{1-\gamma}\overline{A_i^{(t)}}(s,a_i)  + \frac{\lambda}{d^{(t)}(s)\pi_i^{(t)}(a_i|s)} - \frac{\lambda|\cA_i|}{d^{(t)}(s)}\right)\\
    &=\sum_{i=1}^n\sum_{s,a_i}d^{(t)}(s)\pi_i^{(t)}(a_i|s)\Delta_i^{(t)}(s,a_i)f_i^{(t)}(s,a_i)\\
    &=\sum_{i=1}^n\sum_{s,a_i}d^{(t)}(s)\pi_i^{(t)}(a_i|s)\Delta_i^{(t)}(s,a_i)\frac{1}{\eta}\left(\log\left(\frac{\pi_i^{(t+1)}(a_i|s)}{\pi_i^{(t)}(a_i|s)}\right)+\log(Z_t^{i,s})\right)\\
    &=\frac{1}{\eta}\sum_{i=1}^n\sum_{s,a_i}d^{(t)}(s)\pi_i^{(t)}(a_i|s)\Delta_i^{(t)}(s,a_i)\log\left(1+\Delta_i^{(t)}(s,a_i)\right) \\
    &\qquad+ \frac{1}{\eta}\sum_{i=1}^n\sum_{s}d^{(t)}(s)\log(Z_t^{i,s})\sum_{a_i}\pi_i^{(t)}(a_i|s)\Delta_i^{(t)}(s,a_i)\\
    &=\frac{1}{\eta}\sum_{i=1}^n\sum_{s,a_i}d^{(t)}(s)\pi_i^{(t)}(a_i|s)\Delta_i^{(t)}(s,a_i)\log\left(1+\Delta_i^{(t)}(s,a_i)\right)\\
    &=\frac{1}{\eta}\sum_{i=1}^n\sum_{s,a_i}d^{(t)}(s)\pi_i^{(t)}(a_i|s)\left|\Delta_i^{(t)}(s,a_i)\right|\left|\log\left(1+\Delta_i^{(t)}(s,a_i)\right)\right|
\end{align*}
From the boundedness of $\Delta_i^{(t)}(s,a_i)$ in Corollary \ref{coro:NPG-log-barrier-bounded-quantities}, we have that
\begin{equation*}
    \left|\log\left(1+\Delta_i^{(t)}(s,a_i)\right)\right| \ge \frac{1}{2}\left|\Delta_i^{(t)}(s,a_i)\right|
\end{equation*}
Substitute this into the above inequalities, we get
\begin{align*}
     \textup{Part A} \ge \frac{1}{2\eta}\sum_{i=1}^n\sum_{s,a_i}d^{(t)}(s)\pi_i^{(t)}(a_i|s)\Delta_i^{(t)}(s,a_i)^2
\end{align*}
Now we will give a lower bound for part B. Similarly, from the boundedness of $\Delta_i^{(t)}(s,a_i)$ in Corollary \ref{coro:NPG-log-barrier-bounded-quantities}, we have that
\begin{equation*}
    \log\left(1+\Delta_i^{(t)}(s,a_i)\right) - \Delta_i^{(t)}(s,a_i) \ge -\Delta_i^{(t)}(s,a_i)^2.
\end{equation*}
Thus
\begin{align*}
    \textup{Part B} = \lambda\sum_i\sum_{s,a_i}\log\left(1+\Delta^{(t)}(s,a_i)\right)-\Delta_i^{(t)}(s,a_i) \ge -\lambda\sum_i\sum_{s,a_i}\Delta_i^{(t)}(s,a_i)^2.
\end{align*}
Additionally, using Lemma \ref{lemma:NPG-bounded-probability},
\begin{align*}
    \textup{Part B} &\ge -\lambda\sum_i\sum_{s,a_i}\Delta_i^{(t)}(s,a_i)^2\\
    &\ge -4\left(\lambda\max_i|\cA_i|M+\frac{1}{(1-\gamma)^2}\right)\sum_i\sum_{s,a_i}\pi_i^{(t)}(a_i|s)\Delta_i^{(t)}(s,a_i)^2\\
    &\ge -4M\left(\lambda\max_i|\cA_i|M+\frac{1}{(1-\gamma)^2}\right)\sum_i\sum_{s,a_i}d^{(t)}(s)\pi_i^{(t)}(a_i|s)\Delta_i^{(t)}(s,a_i)^2.
\end{align*}

We will now move on to bound the absolute value of Part C.
\begin{align*}
    \left|\textup{Part C}\right| \le \frac{1}{1-\gamma}\sum_{i=1}^n\sum_sd^{(t)}(s)\sum_{a_i}\left|\pi_i^{(t+1)}(a_i|s)-\pi_i^{(t)}(a_i|s)\right|\left|\overline{Q_i^{\widetilde{\theta}^{i+1,(t)}}}(s,a_i) - \overline{Q_i^{(t)}}(s,a_i)\right|.
\end{align*}
Since
\begin{align*}
    & \left|\overline{Q_i^{\widetilde{\theta}^{i+1,(t)}}}(s,a_i) - \overline{Q_i^{(t)}}(s,a_i)\right|\le \sum_{a_{-i}}\pi_{\widetilde{\theta}_{-i}^{i+1,(t)}}(a_{-i}|s)\left|Q^{\widetilde{\theta}^{i+1,(t)}}(s,a_i,a_{-i}) - Q^{(t)}(s,a_i,a_{-i})\right|\\&\quad + \sum_{a_{-i}}\left|\pi_{\widetilde{\theta}_{-i}^{i+1,(t)}}(a_{-i}|s) - \pi^{(t)}_{-i}(a_{-i}|s)\right|\left|Q^{(t)}(s,a_i)\right|\\
    &\le \max_a \left|Q^{\widetilde{\theta}^{i+1,(t)}}(s,a) - Q^{(t)}(s,a)\right| + \frac{1}{1-\gamma}\sum_{j=1}^n\|\pi_{j,s}^{(t+1)}-\pi_{j,s}^{(t)}\|_1.
\end{align*}
From Lemma \ref{lemma:NPG-log-barrier-auxiliary-1},
\begin{align*}
    \max_a \left|Q^{\widetilde{\theta}^{i+1,(t)}}(s,a) - Q^{(t)}(s,a)\right| &\le \frac{1}{(1-\gamma)^2}\max_s \|\pi_{\widetilde{\theta}^{i+1,(t)}_s}-\pi_{\theta^{(t)}_s}\|_1\\
    &\le \frac{1}{(1-\gamma)^2}\max_s \sum_{j=1}^n\|\pi_{j,s}^{(t+1)}-\pi_{j,s}^{(t)}\|_1
\end{align*}
Thus we have that
\begin{align*}
    \left|\overline{Q_i^{\widetilde{\theta}^{i+1,(t)}}}(s,a_i) - \overline{Q_i^{(t)}}(s,a_i)\right|&\le\max_a \left|Q^{\widetilde{\theta}^{i+1,(t)}}(s,a) - Q^{(t)}(s,a)\right| + \frac{1}{1-\gamma}\sum_{j=1}^n\|\pi_{j,s}^{(t+1)}-\pi_{j,s}^{(t)}\|_1\\&\le \frac{2}{(1-\gamma)^2}\max_s \sum_{j=1}^n\|\pi_{j,s}^{(t+1)}-\pi_{j,s}^{(t)}\|_1.
\end{align*}
Thus
\begin{align*}
    |\textup{Part C}|&\le \frac{1}{1-\gamma}\sum_{i=1}^n\sum_sd^{(t)}(s)\sum_{a_i}\left|\pi_i^{(t+1)}(a_i|s)-\pi_i^{(t)}(a_i|s)\right|\frac{2}{(1-\gamma)^2}\max_s \sum_{j=1}^n\|\pi_{j,s}^{(t+1)}-\pi_{j,s}^{(t)}\|_1\\
    &\le \frac{2}{(1-\gamma)^3}\left(\max_s \sum_{j=1}^n\|\pi_{j,s}^{(t+1)}-\pi_{j,s}^{(t)}\|_1\right) \cdot\sum_{i=1}^n\sum_sd^{(t)}(s)\sum_{a_i}\left|\pi_i^{(t+1)}(a_i|s)-\pi_i^{(t)}(a_i|s)\right|\\
    &\le \frac{2}{(1-\gamma)^3}\left(\max_s \sum_{j=1}^n\|\pi_{j,s}^{(t+1)}-\pi_{j,s}^{(t)}\|_1\right) \cdot\sum_sd^{(t)}(s)\sum_{j=1}^n\|\pi_{j,s}^{(t+1)}-\pi_{j,s}^{(t)}\|_1\\
    &\le \frac{2}{(1-\gamma)^3}\left(\max_s \sum_{j=1}^n\|\pi_{j,s}^{(t+1)}-\pi_{j,s}^{(t)}\|_1\right)^2
\end{align*}
From Cauchy-Schwarz inequality,
\begin{align*}
    \left(\sum_{j=1}^n\|\pi_{j,s}^{(t+1)}-\pi_{j,s}^{(t)}\|_1\right)^2 &= \left(\sum_{j=1}^n\sum_{a_j}\pi_i^{(t)}(a_j|s)\left|\Delta^{(t)}_j(s,a_j)\right|\right)^2\\&\le \left(\sum_{j=1}^n\sum_{a_j}\pi_i^{(t)}(a_j|s)\right)\left(\sum_{j=1}^n\sum_{a_j}\pi_j^{(t)}(a_j|s)\Delta^{(t)}_j(s,a_j)^2\right)\\
    &= n \sum_{j=1}^n\sum_{a_j}\pi_j^{(t)}(a_j|s)\Delta^{(t)}_j(s,a_j)^2
\end{align*}
Thus
\begin{align*}
    |\textup{Part C}| &\le \frac{2n}{(1-\gamma)^3}\sum_{i=1}^n\sum_{a_i}\pi_i^{(t)}(a_i|s)\Delta^{(t)}_i(s,a_i)^2 \\
    &\le \frac{2nM}{(1-\gamma)^3}\sum_{i=1}^n\sum_{a_i}d^{(t)}(s)\pi_i^{(t)}(a_i|s)\Delta^{(t)}_i(s,a_i)^2 
\end{align*}
Lastly, we will bound the absolute value of Part D.

\begin{align*}
    |\textup{Part D}|&=\left|\frac{1}{1-\gamma}\sum_{i=1}^n\sum_s \left(d_{\widetilde{\theta}^{i,(t)}}(s)-d^{(t)}(s)\right)\sum_{a_i}\left(\pi^{(t+1)}(a_i|s)-\pi_i^{(t)}(a_i|s)\right)\overline{Q_i^{\widetilde{\theta}^{i+1,(t)}}}(s,a_i)\right|\\
    &\le \frac{1}{(1-\gamma)^2}\sum_{i=1}^n\sum_s \left|d_{\widetilde{\theta}^{i,(t)}}(s)-d^{(t)}(s)\right|\sum_{a_i}\left|\pi^{(t+1)}(a_i|s)-\pi_i^{(t)}(a_i|s)\right|\\
    &\le \frac{1}{(1-\gamma)^2}\sum_{i=1}^n \max_s\|\pi_{i,s}^{(t+1)}-\pi_{i,s}^{(t)}\|_1\sum_s\left|d_{\widetilde{\theta}^{i,(t)}}(s)-d^{(t)}(s)\right|.
\end{align*}
From Corollary \ref{coro:NPG-log-barrier-auxiliary-2} 
\begin{align*}
    \sum_s\left|d_{\widetilde{\theta}^{i,(t)}}(s)-d^{(t)}(s)\right| &\le \frac{1}{1-\gamma}\max_s \left\|\pi_{\widetilde{\theta}^{i,(t)}}(a|s)-\pi^{(t)}(a|s)\right\|_1\\
    &\le \frac{1}{1-\gamma}\max_s\sum_{i=1}^n \|\pi_{i,s}^{(t+1)}-\pi_{i,s}^{(t)}\|_1.\\
\end{align*}
Thus
\begin{align*}
    |\textup{Part D}|&\le \frac{1}{(1-\gamma)^2}\sum_{i=1}^n \max_s\|\pi_{i,s}^{(t+1)}-\pi_{i,s}^{(t)}\|_1\sum_s\left|d_{\widetilde{\theta}^{i,(t)}}(s)-d^{(t)}(s)\right|\\
    &\le \frac{1}{(1-\gamma)^3}\left(\sum_{i=1}^n \max_s\|\pi_{i,s}^{(t+1)}-\pi_{i,s}^{(t)}\|_1\right)\left(\max_s\sum_{i=1}^n \|\pi_{i,s}^{(t+1)}-\pi_{i,s}^{(t)}\|_1\right)\\
    &\le  \frac{1}{(1-\gamma)^3}\left(\sum_{i=1}^n \max_s\|\pi_{i,s}^{(t+1)}-\pi_{i,s}^{(t)}\|_1\right)^2
\end{align*}
From Cauchy-Schwarz inequality
\begin{align*}
    \left(\sum_{i=1}^n \max_s\|\pi_{i,s}^{(t+1)}-\pi_{i,s}^{(t)}\|_1\right)^2 & \le n\sum_{i=1}^n \max_s\left( \sum_{a_i}\left|\pi_{i}^{(t+1)}(a_i|s)-\pi_{i}^{(t)}(a_i|s)\right|\right)^2\\
    &= n\sum_{i=1}^n \max_s \left( \sum_{a_i}\pi_{i}^{(t)}(a_i|s)\left|\Delta_i^{(t)}(a_i|s)\right|\right)^2\\
    &\le n\sum_{i=1}^n\max_s\left(\sum_{a_i}\pi_{i}^{(t)}(a_i|s)\right)\left(\sum_{a_i}\pi_{i}^{(t)}(a_i|s)\Delta_i^{(t)}(a_i|s)^2\right)\\
    &\le n\sum_{i=1}^n\max_s\sum_{a_i}\pi_{i}^{(t)}(a_i|s)\Delta_i^{(t)}(a_i|s)^2\\
    &\le n\sum_{i=1}^n\sum_{s,a_i}\pi_{i}^{(t)}(a_i|s)\Delta_i^{(t)}(a_i|s)^2.
\end{align*}
Thus
\begin{align*}
    |\textup{Part D}|& \le \frac{n}{(1-\gamma)^3}\sum_{i=1}^n\sum_{s,a_i}\pi_{i}^{(t)}(a_i|s)\Delta_i^{(t)}(a_i|s)^2\\
    &\le\frac{nM}{(1-\gamma)^3}\sum_{i=1}^n\sum_{s,a_i} d^{(t)}(s) \pi_{i}^{(t)}(a_i|s)\Delta_i^{(t)}(a_i|s)^2,\\
\end{align*}

Combining the bounds on Part A,B,C,D we get
\begin{align*}
    &\quad \widetilde{\Phi}(\theta^{(t+1)}) - \widetilde{\Phi}(\theta^{(t)}) = \textup{Part A} + \textup{Part B} + \textup{Part C} + \textup{Part D}\\
    &\ge \left(\frac{1}{2\eta}-4\lambda\max_{i}|\cA_i|M^2-\frac{4M}{(1-\gamma)^2}-\frac{3nM}{(1-\gamma)^3}\right)\sum_{i=1}^n\sum_{s,a_i}d^{(t)}(s)\pi_i^{(t)}(a_i|s)\Delta_i^{(t)}(s,a_i)^2,
\end{align*}
which completes the proof.
\end{proof}
\end{lemma}
\begin{lemma}\label{lemma:NPG-log-barrier-Delta-to-f}
Under the condition as in Lemma \ref{lemma:NPG-bounded-probability},
\begin{equation*}
    \sum_{i=1}^n\sum_{s,a_i}d^{(t)}(s)\pi_i^{(t)}(a_i|s)\Delta_i^{(t)}(s,a_i)^2\ge \frac{\eta^2}{9}\sum_{i=1}^n\sum_{s,a_i}d^{(t)}(s)\pi_i^{(t)}(a_i|s)f_i^{(t)}(s,a_i)^2
\end{equation*}
\begin{proof}
Recall from the definition of $\Delta_i^{(t)}(s,a_i)$:
\begin{equation*}
    \Delta_i^{(t)}(s,a_i)=\frac{\exp{\left(\eta f_i^{(t)}(s,a_i)\right)}}{Z_t^{i,s}} - 1.
\end{equation*}
Thus
\begin{align*}
    &\sum_{a_i}\pi_i^{(t)}(a_i|s)\Delta_i^{(t)}(s,a_i)^2 = \frac{1}{\left(Z_t^{i,s}\right)^2}\sum_{a_i}\pi_i^{(t)}(a_i|s)\left(\exp{\left(\eta f_i^{(t)}(s,a_i)\right)} - Z_t^{i,s}\right)^2\\
    &=\frac{1}{\left(Z_t^{i,s}\right)^2}\left[\sum_{a_i}\pi_i^{(t)}(a_i|s)\left(\exp{\left(\eta f_i^{(t)}(s,a_i)\right)}\!-\!1\right)^2\right.\\&\qquad-\left.2\sum_{a_i}\pi_i^{(t)}(a_i|s)\left(\exp{\left(\eta f_i^{(t)}(s,a_i)\right)}\!-\!1\right)\left( Z_t^{i,s}\!-\!1\right) \!+\! \sum_{a_i}\pi_i^{(t)}(a_i|s)\left( Z_t^{i,s}-1\right)^2\right]\\
    &=\frac{1}{\left(Z_t^{i,s}\right)^2}\left[\sum_{a_i}\pi_i^{(t)}(a_i|s)\left(\exp{\left(\eta f_i^{(t)}(s,a_i)\right)}\!-\!1\right)^2-\left( Z_t^{i,s}-1\right)^2\right]
\end{align*}
Since $|e^x - 1|\ge \frac{|x|}{2}$ for $x\ge -1$, we have that
\begin{align*}
    \sum_{a_i}\pi_i^{(t)}(a_i|s)\left(\exp{\left(\eta f_i^{(t)}(s,a_i)-1\right)}\right)^2 &\ge \sum_{a_i}\pi_i^{(t)}(a_i|s)\left(\frac{\eta}{2} f_i^{(t)}(s,a_i)\right)^2\\
    &\ge\frac{\eta^2}{4}\sum_{a_i}\pi_i^{(t)}(a_i|s)f_i^{(t)}(s,a_i)^2
\end{align*}
Additionally, as is proved in Lemma \ref{lemma:NPG-bounded-probability},
\begin{align*}
    Z_t^{i,s} &= \sum_{a_i}\pi_i^{(t)}(a_i|s)\exp{\left(\eta f^{(t)}(s,a_i)\right)}\\
    &\le \sum_{a_i}\pi_i^{(t)}(a_i|s) \left(1+\left(\eta f^{(t)}(s,a_i)\right)+\left(\eta f^{(t)}(s,a_i)\right)^2\right)\\
    &\le 1 + \eta^2\sum_{a_i}\pi_i^{(t)}(a_i|s)f_i^{(t)}(s,a_i)^2.
\end{align*}
Thus
\begin{align*}
    &\sum_{a_i}\pi_i^{(t)}(a_i|s)\Delta_i^{(t)}(s,a_i)^2=\frac{1}{\left(Z_t^{i,s}\right)^2}\left[\sum_{a_i}\pi_i^{(t)}(a_i|s)\left(\exp{\left(\eta f_i^{(t)}(s,a_i)\right)}\!-\!1\right)^2-\left( Z_t^{i,s}-1\right)^2\right]\\
    &\ge \frac{1}{\left(Z_t^{i,s}\right)^2}\left[\frac{\eta^2}{4}\sum_{a_i}\pi_i^{(t)}(a_i|s)f_i^{(t)}(s,a_i)^2-\left( \eta^2\sum_{a_i}\pi_i^{(t)}(a_i|s)f_i^{(t)}(s,a_i)^2 \right)^2\right]\\
    & = \frac{1}{\left(Z_t^{i,s}\right)^2} \left(\eta^2\sum_{a_i}\pi_i^{(t)}(a_i|s)f_i^{(t)}(s,a_i)^2\right)\left(\frac{1}{4}-\eta^2\sum_{a_i}\pi_i^{(t)}(a_i|s)f_i^{(t)}(s,a_i)^2\right).
\end{align*}
From Corollary \ref{coro:NPG-log-barrier-bounded-quantities}
\begin{align*}
    &-\frac{1}{15}\le \eta f_i^{(t)}(s,a_i)\le \frac{1}{3}\\
    \quad \Longrightarrow \quad &\eta^2 \sum_{a_i}\pi_i^{(t)}(s,a_i)f_i^{(t)}(s,a_i)^2 \le \frac{1}{9}.
\end{align*}
Thus
\begin{align*}
    \sum_{a_i}\!\pi_i^{(t)}\!(a_i|s)\Delta_i^{(t)}\!(s,a_i)^2  & \frac{1}{\left(\!Z_t^{i,s}\!\right)^{\!2}} \left(\eta^2\sum_{a_i}\pi_i^{(t)}\!(a_i|s)f_i^{(t)}\!(s,a_i)^2\right)\left(\frac{1}{4}\!-\!\eta^2\!\sum_{a_i}\!\pi_i^{(t)}\!(a_i|s)f_i^{(t)}\!(s,a_i)^2\right)\\
    &\ge  \frac{1}{\left(Z_t^{i,s}\right)^2} \left(\eta^2\sum_{a_i}\pi_i^{(t)}(a_i|s)f_i^{(t)}(s,a_i)^2\right)\left(\frac{1}{4}-\frac{1}{9}\right)\\
    &\ge \left(\frac{9}{10}\right)^2\left(\frac{1}{4}-\frac{1}{9}\right)\left(\eta^2\sum_{a_i}\pi_i^{(t)}(a_i|s)f_i^{(t)}(s,a_i)^2\right)\\
    &\ge \frac{\eta^2}{9}\sum_{a_i}\pi_i^{(t)}(a_i|s)f_i^{(t)}(s,a_i)^2,
\end{align*}
Thus
\begin{equation*}
    \sum_{i=1}^n\sum_{s,a_i}d^{(t)}(s)\pi_i^{(t)}(a_i|s)\Delta_i^{(t)}(s,a_i)^2\ge \frac{\eta^2}{9}\sum_{i=1}^n\sum_{s,a_i}d^{(t)}(s)\pi_i^{(t)}(a_i|s)f_i^{(t)}(s,a_i)^2
\end{equation*}
which completes the proof.
\end{proof}
\end{lemma}
\begin{lemma}\label{lemma:NPG-log-barrier-f-to-NEgap}
\begin{equation*}
    \NEgap(\theta^{(t)})\le\frac{\sum_{i=1}^n\sum_{s,a_i}d^{(t)}(s)\pi_i^{(t)}(a_i|s)f_i^{(t)}(s,a_i)^2}{4\lambda} + \lambda\max_i|\cA_i|M,
\end{equation*}
where $M = \sup_{\theta}\max_s\frac{1}{d_{\theta}(s)}$.
\begin{proof}

We will now prove the lemma. 
\begin{align*}
    &\quad d^{(t)}(s)\pi_i^{(t)}f_i^{(t)}(s,a_i)^2 = d^{(t)}(s) \pi_i^{(t)}(a_i|s)\left(\frac{1}{1-\gamma}\overline{A_i^{(t)}}(s,a_i) + \lambda \frac{1}{d^{(t)}(s)\pi_i^{(t)}(a_i|s)}-\frac{\lambda|\cA_i|}{d^{(t)}(s)}\right)^2\\
    &=\! d^{(t)}\!(s)\pi_i^{(t)}\!(a_i|s)\left(\frac{1}{1\!-\!\gamma}\overline{A_i^{(t)}}\!(s,a_i)\!-\!\frac{\lambda|\cA_i|}{d^{(t)}\!(s)}\right)^{\!2} \!\!+\! \frac{\lambda^2}{d^{(t)}\!(s)\pi_i^{(t)}\!(a_i|s)} \!+\! 2\lambda\left(\frac{1}{1\!-\!\gamma}\overline{A_i^{(t)}}\!(s,a_i)\!-\!\frac{\lambda|\cA_i|}{d^{(t)}\!(s)}\right)\\
    &\ge 4\lambda\left(\frac{1}{1-\gamma}\overline{A_i^{(t)}}(s,a_i)\!-\!\frac{\lambda|\cA_i|}{d^{(t)}(s)}\right).\\
& \Longrightarrow \quad \frac{1}{1-\gamma}\overline{A_i^{(t)}}(s,a_i) \le \frac{d^{(t)}(s)\pi_i^{(t)}(a_i|s)f_i^{(t)}(s,a_i)^2}{4\lambda} + \frac{\lambda|\cA_i|}{d^{(t)}(s)}\\
&\qquad\qquad\qquad\qquad\qquad\le \frac{\sum_i\sum_{s,a_i}d^{(t)}(s)\pi_i^{(t)}(a_i|s)f_i^{(t)}(s,a_i)^2}{4\lambda} + \lambda\max_i|\cA_i|M.
\end{align*}
Thus from Lemma \ref{lemma:NE-gap-bound},
\begin{align*}
    \NEgap_i(\theta^{(t)})\le \frac{1}{1-\gamma}\max_{s,a_i}\overline{A_i^{(t)}}(s,a_i)\le \frac{\sum_i\sum_{s,a_i}d^{(t)}(s)\pi_i^{(t)}(a_i|s)f_i^{(t)}(s,a_i)^2}{4\lambda} + \lambda\max_i|\cA_i|M,
\end{align*}
which completes the proof.
\end{proof}
\end{lemma}
We are now ready to prove Theorem \ref{thm:log-barrier-NPG}.

\begin{proof}[Proof of Theorem \ref{thm:log-barrier-NPG}]
From Lemma \ref{lemma:NPG-log-barrier-sufficient-ascent} we have that for\\
$\eta \le \min\left\{\frac{1}{15\left(\frac{1}{(1-\gamma)^2} + \lambda|\cA_i|M \right)},\frac{1}{4\left(4\lambda\max_i|\cA_i|M^2 + \frac{4M}{(1-\gamma)^2} + \frac{3nM}{(1-\gamma)^3}\right)}\right\}$,
\begin{equation*}
    \widetilde{\Phi}(\theta^{(t+1)})-\widetilde{\Phi}(\theta^{(t)})\ge\frac{1}{4\eta}\sum_{i=1}^n\sum_{s,a_i}d^{(t)}(s)\pi_i^{(t)}(a_i|s)\Delta_i^{(t)}(s,a_i)^2.
\end{equation*}
From Lemma \ref{lemma:NPG-log-barrier-Delta-to-f},
\begin{align*}
     \widetilde{\Phi}(\theta^{(t+1)})-\widetilde{\Phi}(\theta^{(t)})&\ge\frac{1}{4\eta}\sum_{i=1}^n\sum_{s,a_i}d^{(t)}(s)\pi_i^{(t)}(a_i|s)\Delta_i^{(t)}(s,a_i)^2\\
     &\ge \frac{\eta}{36}\sum_{i=1}^n\sum_{s,a_i}d^{(t)}(s)\pi_i^{(t)}(a_i|s)f_i^{(t)}(s,a_i)^2
\end{align*}
Thus by telescoping we have
\begin{align*}
    \frac{\sum_{t=0}^{T-1}\sum_{i=1}^n\sum_{s,a_i}d^{(t)}(s)\pi_i^{(t)}(a_i|s)f_i^{(t)}(s,a_i)^2}{T}\le \frac{36\left(\widetilde{\Phi}(\theta^{(T)})-\widetilde{\Phi}(\theta^{(0)})\right)}{\eta T}.
\end{align*}
From Lemma \ref{lemma:NPG-log-barrier-f-to-NEgap},
\begin{align*}
    \frac{\sum_{t=0}^{T-1}\NEgap(\theta^{(t)})}{T}&\le \frac{1}{4\lambda} \frac{\sum_{t=0}^{T-1}\sum_{i=1}^n\sum_{s,a_i}d^{(t)}(s)\pi_i^{(t)}(a_i|s)f_i^{(t)}(s,a_i)^2}{T} + \lambda\max_i|\cA_i|M\\
    &\le \frac{9\left(\widetilde{\Phi}(\theta^{(T)})-\widetilde{\Phi}(\theta^{(0)})\right)}{\eta\lambda T} + \lambda\max_i|\cA_i|M.
\end{align*}
Specifically, set $\lambda = \frac{\epsilon}{2\max_i|\cA_i|M}$ and $\theta^{(0)} = \mathbf{0}$, then for any
\begin{align*}
    T &\ge \frac{18\left(\widetilde{\Phi}(\theta^{(T)})-\widetilde{\Phi}(\theta^{(0)})\right)}{(1-\gamma)\eta\lambda\epsilon}=\frac{36\max_i|\cA_i|(\phi_{\max}-\phi_{\min})M}{(1-\gamma)\eta\epsilon^2}\\
    &\ge O\left(\frac{n\max_i|\cA_i|(\phi_{\max}-\phi_{\min})M^2}{(1-\gamma)^4\epsilon^2}\right),
\end{align*}
we have
\begin{align*}
    \frac{\sum_{t=0}^{T-1}\NEgap(\theta^{(t)})}{T}&\le \frac{\epsilon}{2}+\frac{\epsilon}{2} = \epsilon,
\end{align*}
which completes the proof.
\end{proof}

\section{Smoothness Proofs}
{This section mainly focuses on the smoothness of $\Phi$ and $\widetilde\Phi$. We first state the smoothness results in Lemma \ref{lemma:smoothness-unregularized} and Lemma \ref{lemma:smoothness-log-barrier}. The auxiliary lemmas used during proof of the above two lemmas are stated in Lemma \ref{lemma:smoothness-unregularized-auxillary} and Appendix \ref{apdx:useful-lemmas}.} 
\begin{lemma}[Smoothness of $\Phi(\theta)$]\label{lemma:smoothness-unregularized}
\begin{equation*}
    \left\|\nabla_\theta\Phi(\theta') -\nabla_\theta\Phi(\theta)\right\|_2 \le \frac{6n}{(1-\gamma)^3}\|\theta'-\theta\|_2
\end{equation*}
\begin{proof}
From Lemma \ref{lemma:smoothness-unregularized-auxillary} we have that
\begin{align*}
    \left\|\nabla_\theta\Phi(\theta') -\nabla_\theta\Phi(\theta)\right\|_2^2 &= \sum_{i=1}^n\left\|\nabla_{\theta_i}\Phi(\theta') -\nabla_{\theta_i}\Phi(\theta)\right\|_2^2\\
    &\le \sum_{i=1}^n\left\|\nabla_{\theta_i}\Phi(\theta') -\nabla_{\theta_i}\Phi(\theta)\right\|_1^2\\
    &\le \sum_{i=1}^n \left(\frac{6}{(1-\gamma)^3}\sum_{i=1}^n\|\theta_i'-\theta_i\|_2\right)^2\\
    &=\frac{36n}{(1-\gamma)^6}\left(\sum_{i=1}^n\|\theta_i'-\theta_i\|_2\right)^2\\
    &\le \frac{36n^2}{(1-\gamma)^6}\sum_{i=1}^n\|\theta_i'-\theta_i\|_2^2\\
    &= \frac{36n^2}{(1-\gamma)^6}\|\theta'-\theta\|_2^2,
\end{align*}
thus
\begin{equation*}
     \left\|\nabla_\theta\Phi(\theta') -\nabla_\theta\Phi(\theta)\right\|_2 \le \frac{6n}{(1-\gamma)^3}\|\theta'-\theta\|_2
\end{equation*}
\end{proof}
\end{lemma}

\begin{lemma}[Smoothness of $\widetilde\Phi(\theta)$]\label{lemma:smoothness-log-barrier}
\begin{equation*}
    \left\|\nabla_\theta\widetilde\Phi(\theta') -\nabla_\theta\widetilde\Phi(\theta)\right\|_2 \le \left(\frac{6n}{(1-\gamma)^3} + 2\lambda\max_i|\cA_i|\right)\|\theta'-\theta\|_2
\end{equation*}
\begin{proof}
Since
\begin{align*}
    \frac{\partial\sum_{i=1}^n\sum_{s,a_i}\log\pi_{\theta_i}(a_i|s)}{\partial \theta_{s,a_i}} = 1-|\cA_i|\pi_{\theta_i}(a_i|s)
\end{align*}
we have that
\begin{align*}
    &\left\|\nabla_\theta\! \left(\!\sum_{i=1}^n\sum_{s,a_i}\!\log\pi_{\theta_i'}(a_i|s)\! -\! \sum_{i=1}^n\sum_{s,a_i}\log\pi_{\theta_i}(a_i|s)\!\right)\right\|_2^2 \!\!=\! \sum_{i=1}^n|\cA_i|^2\sum_{s}\sum_{a_i}\left(\pi_{\theta_i'}(a_i|s)\!\! -\! \pi_{\theta_i}(a_i|s)\right)^2\\
    & \le\sum_{i=1}^n\!|\cA_i|^2\sum_{s} \|\pi_{\theta_{i,s}'}\!\!\!-\!\pi_{\theta_{i,s}}\|_1^2\stackrel{ \mbox{(Corollary \ref{coro:smoothness-auxillary-2})}}{\le} 4\sum_{i=1}^n|\cA_i|^2\sum_{s}\|\theta_{i,s}' \!\!\!-\! \theta_{i,s}\|_2^2\le 4\max_i|\cA_i|^2 \|\theta'\!\!-\!\theta\|_2^2
\end{align*}
Thus
\begin{align*}
    &\quad \left\|\nabla_\theta\widetilde\Phi(\theta') -\nabla_\theta\widetilde\Phi(\theta)\right\|_2 \\
    &\le  \left\|\nabla_\theta\Phi(\theta') -\nabla_\theta\Phi(\theta)\right\|_2 +  \lambda\left\|\nabla_\theta \left(\sum_{i=1}^n\sum_{s,a_i}\log\pi_{\theta_i'}(a_i|s) - \sum_{i=1}^n\sum_{s,a_i}\log\pi_{\theta_i}(a_i|s)\right)\right\|_2\\
    &\le \left(\frac{6n}{(1-\gamma)^3} + 2\lambda\max_i|\cA_i|\right)\|\theta'-\theta\|_2. \qedhere
\end{align*}
\end{proof}
\end{lemma}
\begin{lemma}\label{lemma:smoothness-unregularized-auxillary}
\begin{align*}
     \left\|\nabla_{\theta_i}\Phi(\theta') -\nabla_{\theta_i}\Phi(\theta)\right\|_1  \le \frac{6}{(1-\gamma)^3}\sum_{i=1}^n\|\theta_i'-\theta_i\|_2
\end{align*}
\begin{proof}
\vspace{-15pt}
\begin{align*}
     &\left\|\nabla_{\theta_i}\Phi(\theta') -\nabla_{\theta_i}\Phi(\theta)\right\|_1 
     =\frac{1}{1-\gamma}\sum_{s,a_i}\left|d_{\theta'}(s)\pi_{\theta_i'}(a_i|s)\overline{A_i^{\theta'}}(s,a_i) - d_{\theta}(s)\pi_{\theta_i}(a_i|s)\overline{A_i^{\theta}}(s,a_i)\right|\\
     &= \frac{1}{1\!-\!\gamma}\!\sum_{s,a_i}\!\left|d_{\theta'}\!(s)\pi_{\theta_i'}\!(a_i|s)\sum_{a_{-\!i}}\!\pi_{\theta_{\!-\!i}'}(a_{-\!i}|s)A_i^{\theta'}(s,a_i,a_{-\!i}) \!-\! d_{\theta}(s)\pi_{\theta_i}(a_i|s)\!\sum_{a_{-\!i}}\!\pi_{\theta_{\!-\!i}}\!(a_{-\!i}|s)A_i^{\theta}(s,a_i)\right|\\
     &\le \frac{1}{1-\gamma}\sum_{s,a}\left|d_{\theta'}(s)\pi_{\theta'}(a|s)A_i^{\theta'}(s,a) - d_{\theta}(s)\pi_{\theta}(a|s)A_i^{\theta}(s,a)\right|\\
     &\le \frac{1}{1\!-\!\gamma}\!\left(\!\sum_{s,a}\!\left|d_{\theta'}\!(s)\pi_{\theta'}\!(a|s) \!-\! d_{\theta}(s)\pi_{\theta}(a|s)\right|\left|A_i^{\theta'}\!(s,a)\right| \!+\! \sum_{s,a}d_{\theta}(s)\pi_{\theta}(a|s)\left|A_i^{\theta'}(s,a) \!-\! A_i^{\theta}(s,a)\right| \right)\\
     &\le \frac{1}{1-\gamma}\left(\frac{1}{1-\gamma}\sum_{s,a}\left|d_{\theta'}(s)\pi_{\theta'}(a|s) - d_{\theta}(s)\pi_{\theta}(a|s)\right| + \max_{s,a_i}\left|A_i^{\theta'}(s,a) - A_i^{\theta}(s,a)\right| \right)
\end{align*}
From Lemma \ref{lemma:NPG-log-barrier-auxiliary-1} and Corollary \ref{coro:smoothness-auxiliary}, we have that
\begin{align*}
    \left\|\nabla_{\theta_i}\Phi(\theta') -\nabla_{\theta_i}\Phi(\theta)\right\|_1  &\le \frac{3}{(1-\gamma)^3}\max_s\|\pi_{\theta_s'}-\pi_{\theta_s}\|_1\\
    &\le \frac{3}{(1-\gamma)^3}\max_s\sum\|\pi_{\theta_{i,s}'}-\pi_{\theta_{i,s}}\|_1
\end{align*}
From Corollary \ref{coro:smoothness-auxillary-2} we have that
\begin{align*}
    \left\|\nabla_{\theta_i}\Phi(\theta') -\nabla_{\theta_i}\Phi(\theta)\right\|_1  
    &\le \frac{3}{(1-\gamma)^3}\max_s\sum\|\pi_{\theta_{i,s}'}-\pi_{\theta_{i,s}}\|_1\\
    &\le \frac{6}{(1-\gamma)^3}\sum_{i=1}^{n}\|\theta_i-\theta_i\|_2. \qedhere
\end{align*}
\end{proof}
\end{lemma}

\section{Some Useful Lemmas}\label{apdx:useful-lemmas}
\begin{lemma}\label{lemma:NPG-log-barrier-auxiliary-1}
\begin{align*}
\vspace{-20pt}
    \left|Q^{\theta'}(s,a)-Q^\theta(s,a)\right| &\le \frac{1}{(1-\gamma)^2}\max_s \|\pi_{\theta'_s}-\pi_{\theta_s}\|_1\\
    \left|V^{\theta'}(s)-V^\theta(s)\right| &\le \frac{1}{(1-\gamma)^2}\max_s \|\pi_{\theta'_s}-\pi_{\theta_s}\|_1,
\end{align*}
and thus
\begin{equation*}
    \left|A^{\theta'}(s,a)-A^\theta(s,a)\right| \le \frac{2}{(1-\gamma)^2}\max_s \|\pi_{\theta'_s}-\pi_{\theta_s}\|_1
\end{equation*}
\begin{proof}
From performance difference lemma we have that
\begin{align*}
    &\quad \left|Q^{\theta'}(s,a) -Q^\theta(s,a)\right|\\
    &= \left|\sum_{t=1}^{+\infty} \gamma^t \sum_{s'}\textup{Pr}^{\theta'}(s(t) = s'|s(0)=s,a(0)=a) \sum_{a'}\left(\pi_{\theta'}(a'|s')-\pi_{\theta}(a'|s')\right)Q^{\theta}(s',a')\right|\\
    &\le \left|\sum_{t=1}^{+\infty} \gamma^t \sum_{s'}\textup{Pr}^{\theta'}(s(t) = s'|s(0)=s,a(0)=a) \sum_{a'}\left|\pi_{\theta'}(a'|s')-\pi_{\theta}(a'|s')\right|\left|Q^{\theta}(s',a')\right|\right|\\
    &\le \left|\sum_{t=1}^{+\infty} \gamma^t \max_{s'} \sum_{a'}\left|\pi_{\theta'}(a'|s')-\pi_{\theta}(a'|s')\right|\frac{1}{1-\gamma}\right|\\
    & = \frac{1}{(1-\gamma)^2}\max_s \sum_{a}\left|\pi_{\theta'}(a|s)-\pi_{\theta}(a|s)\right|\\
    &=\frac{1}{(1-\gamma)^2}\max_s \|\pi_{\theta'_s}-\pi_{\theta_s}\|_1
\end{align*}
Same argument also holds for $\left|V^{\theta'}(s)-V^\theta(s)\right|$, and thus
\begin{equation*}
    \left|A^{\theta'}(s,a)-A^\theta(s,a)\right| \le  \left|Q^{\theta'}(s,a)-Q^\theta(s,a)\right| + \left|V^{\theta'}(s)-V^\theta(s)\right|\le \frac{2}{(1-\gamma)^2}\max_s \|\pi_{\theta'_s}-\pi_{\theta_s}\|_1. \qedhere
\end{equation*}
\end{proof}
\end{lemma}
\begin{lemma}\label{lemma:auxillary-d-pi-r}
\begin{equation*}
    \frac{1}{1-\gamma} \sum_{s,a}\left(d_{\theta'}(s)\pi_{\theta'}(a|s) - d_{\theta}(s)\pi_{\theta}(a|s)\right)r(s,a) \le \frac{1}{(1-\gamma)^2}\|r\|_\infty\max_s\|\pi_{\theta_s'}-\pi_{\theta_s}\|_1,
\end{equation*}
where $\|r\|_\infty = \max_{s,a}|r(s,a)|$.
\begin{proof}
For any reward function $r(s,a)$, we can define its value function $V^{\theta}(s)$ and $Q^{\theta}(s,a)$ correspondingly. Using performance difference lemma we have that
\begin{align*}
     &\quad \frac{1}{1-\gamma} \sum_{s,a}\left(d_{\theta'}(s)\pi_{\theta'}(a|s) - d_{\theta}(s)\pi_{\theta}(a|s)\right)r(s,a)\\
     &= \sum_{s}\rho(s)(V^{\theta'}(s)-V^\theta(s)) \\
    &=\frac{1}{1-\gamma}\sum_s d_{\theta'}(s)\sum_a\left(\pi_{\theta'}(a|s)-\pi_\theta(a|s)\right)Q^{\theta}(s,a)\\
    &\le \frac{1}{(1-\gamma)^2}\|r\|_\infty\sum_sd_{\theta'}(s)\|\pi_{\theta_s'}-\pi_{\theta_s}\|_1\\
    &\le \frac{1}{(1-\gamma)^2}\|r\|_\infty\max_s\|\pi_{\theta_s'}-\pi_{\theta_s}\|_1. \qedhere
\end{align*}
\end{proof}
\end{lemma}
We have the following two corollaries for Lemma \ref{lemma:auxillary-d-pi-r}.
\begin{corollary}\label{coro:NPG-log-barrier-auxiliary-2}
\begin{equation*}
    \frac{1}{1-\gamma} \sum_{s}\left|d_{\theta'}(s)-d_\theta(s)\right| \le \frac{1}{(1-\gamma)^2}\max_s\|\pi_{\theta_s'}-\pi_{\theta_s}\|_1
\end{equation*}
\begin{proof}
\begin{align*}
     \frac{1}{1-\gamma} \sum_{s}\left|d_{\theta'}(s)-d_\theta(s)\right| &= \max_{-1\le r(s)\le 1}\frac{1}{1-\gamma}\sum_{s,a}(d_{\theta'}(s)\pi_{\theta'}(a|s) - d_{\theta}(s)\pi_{\theta}(a|s))r(s) \\
    &\le \frac{1}{(1-\gamma)^2}\max_s\|\pi_{\theta_s'}-\pi_{\theta_s}\|_1. \qedhere
\end{align*}
\end{proof}
\end{corollary}
\begin{corollary}\label{coro:smoothness-auxiliary}
\begin{equation*}
    \frac{1}{1-\gamma} \sum_{s}\left|d_{\theta'}(s)\pi_{\theta'}(a|s) - d_{\theta}(s)\pi_{\theta}(a|s)\right| \le \frac{1}{(1-\gamma)^2}\max_s\|\pi_{\theta_s'}-\pi_{\theta_s}\|_1.
\end{equation*}
\begin{proof}
\begin{align*}
     &\quad \frac{1}{1-\gamma} \sum_{s}\left|d_{\theta'}(s)\pi_{\theta'}(a|s) - d_{\theta}(s)\pi_{\theta}(a|s)\right| \\
     &= \max_{-1\le r(s,a)\le 1}\frac{1}{1-\gamma}\sum_{s,a}(d_{\theta'}(s)\pi_{\theta'}(a|s) - d_{\theta}(s)\pi_{\theta}(a|s))r(s,a) \\
    &\le \frac{1}{(1-\gamma)^2}\max_s\|\pi_{\theta_s'}-\pi_{\theta_s}\|_1. \qedhere
\end{align*}
\end{proof}
\end{corollary}

\begin{lemma}\label{lemma:smoothness-auxillary-2}
\begin{equation*}
   \sum_{a_i}\left(\pi_{\theta_i'}(a_i|s) - \pi_{\theta_i}(a_i|s)\right)f(a_i)\le 2\|f\|_\infty\|\theta_{i,s}'-\theta_{i,s}\|_2
\end{equation*}
\begin{proof}
It suffices to show that
\begin{equation*}
    \left\|\nabla_{\theta_{i,s}}\sum_{a_i}\pi_{\theta_i}(a_i|s)f(a_i)\right\|_2\le 2\|f\|_\infty,~\forall \theta,
\end{equation*}
then by Lagrange mean value theorem,
\begin{align*}
    &\quad \sum_{a_i}\left(\pi_{\theta_i'}(a_i|s) - \pi_{\theta_i}(a_i|s)\right)f(a_i)\\
    &\le \max_{t, \bar \theta = t\theta + (1-t)\theta'} \left\|\nabla_{\theta_{i,s}}\sum_{a_i}\pi_{\bar \theta_i}(a_i|s)f(a_i)\right\|_2\|\theta_{i,s}'-\theta_{i,s}\|_2\le 2\|f\|_\infty\|\theta_{i,s}'-\theta_{i,s}\|_2.
\end{align*}
Since 
\begin{equation*}
    \frac{\partial \sum_{a_i}\pi_{ \theta_i}(a_i|s)f(a_i)}{\partial \theta_{a_i,s}} = \pi_{\theta_i}(a_i|s)(f(a_i) - \bar f), ~\textup{where } \bar f = \sum_{a_i} \pi_{\theta_i}(a_i|s)f(a_i),
\end{equation*}
we have
\begin{align*}
    \left\|\nabla_{\theta_{i,s}}\sum_{a_i}\pi_{\theta_i}(a_i|s)f(a_i)\right\|_2^2 = \sum_{a_i}\pi_{\theta_i}(a_i|s)^2(f(a_i) - \bar f)^2\le \sum_{a_i}\pi_{\theta_i}(a_i|s)^2(2\|f\|_\infty)^2\le 4\|f\|_\infty^2,
\end{align*}
which completes the proof.
\end{proof}
\end{lemma}

\begin{corollary}\label{coro:smoothness-auxillary-2}(of Lemma \ref{lemma:smoothness-auxillary-2})
\begin{equation*}
    \|\pi_{\theta_{i,s}'}-\pi_{\theta_{i,s}}\|_1 \le2 \|\theta_{i,s}' - \theta_{i,s}\|_2\le 2\|\theta_i'-\theta_i\|_2
\end{equation*}
\begin{proof}
\begin{equation*}
    \|\pi_{\theta_{i,s}'}-\pi_{\theta_{i,s}}\|_1 = \max_{f: \|f\|_\infty\le 1} \sum_{a_i}\left(\pi_{\theta_i'}(a_i|s) - \pi_{\theta_i}(a_i|s)\right)f(a_i)\le 2\|\theta_{i,s}'-\theta_{i,s}\|_2. \qedhere
\end{equation*}
\end{proof}
\end{corollary}

\end{document}